\documentclass[11pt]{article}
\usepackage{amssymb,amsmath,amsfonts,mathrsfs,amsthm,epsfig,latexsym,color}
\usepackage{enumitem,geometry}
\usepackage{fourier}
\usepackage[all]{xy}
\makeatletter
\newcommand{\subsectionruninhead}{\@startsection{subsection}{2}{0mm}
{-\baselineskip}{-0mm}{\bf\large}}
\newcommand{\subsubsectionruninhead}{\@startsection{subsubsection}{3}{0mm}
{-\baselineskip}{-0mm}{\bf\normalsize}}
\makeatother

\newtheorem*{theorem*}{Theorem}
\newtheorem*{proof*}{Proof}
\newtheorem*{proposition*}{Proposition}
\newtheorem*{notation*}{Notation}
\newtheorem*{corollary*}{Corollary}
\newtheorem*{claim*}{Claim}
\newtheorem*{remark*}{Remark}

\newtheorem{theorem}{Theorem}[section]
\newtheorem{proposition}{Proposition}[section]
\newtheorem{corollary}[proposition]{Corollary}
\newtheorem{lemma}[proposition]{Lemma}
\newtheorem{claim}[proposition]{Claim}

\theoremstyle{definition}

\theoremstyle{remark}
\newtheorem{remark}[proposition]{Remark}

\numberwithin{equation}{section}

\def\NN{\mathbb{N}}
\def\CC{\mathbb{C}}
\def\RR{\mathbb{R}}
\def\QQ{\mathbb{Q}}

\def\TT{\mathbb{T}}
\def\ZZ{\mathbb{Z}}
\def\cF{\mathcal{F}}
\def\cC{\mathcal{C}}
\def\cL{\mathcal{L}}

\def\cD{\mathcal{D}}
\def\cR{\mathcal{R}}
\def\cU{\mathcal{U}}
\def\H{{\rm Hol}}

\def\eps{\varepsilon}

\def\e{{\varepsilon}}

\begin{document}
	\title{Joint integrability and spectral rigidity for Anosov diffeomorphisms}
	\author{Andrey Gogolev\footnote{The first author was partially supported by NSF grant DMS-1955564} ~ and ~ Yi Shi\footnote{The second author was partially supported by National Key R\&D Program of China (2021YFA1001900) and NSFC (12071007, 11831001, 12090015).}}
	
	\date{}

	\maketitle

	\begin{abstract}
		Let $f\colon\TT^d\to\TT^d$ be an Anosov diffeomorphism whose linearization $A\in{\rm GL}(d,\ZZ)$ is irreducible. Assume that $f$ is also absolutely partially hyperbolic where a weak stable subbundle is considered as the center subbundle. We show that if the strong stable and  unstable subbundles are jointly integrable, then $f$ is dynamically coherent and all foliations match corresponding linear foliation under the conjugacy to the linearization $A$. Moreover, $f$ admits the finest dominated splitting in weak stable subbundle with dimensions matching those for $A$, and it has spectral rigidity along all these subbundles.
		
		In dimension 4 we are also able to obtain a similar result which allows to group the weak stable and unstable subbundles into a center subbundle and assumes joint integrability of strong stable and unstable subbundles. As an application, we show that for every symplectic diffeomorphism $f\in{\rm Diff}^2_{\omega}(\TT^4)$ which is $C^1$-close to an irreducible non-conformal automorphism $A\in{\rm Sp}(4,\ZZ)$, the extremal subbundles of $f$ are jointly integrable if and only if $f$ is smoothly conjugate to $A$. 
	\end{abstract}
	
	\section{Introduction}\label{sec:introduction}
	
	We explore the interplay between joint integrability of extremal subbundles of an Anosov diffeomorphism and its rigidity properties. Joint integrability is a special property which, more generally, appears in partially hyperbolic dynamics. If $E^s\oplus E^c\oplus E^u$ is a partially hyperbolic splitting of a diffeomorphism $f$ then if $E^s\oplus E^u$ is integrable to a foliation then we say that $f$ has the {\it joint integrability property}. { The joint integrability property is the complete opposite of the accessibility property of partially hyperbolic diffeomorphisms. When dim$E^c=1$, it has been proved \cite{D,HHU} that accessible systems are $C^1$-open and  $C^\infty$-dense in the space of partially hyperbolic diffeomorphisms.} So one does not expect that joint integrability occurs outside some special situations. However, in general, it could be hard to make this expectation precise.
	
	Probably the first appearance of joint integrability was in the work of Plante~\cite{Pl} who proved that joint integrability of stable and unstable foliations of an Anosov flow implies that the flow is a suspension an Anosov diffeomorphism.\footnote{It is still an open problem to decide whether the leaves of the $su$-foliation have to be compact.} In pursuit of stable ergodicity, F. Rodriguez Hertz showed~\cite{FRH} that for small perturbations of certain partially hyperbolic toral automorphisms with two dimensional center joint integrability only happens when the perturbation is conjugate to the automorphism via a conjugacy which is smooth along the center foliation. Gan and Shi~\cite{GS}, in pursuit of ergodicity on $\TT^3$, showed that in an Anosov homotopy class joint integrability also implies existence of a conjugacy which is smooth along the center (weak stable) foliation. Hammerlindl \cite{H} studied what he called $AB$-systems, i.e., partially hyperbolic diffeomorphisms which are leaf conjugate to fiberwise Anosov diffeomorphisms on nilmanifolds over a circle. For $C^2$ conservative $AB$-systems, Hammerlindl showed that joint integrability implies that the $AB$-system is topologically conjugate to a fiberwise Anosov diffeomorphism. Under some further assumptions, the conjugacy is actually smooth along the 1-dimensional center~\cite[Proposition 18]{DWX}.
	%{\color{red} Hammerlindl studied what he called $AB$-systems, i.e. partially hyperbolic diffeomorphisms which leaf conjugate to fiberwise Anosov diffeomorphisms on nilmanifolds over a circle, including systems leaf conjugate to partially hyperbolic toral automorphisms with one dimensional circle center and suspension Anosov flows with tori sections. For $C^2$ conservative $AB$-systems, Hammerlindl showed that joint integrability implies the system topologically conjugates Anosov diffeomorphisms over a rotation on circle. Thus every leaf of $su$-foliation are compact. Later, in their study of centralizer rigidity, Damjanovi{\'c}, Wilkinson and Xu showed~\cite[Proposition 18]{DWX} that for small ergodic perturbations of certain partially hyperbolic toral automorphisms with circle center joint integrability only happens when the perturbation is conjugate to a product with a circle rotation via a conjugacy which is $C^1$-smooth along the center foliation. }

	In this paper we obtain strong rigidity properties from joint integrability in the setting of Anosov diffeomorphisms on tori. More specifically, joint integrability of the strong stable and full unstable subbundles imply existence of fine dominated splitting along the weak stable subbundle as well as Lyapunov exponents rigidity. This builds an equivalence bridge between the geometric rigidity (joint integrability) and dynamical spectral rigidity (Lyapunov exponents rigidity) for Anosov diffeomorphisms on tori.
	
	Indeed, our results provide a new point of view on the rigidity problem for Anosov automorphisms. Normally, to obtain rigidity, various authors assume vanishing of the Lyapunov exponent obstructions at periodic points or for the volume~\cite{dlL, Go08, KS, SY}, joint integrability assumption has a very different nature. Interestingly, most of our result are global, that is, we do not require the Anosov diffeomorphism to be close to the linear model. Joint integrability also appears in other problems in hyperbolic dynamics. For example, recently joint integrability naturally appeared as a tool in higher dimensional smooth classification problem for Anosov diffeomorphisms and flows~\cite{GRH}.

	Yet another source of rigidity is measure theoretic absolute continuity property of the center foliation. In the context of volume preserving partially hyperbolic diffeomorphisms with one dimensional center, Avila, Viana and Wilkinson obtained strong rigidity results from absolute continuity of the center foliation~\cite{AVW1, AVW2}. {  Also, for 3-dimensional Anosov diffeomorphisms with a partially hyperbolic splitting, absolute continuity of the weak (center) foliation implies Lyapunov exponents rigidity along the strong foliation~\cite{Go12}.}
	
	We proceed with precise description of our setup and statements of our results.
	
	\subsection{The setting and global leaf conjugacy}
	Let $f\colon\TT^d\to\TT^d$ be an Anosov diffeomorphism. Denote by $A=f_*:H_1(\TT^d;\ZZ)\to H_1(\TT^d;\ZZ)$ the linear part of $f$, then $A\in{\rm GL}(d,\ZZ)$ gives a hyperbolic automorphism and $f$ is topologically conjugate to $A$  by work of Franks and Manning~\cite{Fr, Mn}. The conjugacy is a homeomorphism $h\colon \TT^d\to\TT^d$  which sends the orbits of $f$ to the orbits of $A$: $h\circ f=A\circ h$.

	Let $T\TT^d=E^s\oplus E^u$ be the hyperbolic splitting for $f$. Assume that $f$ is also {\it absolutely partially hyperbolic} with respect to the splitting:
	$$
	T\TT^d=E^{ss}\oplus E^{ws}\oplus E^u,
	$$
	where $E^s=E^{ss}\oplus E^{ws}$. This means that there exist $0<\mu<1$ and $k\in\NN$, such that
	$$
	\|Df^k|_{E^{ss}(x)}\|<\mu<m\big(Df^k|_{E^{ws}(x)}\big)\leq
	\|Df^k|_{E^{ws}(x)}\|<1<m\big(Df^k|_{E^u(x)}\big),
	\qquad\forall x\in\TT^d.
	$$
	Here $m(\cdot)$ stands for the conorm of linear operators, defined by $m(L)=\|L^{-1}\|^{-1}$.  For the adapted Riemannian metric the above inequalities hold with $k=1$. Hence, from now on we can assume that $k=1$ in the above inequalities, that is
	\begin{equation}
	\label{eq_mu}
	\|Df|_{E^{ss}(x)}\|<\mu<m\big(Df|_{E^{ws}(x)}\big)\leq
	\|Df|_{E^{ws}(x)}\|<1<m\big(Df|_{E^u(x)}\big),
	\qquad\forall x\in\TT^d.
	\tag{$\ast$}
	\end{equation}
	
%	We are interested in dynamical consequences of joint integrability of invariant subbundles which do not typically integrate together, such as extremal stable and unstable subbundles. It is natural to expect that such joint integrability imposes some rigidity on dynamics. We will confirm these expectations with several results. The ``jointly integrable case'' does appear in various problems in hyperbolic dynamics. For example the proof of stable ergodicity by F. Rodriguez Hertz~\cite{FRH} of partially hyperbolic automorphisms arrives at joint integrability property in the case when the partially hyperbolic diffeomorphism is not accessible. Recently, such joint integrability also naturally appeared in higher dimensional smooth classification problem for Anosov diffeomorphisms~\cite{GRH}.

	The following result lays down foundation for rigidity.
	
	\begin{theorem}\label{thm:conjugate}
		Let { $f\in{\rm Diff}^1(\TT^d)$} be an Anosov diffeomorphism with absolutely partially hyperbolic splitting $T\TT^d=E^{ss}\oplus E^{ws}\oplus E^u$ 
		{as above.}
		%with constant $0<\mu<1$ satisfying
		%$$
		%\|Df|_{E^{ss}(x)}\|<\mu<m(Df|_{E^{ws}(x)})\leq
		%\|Df|_{E^{ws}(x)}\|<1<\|Df|_{E^u(x)}\|.
		%$$
		If $E^{ss}\oplus E^u$ is integrable, then
		\begin{enumerate}
			\item The linear part $A\in{\rm GL}(d,\ZZ)$ of $f$ admits a partially hyperbolic splitting with matching dimensions:
			$$
			T\TT^d=L^{ss}\oplus L^{ws}\oplus L^u,
			\qquad {\it and} \qquad
			{\rm dim}L^{\sigma}={\rm dim}E^{\sigma}, \quad\sigma=ss,ws,u.
			$$
			\item Any eigenvalues $\mu^{ss}$ and $\mu^{ws}$ of  $A|_{L^{ss}}$ and  $A|_{L^{ws}}$, respectively, satisfy
			$$
			0<|\mu^{ss}|<\mu<|\mu^{ws}|<1.
			$$
			where $\mu$ is the constant from the definition~$(\ast)$ of partial hyperbolicity of $f$.
			\item Diffeomorphism $f$ is dynamically coherent and the invariant foliations for $f$ match the corresponding invariant foliations for $A$ under the conjugacy:
			$$
			h\big(\cF^{\sigma}\big)=\cL^{\sigma},
			\qquad 
			\sigma=ss,ws,u.
			$$
			Here $\cF^{ss}$,$\cF^{ws}$,$\cF^u$ are $f$-invariant foliations tangent to $E^{ss}$, $E^{ws}$, $E^u$, respectively; and $\cL^{ss}$, $\cL^{ws}$, $\cL^u$ are $A$-invariant foliations tangent to $L^{ss}$, $L^{ws}$, $L^u$, respectively.
		\end{enumerate}
	\end{theorem}

	\begin{remark} We recall that in our context {\it dynamically coherent} property of $f$ means that the distribution $E^{ws}$ integrates to an $f$-invariant foliation $\cF^{ws}$. One could expect that Anosov diffeomorphisms $f\colon\TT^d\to\TT^d$ which are also absolutely partially hyperbolic are always dynamically coherent (without assuming joint integrability). This is only known in dimension 3~\cite{BBI}. We also note that it is not hard to see that joint integrability is a necessary condition for the matching property $h(\cF^{ss})=\cL^{ss}$ to hold.
	\end{remark}
	
	\subsection{The results of spectral rigidity} 	
	
	We begin by stating the results in the case when one of extremal subbundles is the full unstable subbundle. 
	We say that $f$ is {\it irreducible} if its linear part $A\in{\rm GL}(d,\ZZ)$ is irreducible, i.e., the characteristic polynomial of $A$ is irreducible over $\ZZ$. In particular, if $A\in{\rm GL}(d,\ZZ)$ is irreducible, then every eigenspace of $A$ is dense after projecting on the torus $\TT^d$.
	
	For every $A\in{\rm GL}(d,\ZZ)$, an $A$-invariant splitting
	$$
	T\TT^d=L_1\oplus L_2\oplus \cdots\oplus L_m,
	$$
	is the {\it finest dominated splitting} for $A$ if it satisfies
	\begin{itemize}
		\item For every $i=1,\cdots,m$, if both $\mu_i$ and $\mu_i'$ are eigenvalues of $A|_{L_i}$, then $|\mu_i|=|\mu_i'|$.
		\item For every $i=1,\cdots,m-1$, if  $\mu_i$ is an eigenvalue of $A|_{L_i}$ and $\mu_{i+1}$ is an eigenvalue $A|_{L_{i+1}}$,  then $|\mu_i|<|\mu_{i+1}|$.
	\end{itemize}	
	If $L\subset T\TT^d$ is an $A$-invariant subbundle, then we can also consider the finest dominated splitting of $A$ restricted to $L$. 
	When $A$ is as in the Theorem~\ref{thm:conjugate}, we will denote the finest dominated splitting for $A|_{L^{ws}}$ in the following way
	$$
	L^{ws}=L^{ws}_1\oplus\cdots\oplus L^{ws}_k
	$$ 
	It is well-known that for sufficiently small perturbations of $A$ this splitting survives. It turns out this is also the case for large pertubations which have the joint integrability property. And we obtain Lyapunov exponents rigidity along this finest dominated splitting.
	%From now on, for every $f\in{\rm Diff}^{1+}(\TT^d)$ satisfies the assumption in Theorem \ref{thm:conjugate} and $A=f_*\in{\rm GL}(d,\ZZ)$ is the linear part of $f$,
	%we denote 

	%the finest dominated splitting of $A$ in $L^{ws}$.

	\begin{theorem}\label{thm:main}
		Let $f\in{\rm Diff}^2(\TT^d)$ be an irreducible Anosov diffeomorphism with absolutely partially hyperbolic splitting 
		$T\TT^d=E^{ss}\oplus E^{ws}\oplus E^u$.
		If $E^{ss}\oplus E^u$ is integrable and $f$ satisfies the center bunching condition:
		$$
		\|Df|_{E^{ws}(x)}\|~<~
		m\big(Df|_{E^{ws}(x)}\big)\cdot m\big(Df|_{E^u(x)}\big),
		\qquad \forall x\in\TT^d,
		$$
		then 
		\begin{itemize}
			\item $f$ has a dominated splitting on $E^{ws}$ with the same dimensions as the finest dominated splitting for $A|_{L^{ws}}$:
			$$
			E^{ws}=E^{ws}_1\oplus\cdots\oplus E^{ws}_k, 
			\qquad {\it  } \qquad
			{\dim}E^{ws}_i={\dim}L^{ws}_i, \quad\forall i=1,\cdots,k.
			$$
			\item $f$ is spectrally rigid along $E^{ws}_i$ for every $i=1,\cdots,k$. That is, the Lyapunov exponents of $f$ along $E^{ws}_i$ for all ergodic measures are all equal and  satisfy
			$$
			\lambda(E^{ws}_i, f)=\lambda(L^{ws}_i,A),
			\qquad \forall i=1,\cdots,k.
			$$
		\end{itemize}	
	\end{theorem}
	
	%{\color{cyan} theorem remark or addendum?
	%	\marginpar{\color{red}S: This theorem is not needed any more?}
	%\begin{theorem}\label{thm:main}
	%	Let $f\in{\rm Diff}^{1+\alpha}(\TT^d)$ be an irreducible Anosov diffeomorphism with a uniform dominated splitting
	%	$T\TT^d=E^{sss}\oplus\oplus E^{ss}\oplus E^{ws}\oplus E^{wws} \oplus E^u$.
	%	If $E^{ss}\oplus E^{wws}\oplus E^u$ is integrable, then $f$ has the finest dominated splitting on $E^{ws}$ with the same dimensions as the finest dominated splitting for $A|_{L^{ws}}$:
	%	$$
	%	E^{cs}=E^{ws}_1\oplus\cdots\oplus E^{ws}_k, 
	%	\qquad {\it with } \qquad
	%	{\dim}E^{ws}_i={\dim}L^{ws}_i, \quad\forall i=1,\cdots,k.
	%	$$
	%	Moreover, $f$ has spectrum rigidity on $E^{ws}_i$ for every $i=1,\cdots,k$. That is, the Lyapunov exponents of $f$ in $E^{ws}_i$ for all ergodic measures are equal and  satisfy
	%	$$
	%	\lambda(E^{ws}_i, f)~\equiv~\lambda(L^{ws}_i,A),     ???\lambda(E^{ws}_i, f)=\lambda(L^{ws}_i,A),?????
	%	\qquad \forall i=1,\cdots,k.
	%	$$
	%\end{theorem}
	%}
	\begin{remark}
		Note that the second item in the theorem actually implies that the splitting $E^{ws}=E^{ws}_1\oplus\cdots\oplus E^{ws}_k$ is, in fact, the finest dominated splitting for $f$ along $E^{ws}$.
	\end{remark}
	\begin{remark}\label{rk:ss-preserve}
		In the course of the proof of Theorem \ref{thm:main}, we will show that $E^{ss}\oplus E^u$ being integrable is equivalent to the conjugacy $h$ taking the strong stable foliation $\cF^{ss}$ of $f$ to the strong stable foliation $\cL^{ss}$ of $A$. 
		{ Moreover, if dim$L^{ws}_i=1$ or dim$L^{ws}_i=2$ for every $i=1,\cdots,k$, then the arguments of~\cite{GKS} show that the spectral rigidity of $f$ along $E^{ws}$ is equivalent to the conjugacy $h$ is  $C^{1+{\rm H\"older}}$-smooth along $E^{ws}$.}
	\end{remark}
	\begin{remark}
		It is clear that in the case when $\dim E^{ws}=1$ center bunching condition is automatically satisfied and one still has spectral rigidity along $E^{ws}$.
	\end{remark}
	
	We deduce the following global rigidity for codimension-one Anosov diffeomorphisms.
	%{\color{cyan}
	%\begin{corollary}
	%	Let $f_1, f_2\in{\rm Diff}^{1+\alpha}(\TT^d)$, $\alpha>0$, be a codimension-one Anosov diffeomorphisms with linear part $A$ which has real spectrum. If $f_1$ and $f_2$
	%	are absolutely partially hyperbolic 
	%	$$
	%	T\TT^d=E^{ss}\oplus E^{ws}\oplus E^u
	%	\qquad {\it with} \qquad
	%	{\rm dim}E^{ss}={\rm dim}E^u=1,
	%	$$ 
	%Assume that
	%		 $E^{ss}_{f_i}\oplus E^u_{f_i}$ is integrable, $i=1,2$ and that $f_1$ and $f_2$ have matching periodic data along $E^{ss}$ and $E^u$.
	%Then $f_1$ is $C^1$ conjugate to $f_2$.
	%\end{corollary}
	
	%In the case when $f_2=L$ we have following.
	%}
	
	\begin{corollary}\label{cor:codim1}
		Let $f\in{\rm Diff}^2(\TT^d)$ be a codimension-one Anosov diffeomorphism whose linear part $A$ has real spectrum. If $f$
		is absolutely partially hyperbolic 
		$$
		T\TT^d=E^{ss}\oplus E^{ws}\oplus E^u
		\qquad {\it with} \qquad
		{\rm dim}E^{ss}={\rm dim}E^u=1,
		$$ 
		then the followings are equivalent:
		\begin{itemize}
			\item $E^{ss}\oplus E^u$ is integrable and $f$ and has the same periodic data as its linear part $A$ along $E^{ss}$ and along $E^u$.
			\item $f$ is $C^{1+{\rm H\"older}}$ conjugate to $A$.
		\end{itemize} 
	\end{corollary}

	\begin{remark}
		The center bunching condition in Theorem \ref{thm:main} is automatically satisfied for codimension one Anosov diffeomorphisms in Corollary \ref{cor:codim1} by the assumption in the periodic data along $E^{ss}$ and $E^u$.
		If $f$ is conservative, then the periodic data condition along $E^{ss}$ and $E^u$ can be replaced with the metric entropy condition $h_m(f)=h_m(A)$, see \cite{SY}.
	\end{remark}
	
	\begin{remark}
		{We point out that this is a global rigidity result --- we do not assume that $f$ is $C^1$-close to $A$. This is rare in higher dimensions.}	
		It is also interesting to compare the above corollary to the periodic data rigidity result for irreducible automorphisms with real spectrum~\cite[Theorem A]{Go17}. There smoothness of the conjugacy is being established along all the one-dimensional foliations inductively from weakest to strongest. Consequenctly joint integrability of $E^{ss}\oplus E^u$ can be concluded at the very last step. In a sense, Corollary~\ref{cor:codim1} is a converse where this last step allows us to recover all the previous steps, that is, smoothness of the conjugacy along $E^{ws}$.
	\end{remark}
	
	%{\color{cyan}
		
	%	Further, in the first case one can also bring $f$ to a certain ``standard form.''
	%	\begin{corollary}
	%		Let $f\in{\rm Diff}^{1+\alpha}(\TT^d)$ be a codimension-one Anosov diffeomorphism whose linear part $A$ has real spectrum. If $f$
	%		is absolutely partially hyperbolic and  $E^{ss}\oplus E^u$ is integrable then $f$ is $C^1$ conjugate to a diffeomorphism of the form $g\circ A$, where $g$ is a diffeomorphism which preserves individual leaves of $\cL^{ss}\oplus \cL^u$.
	%	\end{corollary}
		
	%}

	%\vskip5mm
	
	We also get a local rigidity theorem without center bunching assumption in Theorem \ref{thm:main}.
	We say $A\in{\rm GL}(d,\ZZ)$ is \emph{generic} if it satisfies
	\begin{enumerate}
		\item $A$ is hyperbolic, i.e., the spectrum of $A$ is disjoint from the unit circle in $\CC$;
		\item $A$ is irreducible, i.e., its characteristic polynomial is irreducible over $\QQ$;
		\item no three eigenvalues of $A$ have the same absolute value, and if two eigenvalues of $A$ have the same absolute value then they are a pair of
		complex conjugate eigenvalues.
	\end{enumerate}
	It was proved in \cite{GKS} that “most” automorphisms in $A\in{\rm GL}(d,\ZZ)$ are generic, that is, the proportion of generic automorphisms contained in ${\rm GL}(d,\ZZ)$  goes to 1 as $\|A\|\rightarrow+\infty$.

	Let $A\in{\rm GL}(d,\ZZ)$ be a generic automorphism with finest dominated splitting
	$$
	T\TT^d=L^s_1\oplus\cdots\oplus L^s_k\oplus L^u_1\oplus\cdots\oplus L^u_l,
	$$
	then ${\rm dim}L^{s/u}_i\leq 2$ for every $i$. Recall that every $f\in{\rm Diff}^1(\TT^d)$ which is sufficiently $C^1$-close to $A$ is Anosov and admits the dominated splitting 
	$$
	T\TT^d=E^s_1\oplus\cdots\oplus E^s_k\oplus E^u_1\oplus\cdots\oplus E^u_l
	$$
	which is a close to the splitting for $A$ and has matching dimensions of all subbubdles.

	\begin{theorem}\label{thm:local}
		Let $A\in{\rm GL}(d,\ZZ)$ be a generic automorphism and assume that $f\in{\rm Diff}^2(\TT^d)$ is $C^1$-close to $A$ with dominated splitting 
		$$
		T\TT^d=E^s_1\oplus\cdots\oplus E^s_k\oplus E^u_1\oplus\cdots\oplus E^u_l
		$$
		matching the finest dominated splitting for $A$ as above. Assume $k\geq 2$ and $1\leq i<k$, denote by
		$E^s_{(1,i)}=E^s_1\oplus\cdots\oplus E^s_i$. 
		Then the following properties are equivalent
		\begin{enumerate}
			\item The subbundle $E^s_{(1,i)}\oplus E^u$ is integrable.
			\item The subbundle $E^s_i\oplus E^u$ is integrable.
			\item $f$ is spectrally rigid along $E^s_j$ for every $j=i+1,\cdots,k$, i.e., the Lyapunov exponents of $f$ along $E^s_j$ for all ergodic measures are equal and  satisfy
			$$
			\lambda(E^s_j, f)=\lambda(L^s_j,A),
			\qquad \forall j=i+1,\cdots,k.
			$$
			\item Let the homeomorphism $h:\TT^d\rightarrow\TT^d$ be the conjugacy $h\circ f=A\circ h$, then $h$ is $C^{1+{\rm H\"older}}$-smooth along $E^s_j$ and $Dh(E^s_j)=L^s_j$ for every $j=i+1,\cdots,k$.
		\end{enumerate}
	\end{theorem}

	\subsection{Integrability of extremal bundles in dimension 4} 
	
	The case when the jointly integrable extremal subbundles are proper subbundles of the stable and unstable subbundles appears to be very challenging and requires different techniques. We obtain a result in dimension 4 which takes a particularly nice form in the symplectic setting.
	
	Let $A\in{\rm Sp}(4,\ZZ)$ be a symplectic Anosov automorphism with finest dominated splitting
	$$
	T\TT^4=
	L^{ss}\oplus L^{ws}\oplus L^{wu}\oplus L^{uu}.
	$$
	This is equivalent to $A$ not being conformal in the unstable subbundle.
	It is clear that $L^{ss}\oplus L^{uu}$ and $L^{ws}\oplus L^{wu}$ are $A$-invariant symplectic subbundles. 
	
	Denote by ${\rm Diff}^1_{\omega}(\TT^4)$ the space of $C^1$ diffeomorphisms on $\TT^4$ which preserve the standard symplectic form $\omega$.
	There exists a $C^1$-neighborhood $\cU_{\omega}$ of $A$ in ${\rm Diff}^1_{\omega}(\TT^4)$ such that every $f\in\cU_{\omega}$ admits the finest dominated splitting corresponding to the finest dominated splitting for $A$:
	$$
	T\TT^4=
	E^{ss}\oplus E^{ws}\oplus E^{wu}\oplus E^{uu}.
	$$
	where the subbundles $E^{ss}\oplus E^{uu}$ and $E^{ws}\oplus E^{wu}$ are $Df$-invariant symplectic subbundles. By the invariant manifolds theory~\cite{HPS} we have that that the center symplectic subbundle $E^{ws}\oplus E^{wu}$ is always integrable. 
	We call $E^{ss}\oplus E^{uu}$ \emph{the extremal symplectic subbundle} of $f$. 
	
	We show that the extremal symplectic subbundle of an irreducible $f$ is integrable if and only if $f$ is smoothly conjugate to $A$.
	
	\begin{theorem}\label{thm:T4}
		Let $A\in{\rm Sp}(4,\ZZ)$ be an irreducible non-conformal Anosov automorphism, and let $f\in{\rm Diff}^2_{\omega}(\TT^4)$ be a symplectic diffeomorphism which is sufficiently $C^1$-close to $A$. The extremal symplectic bundle of $f$ is integrable if and only if $f$ is $C^{1+{\rm H\"older}}$-smoothly conjugate to $A$.
	\end{theorem}

    \begin{remark}   
    	If  $f\in{\rm Diff}^{\infty}_{\omega}(\TT^4)$ is sufficiently close to $A$ in $C^r$ topology for certain large $r$, then the regularity of the conjugacy can be bootstrapped to $C^{\infty}$ by recent work of Kalinin, Sadovskaya and Wang~\cite{KSW}.
	%{\color{red}
	%Actually, we can prove the same result in the following setting: the irreducible Anosov automorphism $A\in{\rm Sp}(6,\ZZ)$ has dominated  splitting 
    	%$$
    	%T\TT^d=L^{ss}\oplus L^{ws}\oplus L^{wu}\oplus L^{uu}
    	%\qquad \text{and} \qquad {\rm dim}L^{ws}={\rm dim}L^{wu}=2
    	%$$ 
    	%with complex eigenvalues in $L^{ws},L^{wu}$, and $A$ satisfies center bunching condition with respect to the center bundle $L^{ws}\oplus L^{wu}$.
    	%Then for every symplectic diffeomorphism $f\in{\rm Diff}^2_{\omega}(\TT^4)$ being sufficiently $C^1$-close to $A$, the extremal symplectic bundle of $f$ is integrable if and only if $f$ is $C^{1+{\rm H\"older}}$-smoothly conjugate to $A$.}
    \end{remark}

	This theorem can be compared to a result of F. Rodriguez Hertz~\cite{FRH}. Let $A\in{\rm Sp}(4,\ZZ)$ is irreducible and partially hyperbolic with the center subbundle corresponding to a pair of complex eigenvalues with absolute value one. For every $C^{22}$-perturbation $f$ of $A$, if the stable and unstable bundles of $f$ are jointly integrable to $su$-foliation, then $f$ smoothly conjugates to $A$ along the center foliation.
	This result relies on the strong center bunching condition which gives very high regularity of $su$-foliation. Then KAM theory can be applied to the action of $su$-foliation to obtain smoothness of the conjugacy~\cite{FRH}. In Theorem~\ref{thm:T4general}, we don't have any center bunching assumptions. So instead we work with the solvable action induced by the $su$-foliation, see Section~\ref{subsec:solvable}.
	
	In fact, we can prove a more general result.
	\begin{theorem}\label{thm:T4general}
		Let $A\in{\rm GL}(d,\ZZ)$ be an irreducible Anosov automorphism which admits a dominated splitting $T\TT^d=L^{ss}\oplus L^{ws}\oplus L^{wu}\oplus L^{uu}$ with ${\rm dim}L^{ws}={\rm dim}L^{wu}=1$. For every $f\in{\rm Diff}^2(\TT^d)$ which is sufficiently $C^1$-close to $A$, let $T\TT^d=E^{ss}\oplus E^{ws}\oplus E^{wu}\oplus E^{uu}$ be the corresponding $Df$-invariant dominated splitting. Then $E^{ss}\oplus E^{uu}$ is integrable if and only if $f$ has spectral rigidity along $E^{ws}$ and $E^{wu}$, (cf. Theorem \ref{thm:C2}). 
	\end{theorem}
	%\marginpar{How would we go from $C^1$ to smooth is also non-trivial. Would need 4-dim. version of my bootstrap, which is probably ok, but needs checking.}

	 The symplectic structure in Theorem \ref{thm:T4} is only used to show the spectral rigidity along $E^{ss}$ and $E^{uu}$, which then gives the smooth conjugacy. We prove Theorems~\ref{thm:T4} and~\ref{thm:T4general} in Section~\ref{sec:T4}.

It appears to be very challenging to generalize the results of this section to dimensions $>4$. For example, we do not know how to do it for irreducible symplectic automorphisms of $\TT^6$ with real spectrum. Our methods of proof relies on a ``ping-pong argument'' on the center (weak) leaf $\cF^{ws}\oplus\cF^{wu}$. This argument is very 2-dimensional and doesn't generalize to higher dimensions. Still, bringing some more ideas, such as combinatorics of subspace intersections under the strong unstable holonomy we can handle a 6-dimensional case with complex eigenvalues.

\begin{theorem}\label{thm:C4-complex}
	Let $A\colon\TT^6\to\TT^6$ be an irreducible hyperbolic automorphism which admits a dominated splitting
	$$
	T\TT^6=L^{ss}\oplus L^{ws}\oplus L^{wu}\oplus L^{uu}
	\qquad \text{with} \qquad
	{\rm dim}L^{ws}={\rm dim}L^{wu}=2.
	$$
	Assume that $A$ has complex eigenvalues $\mu^{ws},\bar\mu^{ws}$ in $L^{ws}$, and complex eigenvalues $\mu^{wu},\bar\mu^{wu}$ in $L^{wu}$, and that $A$ satisfies the center bunching conditions:
	$$
	\log|\mu^{ws}|-\log|\mu^{wu}|>\log|\mu^{ss}|,
	\qquad \text{and} \qquad
	\log|\mu^{wu}|-\log|\mu^{ws}|<\log|\mu^{uu}|,
	$$
	where $\mu^{ss},\mu^{uu}$ are eigenvalues of $A$ associated to $L^{ss}, L^{uu}$, respectively.
	
	For every $f\in{\rm Diff}^2(\TT^6)$ which is sufficiently $C^1$-close to $A$, let $T\TT^d=E^{ss}\oplus E^{ws}\oplus E^{wu}\oplus E^{uu}$ be the corresponding dominated splitting. Then $E^{ss}\oplus E^{uu}$ is integrable if and only if $f$ is spectrally rigid along $E^{ws}\oplus E^{wu}$.
\end{theorem}

We will not include a proof of this theorem in the interests of space.

	\section{Preliminaries}\label{sec:preli}
	
	We say that a foliation $\cL$ on $\TT^d$ is \emph{a linear foliation} if there exists a linear subspace $V\subset\RR^d$ such that the lifted foliation $\tilde{\cL}$ on $\RR^d$ is given by the translates $x+V, x\in\RR^d$. It is well-known that if one leaf of a linear foliation $\cL$ is dense on $\TT^d$, then every leaf of $\cL$ is dense, i.e., $\cL$ is minimal.
	
	Let $\cL$ and $\cF$ be two foliations on a closed manifold $M$ with $\dim \cL<\dim\cF$. We say \emph{$\cF$ is sub-foliated by $\cL$} if for every $x\in M$ we have $\cL(x)\subset\cF(x)$. Then the restriction $\cL|_{\cF(x)}$ forms a foliation on $\cF(x)$.

	\begin{lemma}\label{lem:linear-foli}
		If $\cF$ is a $C^0$-foliation on $\TT^d$ which is sub-foliated by a minimal linear foliation $\cL$, then $\cF$ is a minimal linear foliation. 
	\end{lemma}
	
	\begin{proof}
		Let $\tilde{\cF}$ and $\tilde{\cL}$ be the lifts of $\cF$ and $\cL$ to $\RR^d$. Let $0\in\RR^d$ be the origin, and let $\pi:\RR^d\rightarrow\TT^d$ be the covering map. Since $\pi(\tilde{\cF}(0))=\cF(\pi(0))\supset\cL(\pi(0))$ is dense in $\TT^d$, we only need to show that $\tilde{\cF}(0)$ is a linear subspace of $\RR^d$.
		
		We proceed to that show $\tilde{\cF}(0)$ is closed under addition. Since $\cL$ is minimal, for any $x,y\in\tilde{\cF}(0)$, there exist $k_m\in\ZZ^d$ and $v_m\in\tilde{\cL}(0)$, such that $k_m+v_m\rightarrow x$ as $m\rightarrow\infty$. Since $\tilde{\cF}$ is sub-foliated by $\tilde{\cL}$, we have
		$$
		y+k_m+v_m\in\tilde{\cF}(y+k_m)=\tilde{\cF}(k_m)=\tilde{\cF}(k_m+v_m).
		$$
		Sending $m\rightarrow\infty$, we obtain $x+y\in\tilde{\cF}(x)=\tilde{\cF}(0)$ by the continuity.
		
		In a similar way, if $x\in\tilde{\cF}(0)$ then $0\in \tilde{\cF}(x)$ and we can find $k_m\in\ZZ^d$ and $v_m\in\tilde{\cL}(0)$, such that $k_m+v_m+x\rightarrow 0$ as $m\to\infty$. Then
		$$
		k_m+v_m+x \in\tilde\cF(k_m+x)=\tilde\cF(k_m)=\tilde\cF(k_m+v_m)
		$$
		Taking the limit we have $0\in \tilde\cF(-x)$ and, hence $-x\in \tilde\cF(0)$.
		We can conclude that $\tilde{\cF}(0)$ is a closed connected subgroup of $\RR^d$. It follows that $\tilde{\cF}(0)$ is a linear subspace of $\RR^d$ and, hence, $\cF$ is a minimal linear foliation on $\TT^d$.
	\end{proof}
	
	Let $\cF$ be a foliation on a closed Riemannian manifold $M$, and $\tilde{\cF}$ be its lift to the universal cover $\tilde{M}$. We say $\cF$ is \emph{quasi-isometric} if there exist $a,b>0$, such that for every $x\in\tilde{M}$ and $y\in\tilde{\cF}(x)$ we have
	$$
	d_{\tilde{\cF}}(x,y)\leq a\cdot d(x,y)+b,
	$$ 
	where  $d(\cdot,\cdot)$ denotes the distance on $\tilde{M}$ induced by the lifted Riemannian metric and  $d_{\tilde{\cF}}(\cdot,\cdot)$ denotes the distance along the leaf of $\tilde{\cF}$ induced by the restriction of the Riemannian metric to the leaves of $\tilde\cF$.

	\begin{lemma}\label{lem:quasi-isometric}
		Let $\cF$ be a $C^0$-foliation on $\TT^d$ with $C^1$-leaves. If there exists a homeomorphism $h:\TT^d\rightarrow\TT^d$ homotopic to identity, such that $h(\cF)$ is a linear foliation, then $\cF$ is quasi-isometric.
	\end{lemma}
	
	\begin{proof}
		First we note that since $\cF$ is a $C^0$-foliation with $C^1$-leaves, the tangent bundle $T\cF\subset T\TT^d$ is a uniformly continuous bundle on $\TT^d$. We can replace $h$ with a homeomorphism $\bar h$ which is $C^0$ close to $h$, $\bar h(\cF)=h(\cF)$ and $\bar h$ is a uniformly $C^1$ diffeomorphism when restricted to the leaves of $\cF$. This can be achieved in a standard way by mollifying $h$ along the leaves of $\cF$ in foliation charts. We denote by $\tilde h$ the lift of $\bar h$. We have that $\tilde h-id_{\RR^d}$ is bounded.
		We will now work on the universal cover $\RR^d$.
		
		For every $x\in\RR^d$ consider the lines $\ell$ which pass through $\tilde h(x)$ and are contained in $\tilde h(\tilde\cF(x))$. For each such $x$ and $\ell$ consider the following quantity 
		which measures how close a leaf of $\tilde\cF$ can return to itself along $\tilde h^{-1}(\ell)$.
		$$
		\cR(x,\ell)=\min\left\{1, \inf_{y\in \tilde h^{-1}(\ell), d_{\tilde{\cF}}(x,y)\ge 1}d(x,y)\right\}
		$$
		Because $\tilde h$ is finite distance away from $id_{\RR^d}$ for points on $\ell$ which are sufficiently far from $x$, their images on $h^{-1}(\ell)$ are also sufficiently far. Hence there exists $c_0>1$ such that 
		$$
		\cR(x,\ell)=\min\left\{1, \inf_{y\in \tilde h^{-1}(\ell), c_0 \ge d_{\tilde{\cF}}(x,y)\ge 1}d(x,y)\right\}
		$$
		From this expression it is clear that $\cR$ is a positive function which is continuous in both $x$ and $\ell$. Further, because $\cR$ is positive and equivariant under the Deck group action we have that $\cR$ is uniformly positive, that is, there exists $\eps_0>0$ such that $\cR(x,\ell)>2\eps_0$.
		
		Now consider any straight segment $[\tilde h(a), \tilde h(b)]\subset \tilde h(\tilde\cF(a))$. Then $\gamma=\tilde h^{-1}([\tilde h(a), \tilde h(b)])$ is a uniformly $C^1$ curve (recall that $\tilde h$ is $C^1$ along the leaves ) which connects $a$ to $b$ inside $\tilde \cF(a)$. Curve $\gamma$ belongs to a pre-image of a line through $\tilde h(a)$. Because of the above discussion $\gamma$ does not return to itself within $2\eps_0$ inside $\tilde \cF(a)$. Since $\tilde \cF(a)\subset \RR^d$, after taking even smaller $\eps_0$ if needed, the same is true in $\RR^d$: curve $\gamma$ does not return to itself within $2\eps_0$ inside $\RR^d$. Hence the map $\tilde h^{-1}:[\tilde h(a),\tilde h(b)]\to\gamma$ can be extended to an injective map of the $\eps_0$-neighborhood of $[\tilde h(a),\tilde h(b)]$ to the $\eps_0$-neighborhood $\cU(\gamma,\eps_0)$ of $\gamma$. Indeed, since $\gamma$ is a $C^1$ curve, we can assume that $\gamma$ is almost linear on the $\eps_0$-scale (otherwise, again, we can replace $\eps_0$ with an even smaller number)  and use normal hyperplanes to $\gamma$ to define such maps of $\eps_0$-neighborhoods. Because this map is injective we have an estimate on the volume
		$$
		{\it vol}(\cU(\gamma,\eps_0))\ge C_1\eps_0^{d-1} {\it length}(\gamma)
		$$
		On the other hand because $\tilde h$ is within bounded distance, say $R_0$, away from identity, we have that $\gamma$ belongs to an $R_0$-neighborhood of $[\tilde h(a), \tilde h(b)]$. And, hence, by choosing slightly larger $R_0$ we also have that $\cU(\gamma,\eps_0)$ belongs to the $R_0$-neighborhood of $[\tilde h(a), \tilde h(b)]$ and we can estimate the volume from above
		$$
		{\it vol}(\cU(\gamma,\eps_0))\le C_2d(\tilde h(a),\tilde h(b)) R_0^{d-1}+V_0
		$$
		where $V_0$ is the volume of a ball of radius $R_0$. Putting these estimates together we have 
		\begin{align*}
		d_{\tilde \cF}(a,b)\le  {\it length}(\gamma) 
		&\le \frac{{\it vol}(\cU(\gamma,\eps_0))}{C_1\eps_0^{d-1}}
		\le\frac{C_2 R_0^{d-1}}{C_1\eps_0^{d-1}}d(\tilde h(a),\tilde h(b)) +\frac{V_0}{C_1\eps_0^{d-1}}\\
		&\le \frac{C_2 R_0^{d-1}}{C_1\eps_0^{d-1}}d(a,b) +\frac{2 C_2 R_0^{d}}{C_1\eps_0^{d-1}}+\frac{V_0}{C_1\eps_0^{d-1}}
		\end{align*}
		which precisely means that $\tilde\cF$ is a quasi-isometric foliation.
	\end{proof}
	
	The following lemma is due to Brin.
	
	\begin{lemma}[Theorem 1, \cite{Br}]\label{lem:Brin}
		Let $f$ be an absolutely partially hyperbolic diffeomorphism on a closed Riemannian manifold $M$. If both stable and unstable foliations of $f$ are quasi-isometric in the universal cover $\tilde{M}$, then the center-unstable and center-stable bundles of $f$ are locally uniquely integrable. In particular, $f$ is dynamically coherent.
	\end{lemma}

	\section{Matching to the linear model: proof of Theorem \ref{thm:conjugate}}
	
	\begin{proof}[Proof of Theorem \ref{thm:conjugate}]
		Let $A\in {\rm GL}(d,\ZZ)$ be the linear  part of $f$, then $A$ is Anosov with the hyperbolic splitting $T\TT^d=L^s\oplus L^u$ of matching dimensions, ${\dim}L^s={\dim}E^s$ and ${\dim}L^u={\dim}E^u$. Denote by $\cF^s,\cF^u$ the stable and unstable foliations of $f$, and by $\cL^s,\cL^u$ the stable and unstable foliations of $A$. If $h:\TT^d\rightarrow\TT^d$ is the topological conjugacy $h\circ f=A\circ h$, then
		$$
		h(\cF^s)=\cL^s, \qquad {\rm and} \qquad h(\cF^u)=\cL^u.
		$$ 
		
		By our assumption $E^{ss}\oplus E^u$ is integrable and we let $\cF^{su}$ be the foliation tangent to $E^{ss}\oplus E^u$. Clearly, $\cF^{su}$ is sub-foliated by the unstable foliation $\cF^u$. This implies $h(\cF^{su})$ is sub-foliated by the minimal linear foliation $\cL^u=h(\cF^u)$. Therefore, by Lemma~\ref{lem:linear-foli},  $\cL^{su}=h(\cF^{su})$ is a minimal linear foliation on $\TT^d$. Hence
		$$
		\cF^{ss}=\cF^s\cap\cF^{su}=h^{-1}\big(\cL^s\cap\cL^{su}\big)
		$$
		which is the $(h^{-1})$-image of a linear foliation $\cL^{ss}:=\cL^s\cap\cL^{su}$. By Lemma \ref{lem:quasi-isometric} we have that $\cF^{ss}$ is quasi-isometric.
		
		Since both $\cF^{ss}$ and $\cF^u=h^{-1}(\cL^u)$ are quasi-isometric, and $f$ is absolutely partially hyperbolic, Lemma~\ref{lem:Brin} implies that $E^{ws}\oplus E^u$ is locally uniquely integrable to a foliation $\cF^{cu}$, and $E^{ws}$ is locally uniquely integrable to a foliation $\cF^{ws}=\cF^{cu}\cap \cF^s$. Similarly, since $\cF^{cu}$ is sub-foliated by $\cF^u$ and $\cL^u=h^{-1}(\cF^u)$ is linear and minimal, we have that $\cL^{cu}:=h^{-1}(\cF^{cu})$ is also a minimal linear foliation. Thus
		$$
		\cL^{ws}:=h^{-1}(\cF^{ws})=h^{-1}\big( \cF^{cu}\cap \cF^s \big)=\cL^{cu}\cap\cL^s
		$$
		is a linear foliation, and $\cF^{ws}$ is quasi-isometric.
		
		Since both $\cF^{ss}$ and $\cF^{ws}$ are $f$-invariant and transverse inside $\cF^s$, both linear foliations $\cL^{ss}$ and $\cL^{ws}$ are $A$-invariant
		$$
		A(\cL^{\sigma})=h\circ f\circ h^{-1}(\cL^{\sigma})
		=h\circ f(\cF^{\sigma})=h(\cF^{\sigma})=\cL^{\sigma}, \qquad \sigma={\it ss, ws},
		$$
		and transverse inside $\cL^s$. This implies $A$ admits an invariant linear splitting
		$$
		T\TT^d=L^{ss}\oplus L^{ws}\oplus L^u, \qquad {\rm with} \qquad
		T{\cL^{ss}}=L^{ss}, \quad{\rm and}\quad T{\cL^{ws}}=L^{ws}.
		$$ 
		Now we need to show $L^s=L^{ss}\oplus L^{ws}$ is a dominated splitting. The following claim proves the second item of Theorem \ref{thm:conjugate}. Recall that the constant $\mu$ comes from the definition of partial hyperbolicity~$(\ast)$ for $f$.
		
		\begin{claim}
			We still denote by $A:\RR^d\rightarrow\RR^d$ the lift of $A$ to $\RR^d$, and by $\tilde{\cL}^{ss},\tilde{\cL}^{ws}$ and $\tilde{\cL}^s$ the lifts of foliations $\cL^{ss},\cL^{ws}$ and $\cL^s$ to $\RR^d$. 
			For every $x\in\RR^d$ and every $y\in\tilde{\cL}^s(x)$, we have
			\begin{itemize}
				\item if $y\in\tilde{\cL}^{ws}(x)$, then~
				$\lim_{n\rightarrow-\infty}\mu^{-n}\cdot d\big(A^n(x),A^n(y)\big)=0$;
				\item if $y\in\tilde{\cL}^{ss}(x)$, then~
				$\lim_{n\rightarrow-\infty}\mu^{-n}\cdot d\big(A^n(x),A^n(y)\big)=+\infty$.
			\end{itemize}
			Then is easily follows that for any eigenvalues $\mu^{ss}$ and $\mu^{ws}$ of $A$ associated to invariant subspaces corresponding $L^{ss}$ and $L^{ws}$, respectively, satisfy
			$$
			0<|\mu^{ss}|<\mu<|\mu^{ws}|<1.
			$$
		\end{claim}
		
		\begin{proof}[Proof of the Claim]
			Let $F:\RR^d\rightarrow\RR^d$ be a lift of $f$ and let $H:\RR^d\rightarrow\RR^d$ be a lift of the conjugacy satisfying
			$$
			H\circ F=A\circ H, \qquad {\rm and } \qquad 
			H(x+k)=H(x)+k, \quad \forall k\in\ZZ^d.
			$$
			There exists a constant $C>0$ such that $d(H(x),x)<C$ for every $x\in\RR^d$. 
			The constant $\mu$ satisfies
			$$
			0<\sup_{z\in\TT^d}\|Df|_{E^{ss}(z)}\|<\mu<
			\inf_{z\in\TT^d}m(Df|_{E^{ws}(z)})<1.
			$$
			
			We denote by $\tilde{\cF}^{ss},\tilde{\cF}^{ws},\tilde{\cF}^s$ the lifts of $\cF^{ss},\cF^{ws},\cF^s$ to $\RR^d$. For every $x\in\RR^d$ and every $y\in\tilde{\cL}^s(x)$, $y\in\tilde{\cL}^{ws}(x)$ if and only if $H^{-1}(y)\in\tilde{\cF}^{ws}(H^{-1}(x))$. If $y\in\tilde{\cL}^{ws}(x)$, then
			\begin{align*}
			\lim_{n\rightarrow-\infty}\frac{d\big(A^n(x),A^n(y)\big)}{\mu^n}
			&=\lim_{n\rightarrow-\infty}
			\frac{d\big(H\circ F^n\circ H^{-1}(x), H\circ F^n\circ H^{-1}(y)\big)}{\mu^n} \\
			&\leq \lim_{n\rightarrow-\infty}
			\frac{d\big(F^n\circ H^{-1}(x), F^n\circ H^{-1}(y)\big)+2C}{\mu^n} \\
			&\leq \lim_{n\rightarrow-\infty}
			\frac{d_{\tilde{\cF}^{ws}}\big(F^n\circ H^{-1}(x), F^n\circ H^{-1}(y)\big)+2C}{\mu^n}\\
			&\leq \lim_{n\rightarrow-\infty}
			\frac{\big[\inf_{z\in\TT^d}m(Df|_{E^{ws}(z)})\big]^n}{\mu^n}\cdot
			d_{\tilde{\cF}^{ws}}\big(H^{-1}(x),H^{-1}(y)\big)+
			\frac{2C}{\mu^n} \\
			&=0.
			\end{align*}
			
			Recall that $\tilde{\cF}^{ss}$ is quasi-isometric. Let $a$ and $b$ be the  quasi-isometric constants for $\tilde{\cF}^{ss}$. If $y\in\cL^{ss}(x)$, then we have
			\begin{align*}
			\lim_{n\rightarrow-\infty}\frac{d\big(A^n(x),A^n(y)\big)}{\mu^n}
			&=\lim_{n\rightarrow-\infty}
			\frac{d\big(H\circ F^n\circ H^{-1}(x), H\circ F^n\circ H^{-1}(y)\big)}{\mu^n} \\
			&\geq \lim_{n\rightarrow-\infty}\frac{d\big(H\circ F^n\circ H^{-1}(x), H\circ F^n\circ H^{-1}(y)\big)}{\mu^n} \\
			&\geq \lim_{n\rightarrow-\infty}
			\frac{d\big(F^n\circ H^{-1}(x), F^n\circ H^{-1}(y)\big)-2C}{\mu^n} \\
			&\geq \lim_{n\rightarrow-\infty}
			\frac{a^{-1}\cdot d_{\tilde{\cF}^{ss}}\big(F^n\circ H^{-1}(x), F^n\circ H^{-1}(y)\big)-a^{-1}b-2C}{\mu^n} \\
			&\geq \lim_{n\rightarrow-\infty}a^{-1}\cdot
			\frac{\big[\sup_{z\in\TT^d}\|Df|_{E^{ss}(z)}\|\big]^n}{\mu^n}
			\cdot d_{\tilde{\cF}^{ss}}\big(H^{-1}(x),H^{-1}(y)\big) -
			\frac{a^{-1}b+2C}{\mu^n} \\
			&=+\infty.
			\end{align*}
			% 	
			%   	Finally, since $A$ is linear, if $\mu^{ss}$ and $\mu^{ws}$ are eigenvalues of $A$ associated to invariant subspaces defined by $L^{ss}$ and $L^{ws}$ respectively, then they satisfy
			%  	$0<|\mu^{ss}|<\lambda<|\mu^{ws}|<1$.
		\end{proof}
		
		The gap between the eigenvalues established in the preceding claim implies that the $A$-invariant splitting $T\TT^d=L^{ss}\oplus L^{ws}\oplus L^u$ is a partially hyperbolic splitting. This proves the first item of Theorem~\ref{thm:conjugate}. Note that we have showed the conjugacy $h$ preserves all the invariant foliations, which proves the third item.  
	\end{proof}
	
	We have finished the proof of Theorem~\ref{thm:conjugate}. We include two more technical lemmas in this section which will be used later.
	
	\begin{lemma}\label{lem:C1-su}
		Let $f\in{\rm Diff}^2(\TT^d)$ be an irreducible Anosov diffeomorphism with absolutely partially hyperbolic splitting $T\TT^d=E^{ss}\oplus E^{ws}\oplus E^u$. If $E^{ss}\oplus E^u$ is integrable and $f$ satisfies the center bunching condition 
		$$
		\|Df|_{E^{ws}(x)}\|~<~
		m\big(Df|_{E^{ws}(x)}\big)\cdot m\big(Df|_{E^u(x)}\big),
		\qquad \forall x\in\TT^d,
		$$
		then the $su$-foliation $\cF^{su}$ with $T\cF^{su}=E^{ss}\oplus E^u$ is a  $C^{1+{\rm H\"older}}$ foliation on $\TT^d$. 
		
		In particular, for every $x\in\TT^d$ and $y\in\cF^{su}(x)$, let  $\,\,\H^{su}_{x,y}:\cF^{ws}(x)\rightarrow\cF^{ws}(y)$ be the holonomy map given by sliding along $\cF^{su}$, satisfying $\H^{su}_{x,y}(x)=y$, then $\H^{su}_{x,y}$ is $C^{1+{\rm H\"older}}$.
	\end{lemma}
	
	\begin{proof}
		By Theorem \ref{thm:conjugate}, $f$ is dynamically coherent with stable foliation $\cF^s$ and center unstable foliation $\cF^{cu}$, where
		$$
		T\cF^s=E^{ss}\oplus E^{ws}, \qquad {\rm and} \qquad
		T\cF^{cu}=E^{ws}\oplus E^u.
		$$ 
		Let $\cF^{ws}=\cF^s\cap\cF^{cu}$ be the weak stable foliation satisfying $T\cF^{ws}=E^{ws}$.
		
		Since $f$ is Anosov and absolutely partially hyperbolic, there exists $\e>0$ such that
		$$
		\|Df|_{E^{ss}(x)}\|^{(1-\e)}\cdot\|Df|_{E^{ws}(x)}\|~<~m\big(Df|_{E^{ws}(x)}\big)
		\qquad \forall x\in\TT^d.
		$$
		Then applying Theorem B of \cite{PSW} and Theorem 2.2 of \cite{Bro} we obtain that the strong stable foliation $\cF^{ss}$ is $C^{1+{\rm H\"older}}$-smooth in every stable leaf $\cF^s(x)$, $x\in\TT^d$. Moreover, the center bunching assumption implies that there exists $\e'>0$ such that
		$$
		\|Df|_{E^{ws}(x)}\|~<~
		m\big(Df|_{E^{ws}(x)}\big)\cdot m\big(Df|_{E^u(x)}\big)^{(1-\e')},
		\qquad \forall x\in\TT^d.
		$$
		Again, by Theorem B of \cite{PSW} and Theorem 2.2 of \cite{Bro} we have that the unstable foliation $\cF^u$ is $C^{1+{\rm H\"older}}$-smooth in every stable leaf $\cF^{cu}(x)$, $x\in\TT^d$. Since $E^{ss}\oplus E^u$ is integrable, this implies the holonomy map induced by the $su$-foliation $\cF^{su}$ between every pair of weak stable leaves in $\cF^{ws}$ is uniformly $C^{1+{\rm H\"older}}$-smooth, and, hence, $\cF^{su}$ is a $C^{1+{\rm H\"older}}$-foliation on $\TT^d$.
	\end{proof}

	\begin{lemma}\label{lem:Diophantine}
		Let $f\in{\rm Diff}^1(\TT^d)$ be an Anosov diffeomorphism satisfying the assumptions of Theorem \ref{thm:conjugate}.
		Let $A=f_*\in{\rm GL}(d,\ZZ)$ be the linear part of $f$ and let $h:\TT^d\rightarrow\TT^d$ be the conjugacy: $h\circ f=A\circ h$.
		There exist $\,\,\theta,\theta'\in(0,1)$, such that for every pair of points $x,y\in\TT^d$, if we let $x'=h(x),y'=h(y)$, then
		\begin{itemize}
			\item there exists two sequences of points
			$$
			z_n'\in\cL^{ss}(x'), \qquad {\it and} \qquad
			w_n'\in\cL^{ws}(y')\cap\cL^u(z_n'), \qquad n\in\NN,
			$$
			such that 
			$$
			d_{\cL^{ws}}(w_n',y')~\leq~ L_n^{-\theta'},
			\qquad {\it and} \qquad
			d_{\cL^u}(w_n',z_n')~\leq~ L_n^{-\theta'},
			$$ 
			where $L_n=d_{\cL^{ss}}(x',z_n')\rightarrow+\infty$ as $n\rightarrow\infty$.

			\item let
			$z_n=h^{-1}(z_n')\in\cF^{ss}(x)$ and $w_n=h^{-1}(w_n')\in\cF^{ws}(y)\cap\cF^u(z_n)$ for every $n\in\NN$. Then
			$$
			d_{\cF^{ws}}(w_n,y)~\leq~ K_n^{-\theta},
			\qquad {\it and} \qquad
			d_{\cF^u}(w_n,z_n)~\leq~ K_n^{-\theta},
			$$ 
			where $K_n=d_{\cF^{ss}}(x,z_n)\rightarrow+\infty$ as $n\rightarrow\infty$.
		\end{itemize}
		
	\end{lemma}

\begin{proof}
We begin by noticing that, since by Theorem~\ref{thm:conjugate} conjugacy $h$ matches all invariant foliations, and because $h$ is H\"older continuous, the second statement is an immediate corollary of the first one. Hence we only need to verify the first one which is a qualitative density property of the strong stable foliation. 

The fact that the foliation $\cL^{ss}$ is algebraic, hence Diophantine, will play an essential role in the proof. Denote by $\cL^{ss}(x', R)$ a ball of radius $R$ inside $\cL^{ss}(x')$. If we can show that $\cL^{ss}(x', R)$ is $CR^{-\theta'}$-dense in $\mathbb T^d$ then the claim follows. Indeed, the point $y'$ can be joined to a point in $\cL^{ss}(x', R)$ via a path inside $\cL^{ws}\oplus \cL^u(y')$ of length $\le CR^{-\theta'}$ and the desired estimates on $d_{\cL^{ws}}(w_n',y')$ and $d_{\cL^u}(w_n',z_n')$ easily follow from transversality of $\cL^{ws}$ and $\cL^u$ inside $\cL^{ws}\oplus \cL^u(y')$. (By taking $R\to\infty$ we will have $L_n\to\infty$ and, choosing a slightly smaller $\theta'$ if needed be can eliminate the constant $C$.)

Let's look at the spectrum of $A$ along the strong stable subspace corresponding to $\cL^{ss}$. We will consider two cases. The first easier case is that $A$ has a real eigenvalue along strong stable subspace. In this case $\cL^{ss}$ admits a 1-dimensional invariant foliation $\cL\subset \cL^{ss}$ corresponding to this eigenvalue. In the second case, there are no real eigenvalues along $\cL^{ss}$, and we have a 2-dimensional invariant sub-foliation $\cL\subset \cL^{ss}$ which corresponds to a pair of complex conjugate eigenvalues. Note that in both cases we know that $\cL$ is a completely irrational minimal foliation since it is an invariant foliation of an irreducible automorphism $A$. In both cases we will, in fact, prove quantitative denseness of  $\cL$. Clearly, because $\cL(x', R)\subset \cL^{ss}(x',R)$ quantitative denseness of $\cL^{ss}$ follows immediately.
\\

\noindent{\bf Case 1.} There exists a 1-dimensional $A$-invariant foliation $\cL\subset\cL^{ss}$.

Recall that we would like to show that a segment of a leaf of $\cL$ of length $2R$ is $CR^{-\theta'}$-dense in $\TT^d$. We can discretize this problem by considering the codimension one transversal $\TT^{d-1}\subset\TT^d$ given by $\TT^{d-1}=\{(x_1, x_2, \ldots x_{d-1}, 0)\in\TT^d\}$. Then the holonomy of $\cL$ induces an irrational translation $T_{\bar \alpha}\colon\TT^{d-1}\to\TT^{d-1}$ given by $ y\mapsto y+\bar \alpha$, where $\bar\alpha\in\TT^{d-1}$ is a Diophantine vector, because it's coordinates are given by the eigenvector coordinates. Now we can quote the following result about discrepancy of the Diophantine translation on tori~\cite[Theorem 1.80]{DT}.
\begin{claim} 
\label{claim_dis}
Let $T_{\bar \alpha}\colon\TT^{d-1}\to\TT^{d-1}$ be translation on the torus with $\bar\alpha$ Diophantine (badly approximable). Then the following upper bound on discrepancy of finite length orbits holds
$$
D(\{y, T_{\bar \alpha}y, T_{\bar \alpha}^2y,\ldots T_{\bar \alpha}^{N-1}y\})\le \frac{C\log^2N}{N^\frac{1}{d-1}}
$$
\end{claim}
Recall that discrepancy measures the maximal deviation of the finite orbit from the uniform distribution on $\TT^{d-1}$ over all balls. In particular, the discrepancy gives an upper bound on the $(d-1)$-volume of the largest ball which is disjoint with the finite orbit. Accordingly, we obtain an upper bound $\Delta(d-1, N)$ on the radius of such ball:
$$
\Delta(d-1, N)=\left(\frac{C\log^2N}{N^\frac{1}{d-1}}\right)^{1/(d-1)}\simeq \frac{C \log^{2/(d-1)}N}{N^\frac{1}{(d-1)^2}}
$$
Hence, any finite orbit sequence $\{y, T_{\bar \alpha}y, T_{\bar \alpha}^2y,\ldots T_{\bar \alpha}^{N-1}y\}$ is $\Delta(d-1, N)$-dense. By taking $\theta'<\frac{1}{(d-1)^2}$ we obtain that any finite orbit of length $N$ is $CN^{-\theta'}$-dense, which completes the proof in the Case~1.
\\

\noindent{\bf Case 2.} There exist a 2-dimensional $A$-invariant foliation $\cL\subset\cL^{ss}$.
The proof in this case follows the same scheme and is based on the same bound on discrepancy, however it is somewhat more technical since we need to work with two dimensional foliation $\cL$. Accordingly we fix a codimension two transversal $\TT^{d-2}\subset\TT^d$ given by $\TT^{d-2}=\{(x_1, x_2, \ldots x_{d-2}, 0, 0)\in\TT^d\}$. Then the foliation $\cL$ induces a $\ZZ^2$-action on $\TT^{d-2}$ whose generators are translations $T_{\bar\alpha}$ and $T_{\bar\beta}$. 

Because $\cL$ is minimal the full $\ZZ^2$-action is also minimal. However, while the coordinates of $\bar\alpha$ and $\bar\beta$ are still algebraic, it could happen that the coordinates of $\bar \alpha$ (or $\bar\beta$ or both) are rationally dependent and, accordingly, the orbits of $T_{\bar \alpha}$ (or $T_{\bar\beta}$ or both) are not dense. If either $T_{\bar \alpha}$ or $T_{\bar \beta}$ is minimal then we can finish the proof in exactly the same way as in Case~1. Otherwise, we can denote by $\TT_{\bar \alpha}$ and $\TT_{\bar \beta}$ the orbit closures $\overline{\{T^n_{\bar\alpha}(y); n\in\ZZ\}}$ and $\overline{\{T^n_{\bar\beta}(y); n\in\ZZ\}}$, correspondingly. These are tori of some positive dimension $<d-2$. Because the joint action is minimal we also have that $\dim \TT_{\bar \alpha}+\dim \TT_{\bar \beta}\ge d-2$ and  $\TT_{\bar \alpha}$ is transverse to $\TT_{\bar \beta}$ inside $\TT^{d-2}$.

Restricting the dynamics of $T_{\bar \alpha}$ to $\TT_{\bar \alpha}$ eliminates all rational dependencies and, hence, this restriction is a Diophantine completely irrational translation on $\TT_{\bar \alpha}$. Therefore, using the Claim~\ref{claim_dis} and arguing in the same way as in Case~1 we have that the finite orbit $\{y, T_{\bar \alpha}y, T_{\bar \alpha}^2y,\ldots T_{\bar \alpha}^{N-1}y\}$ is 
$\Delta(\dim \TT_{\bar \alpha}, N)$-dense inside $\TT_{\bar \alpha}$. By the same token we have that
the finite orbit $\{y, T_{\bar \beta}y, T_{\bar \beta}^2y,\ldots T_{\bar \alpha}^{N-1}y\}$ is 
$\Delta(\dim \TT_{\bar \beta}, N)$-dense inside $\TT_{\bar \beta}$. Now, since we have that $\TT_{\bar \alpha}$ and $\TT_{\bar \beta}$
``span'' $\TT^{d-2}$ we have that the $\ZZ^2$-finite orbit $\{T_{\bar\alpha}^i T_{\bar\beta}^jy: 0\le i,j\le N-1\}$ is
$\max\{\Delta(\dim \TT_{\bar \alpha}, N),\Delta(\dim \TT_{\bar \beta}, N)\}$-dense in $\TT^{d-2}$, which implies that it is $CN^{-\theta'}$-dense for appropriately small $\theta'>0$. Finally, the quantitative denseness of disks in $\cL$ follows from $CN^{-\theta'}$-denseness of the finite $\ZZ^2$-orbits in the standard way.
\end{proof}

\section{Local Lyapunov exponents rigidity}

In this section, we prove Theorem \ref{thm:local}. Let $A\in{\rm GL}(d,\ZZ)$ be a generic automorphism. Recall that this means the following.
\begin{enumerate}
	\item $A$ is hyperbolic, i.e. the spectrum of $A$ is disjoint from the unit circle in $\CC$;
	\item $A$ is irreducible, i.e. its characteristic polynomial is irreducible over $\QQ$;
	\item no three eigenvalues of $A$ have the same absolute value, and if two eigenvalues of $A$ have the same absolute value then they are a pair of
	complex conjugate eigenvalues.
\end{enumerate}

Assume $A$ has the following finest dominated splitting on $\TT^d$:
$$
T\TT^d=L^s_1\oplus\cdots\oplus L^s_k\oplus L^u_1\oplus\cdots\oplus L^u_l.
$$
By the above conditions ${\rm dim}L^{s/u}_i\leq 2$ for every $i$. Moreover, if ${\rm dim}L^{s/u}_i=2$, then $A$ has complex eigenvalues in the associated eigenspace. We further assume $k\geq 2$.

Every $f\in{\rm Diff}^{1+\alpha}(\TT^d)$, which is $C^1$-close to $A$,  is also Anosov and admits the  dominated splitting 
$$
T\TT^d=E^s_1\oplus\cdots\oplus E^s_k\oplus E^u_1\oplus\cdots\oplus E^u_l
$$
with dimensions which match the dimensions of the splitting for $A$. Then the following properties of $f$ are well-known:
\begin{enumerate}
	\item For every $i=1,\cdots,k-1$ and $j=1,\cdots,l-1$, 
	$$
	\max_{x\in\TT^d}\|Df|_{E^s_i(x)}\|~<~\min_{x\in\TT^d}m\big(Df|_{E^s_{i+1}(x)}\big),
	\qquad {\rm and} \qquad
	\max_{x\in\TT^d}\|Df|_{E^u_j(x)}\|~<~\min_{x\in\TT^d}m\big(Df|_{E^u_{j+1}(x)}\big).
	$$
	\item $f$ has a fixed point $p_f$ close to $0$ and if ${\rm dim}E^{s/u}_i=2$, then the derivative $\big(Df|_{E^{s/u}_i(p_f)}\big)$ has complex conjugate eigenvalues.
	\item 
	From the invariant manifold theory of Hirsch-Pugh-Shub \cite{HPS}, for every $i=1,\cdots,k-1$, the $Df$-invariant bundles
	$$
	E^s_{(i,k)}=E^s_i\oplus E^s_{i+1}\oplus\cdots\oplus E^s_k
	$$
	is uniquely integrable. We denote by $\cF^s_{(i,k)}$ the invariant foliations tangent to $E^s_{(i,k)}$.
	Since the bundle $E^s_{(1,j)}=E^s_i\oplus E^s_1\oplus\cdots\oplus E^s_j$ is also uniquely integrable for every $j=1,\cdots,k$, the bundle $E^s_{(i,j)}=E^s_i\oplus\cdots\oplus E^s_j$ is also uniquely integrable to a foliation $\cF^s_{(i,j)}$ for every $1\leq i\leq\cdots\leq j\leq k$. 
	We denote $\cF^s_i=\cF^s_{(i,i)}$. The same holds for the unstable bundle of $f$.
\end{enumerate}

\begin{lemma}\label{lem:i-linear}
	For every $f$ which is $C^1$ to $A$ and $1\leq i<k$, the bundle $E^s_i\oplus E^u$ is integrable if and only if $E^s_{(1,i)}\oplus E^u$ is integrable. In both cases, for every $i<j\leq k$, we have
	$$
	h\big(\cF^s_i\big)=\cL^s_i,\qquad 
	h\big(\cF^s_{(1,i)}\big)=\cL^s_{(1,i)}, \qquad {\it and} \qquad
	h\big(\cF^s_j\big)=\cL^s_j.
	$$
	Here $\cL^s_m$ and $\cL^s_{(1,m)}$ are linear foliations tangent to $L^s_m$ and $L^s_{(1,m)}=L^s_1\oplus\cdots\oplus L^s_m$, respectively.
\end{lemma}

\begin{proof}
	Assume that $E^s_i\oplus E^u$ is integrable and let $\cF^s_i\oplus\cF^u$ be the foliation tangent to $E^s_i\oplus E^u$. Since $h\big(\cF^s_i\oplus\cF^u\big)$ is sub-foliated by $h(\cF^u)$ which is a totally irrational linear foliation, Lemma \ref{lem:linear-foli} implies that $h\big(\cF^s_i\oplus\cF^u\big)$ is a linear foliation. According to Lemma 2.1 of \cite{GKS}, the conjugacy $h$ satisfies $h\big(\cF^s_{(j,k)}\big)=\cL^s_{(j,k)}$ for every $j=1,\cdots,k$. So we have
	$$
	h\big(\cF^s_i\big)~=~h\big(\cF^s_{(i,k)}\big)\cap h\big(\cF^s_i\oplus\cF^u\big)
	$$
	which is the intersection of two $A$-invariant linear foliations. Thus $h(\cF^s_i)$ is an $A$-invariant linear foliation. Moreover, since
	$$
	h\big(\cF^s_i\big)\subset h\big(\cF^s_{(i,k)}\big)=\cL^s_{(i,k)}
	\qquad {\rm and} \qquad
	h\big(\cF^s_{(i+1,k)}\big)=\cL^s_{(i+1,k)},
	$$
	we  have that 
	$h\big(\cF^s_i\big)=\cL^s_i$ which is the only $A$-invariant linear foliation in $\cL^s_{(i,k)}$ which is not contained in $\cL^s_{(i+1,k)}$. 
	
	Since $A$ is irreducible, the linear foliation $h\big(\cF^s_i\big)=\cL^s_i$ is minimal. The foliation $h\big(\cF^s_{(1,i)}\big)$ is sub-foliated by $h\big(\cF^s_i\big)=\cL^s_i$. So $h\big(\cF^s_{(1,i)}\big)$ is an $A$-invariant linear foliation. Moreover,
	$$
	h\big(\cF^s_{(1,i)}\big)\subset h\big(\cF^s\big)=\cL^s
	\qquad {\rm and} \qquad
	h\big(\cF^s_{(i+1,k)}\big)=\cL^s_{(i+1,k)}
	$$
	implies that $h\big(\cF^s_{(1,i)}\big)=\cL^s_{(1,i)}$, because $\cL^s_{(1,i)}$ is the only $A$-invariant linear sub-foliation of $\cL^s$ transverse to $\cL^s_{(i+1,k)}$.
	Finally, the fact that $h\big(\cF^s_{(1,i)}\big)$ and $h\big(\cF^u\big)$ are jointly integrable implies that $\cF^s_{(1,i)}$ and $\cF^u$ are jointly integrable, and their integral foliation $\cF^s_{(1,i)}\oplus\cF^u$ is tangent to $E^s_{(1,i)}\oplus E^u$.
	
	Conversely, if $E^s_{(1,i)}\oplus E^u$ is integrable then the proof is similar. The $h$-image of the integral foliation $\cF^s_{(1,i)}\oplus\cF^u$ is linear, thus 
	$$
	h\big(\cF^s_{(1,i)}\big)=h\big(\cF^s\big)\cap h\big(\cF^s_{(1,i)}\oplus\cF^u\big)
	$$
	is also linear and $A$-invariant. Therefore 
	$h\big(\cF^s_{(1,i)}\big)=\cL^s_{(1,i)}$ and
	$$
	h\big(\cF^s_i\big)~=~h\big(\cF^s_{(1,i)}\big)\cap h\big(\cF^s_{(i,k)}\big)
	~=~\cL^s_{(1,i)}\cap\cL^s_{(i,k)}~=~\cL^s_i.
	$$
	We can conclude that $\cF^s_i$ and $\cF^u$ are jointly integrable to the foliation $\cF^s_i\oplus\cF^u$ which is tangent to $E^s_i\oplus E^u$.
	
	Finally, since the $A$-invariant foliation $h\big(\cF^s_{(1,i+1)}\big)$ is sub-foliated by $h\big(\cF^s_{(1,i)}\big)=\cL^s_{(1,i)}$ and transverse to $h\big(\cF^s_{(i+2,k)}\big)=\cL^s_{(i+2,k)}$ inside $h\big(\cF^s\big)=\cL^s$, we must have $h\big(\cF^s_{(1,i+1)}\big)=\cL^s_{(1,i+1)}$. Thus
	$$
	h\big(\cF^s_{i+1}\big)~=~h\big(\cF^s_{(1,i+1)}\cap\cF^s_{(i+1,k)}\big)
	~=~\cL^s_{(1,i+1)}\cap\cL^s_{(i+1,k)}~=~\cL^s_{i+1}.
	$$
	Proceeding inductively we have $h\big(\cF^s_j\big)=\cL^s_j$ for all $j\in[i+1, k]$.
\end{proof}

The following proposition is a strengthened version of Proposition 4.1 in \cite{GS}.
We also need this proposition in next section. 

\begin{proposition}\label{prop:smallest-Lya-expo}
	Let $g\in{\rm Diff}^{1+\alpha}(\TT^d)$ be an Anosov diffeomorphism which admits an absolutely partially hyperbolic splitting $T\TT^d=E^{ss}\oplus E^{ws}\oplus E^u$:
	$$
	\|Dg|_{E^{ss}(x)}\|<\mu<m(Dg|_{E^{ws}(x)})\leq
	\|Dg|_{E^{ws}(x)}\|<1<\|Dg|_{E^u(x)}\|.
	$$
	Assume $E^{ss}\oplus E^u$ is integrable.
	
	For every periodic point $p\in{\rm Per}(g)$, let $\lambda_0(p)\in(\log\mu,0)$ be the smallest Lyapunov exponent of $p$ inside $E^{ws}(p)$, then there exists $\lambda_0\in\big(\log\mu,0\big)$, such that
	$$
	\lambda_0(p)\equiv\lambda_0, \qquad
	\forall p\in{\rm Per}(g).
	$$
\end{proposition}

\begin{proof}
	Theorem~\ref{thm:conjugate} shows that $g$ is dynamically coherent with invariant foliations $\cF^{\sigma}$ tangent to $E^{\sigma}$ for $\sigma={\it ss,ws,u}$. The linear part $A$ of $g$ also admits partially hyperbolic splitting $T\TT^d=L^{ss}\oplus L^{ws}\oplus L^u$ with same dimensions to the spiltting of $g$. Moreover, if we let $h_g:\TT^d\rightarrow\TT^d$ be the topological conjugacy $h_g\circ g=A\circ h_g$, then $h_g$ is H\"older continuous and preserves all invariant foliations:
	$$
	h_g(\cF^{\sigma})=\cL^{\sigma}, \qquad \sigma={\it ss,ws,u}.
	$$
	
	We only need to show that $\lambda_0(p)=\lambda_0(q)$ for every $p, q\in{\rm Per}(g)$. Assume there exist $p_1,q\in{\rm Per}(g)$ such that $\lambda_0(p_1)<\lambda_0(q)$. Let $\theta$ be the constant given by Lemma \ref{lem:Diophantine}. Let
	\begin{align}\label{equ:chi-delta}
	\chi=-\frac{4\cdot\log\big(\max_{x\in\TT^d}\|Dg|_{E^u(x)}\|\big) }{\theta\cdot\log\big(\max_{x\in\TT^d}\|Dg|_{E^{ss}(x)}\|\big)}>0,
	\qquad {\rm and} \qquad
	\delta=\frac{1}{2\chi+2}\big[ \lambda_0(q)-\lambda_0(p_1) \big]>0.
	\end{align}
	
	For every ergodic measure $m$ of $g$, Theorem 1.1 of \cite{WS} gives that  the smallest Lyapunov exponent $\lambda_0(m)$ of $m$ in the $Dg$-invariant bundle $E^{ws}$ can be approximated by the smallest Lyapunov exponent of periodic orbits in $E^{ws}$. So for $\delta>0$, there exists $p\in{\rm Per}(g)$, such that 
	$$
	\lambda_0(p)~\leq~
	\inf\big\{\lambda_0(m):~m~{\rm is~an~ergodic~measure~of~}g~\big\}+\delta/2.
	$$
	Here we can take $p=p_1$ if $\lambda_0(p_1)<\lambda_0(p)$. By taking the adapted metric, we can assume 
	\begin{align}\label{equ:0-p}
	\lambda_0(p)~\leq~
	\min\left\{\lambda_0(p_1),
	~\log\big[ \min_{x\in\TT^d}m\big(Dg|_{E^{ws}(x)}\big)\big]+\delta
	\right\}
	~\leq~\lambda_0(q)-(2\chi+2)\delta.
	\end{align}
	
	\vskip3mm
	
	Let $\kappa\in\NN$ be the minimal common period of $p,q$, i.e. the smallest integer satisfies $g^{\kappa}(p)=p$ and $g^{\kappa}(q)=q$. Since $\lambda_0(q)$ is the smallest Lyapunov exponent of $q$ in $E^{ws}$, there exists $\eta>0$, such that for every $x$ if it satisfies $d(x,q)\leq\eta$, then
	\begin{align}\label{equ:q}
	m\big(Dg^{\kappa}|_{E^{ws}(x)}\big) ~\geq~ 
	\exp\big[\kappa\left(\lambda_0(q)-\delta\right)\big].
	\end{align}
	
	Let $\cF_0(p)\subseteq\cF^{ws}(p)$ be the Pesin stable manifold of $p$ associated to the Lyapunov exponent $\lambda_0(p)$ in $\cF^{ws}(p)$, that is, $x\in\cF_0(p)$ if and only if for every $\epsilon>0$,
	$$
	\lim_{n\rightarrow+\infty}\frac{1}{n}\log d\big(g^n(x),g^n(y)\big)
	<\lambda_0+\epsilon.
	$$
	Since $\|Dg|_{E^{ws}}\|$ is uniformly contracting, we have the following claim.
	
	\begin{claim}\label{clm:eta1}
		There exists a constant $\eta_1\in(0,\eta/10)$ satisfying the following properties:
		\begin{itemize}
			\item  If $x\in\cF_{0,{\it loc}}(p)\subseteq\cF^{ws}_{\it loc}(p)$ and $d_{\cF^{ws}}(p,x)\leq\eta_1$, then
			$$
			d_{\cF^{ws}}\big(g^{n\kappa}(p),g^{n\kappa}(x)\big)
			~\leq~
			\exp\big[n\kappa\cdot\big(\lambda_0(p)+\delta\big)\big]
			\cdot d_{\cF^{ws}}(p,x),
			\qquad \forall n\geq0.
			$$
			\item If $y\in\cF^{ws}(x)$ and $d_{\cF^{ws}}(x,y)\leq\eta_1$, then
			$$
			d\big(g^j(x),g^j(y)\big)\leq\frac{\eta}{4}, \qquad\forall j\geq0.
			$$
		\end{itemize}
	\end{claim}

	Applying Lemma \ref{lem:Diophantine} to $p$ and $q$ we have that there exist $z_n\in\cF^{ss}(p)$ and $w_n\in\cF^u(z_n)\cap\cF^{ws}(q)$, such that 
	$$
	d_{\cF^u}(w_n,z_n)~\leq~C\cdot L_n^{-\theta},
	\qquad {\it and} \qquad
	d_{\cF^{ws}}(w_n,q)~\leq~C\cdot L_n^{-\theta},
	$$ 
	where $L_n=d_{\cF^{ss}}(p,z_n)\rightarrow+\infty$ as $n\rightarrow\infty$. For $n$ large enough, we can assume that $d_{\cF^{ws}}(w_n,q)<\eta_1$.
	
	\begin{claim}\label{clm:N}
		Let $\eta_1$ be the constant from Claim~\ref{clm:eta1}. There exist a point $x\in\cF_{0,{\it loc}}(p)\subseteq\cF^{ws}_{\it loc}(p)$ and a constant $K>1$, such that if   
		$x_n$ is the unique intersecting point of $\cF^{ss}(x)$ and $\cF^{ws}(z_n)$:
		$$
		x_n~\in~\cF^{ss}(x)\cap\cF^{ws}(z_n)
		~\subset~\cF^s(p),
		$$
		then
		$$
		d_{\cF^{ws}}(x,p)\leq\eta_1, \qquad {\it and} \qquad
		\frac{1}{K}d_{\cF^{ws}}(x,p) \leq d_{\cF^{ws}}(x_n,z_n)\leq\eta_1.
		$$
	\end{claim}
	
	\begin{proof}[Proof of the Claim]
		Since $\cF^{ss}$ and $\cF^{ws}$ have global product structure inside the stable leaf $\cF^s(p)$, we have that $\cF^{ss}(x)$ intersects $\cF^{ws}(z_n)$ at a unique point $x_n$. Moreover, the conjugacy $h_g$ preserves the strong stable foliation $h_g(\cF^{ss})=\cL^{ss}$ which is a linear foliation. This implies that for any $x\in\cF_0(p)\subseteq\cF^{ws}(p)$ and corresponding $x_n\in\cF^{ss}(x)\cap\cF^{ws}(z_n)$ with $z_n\in\cF^{ss}(p)$, must satisfy
		$$
		d_{\cL^{ws}}\big(h_g(x),h_g(p)\big)~=~
		d_{\cL^{ws}}\big(h_g(x_n),h_g(z_n)\big),
		\qquad\forall n\in\NN.
		$$
		
		Since both $h_g$ and $h_g^{-1}$ are uniformly continuous on $\TT^d$, we can take $x\in\cF_{0,{\it loc}}(p)$ with $d_{\cF^{ws}}(p,x)<\eta_1$ small enough, such that for every pair of points $r_1,r_2\in\TT^d$ with 
		$$
		r_2\in\cL^{ws}(r_1),
		\qquad {\rm and} \qquad
		d_{\cL^{ws}}(r_1,r_2)=d_{\cL^{ws}}\big(h_g(x),h_g(p)\big),
		$$
		we have
		$$
		\epsilon_0\leq d_{\cF^{ws}}\big(h_g^{-1}(r_1),h_g^{-1}(r_2)\big)\leq\eta_1.
		$$
		Here the constant $\epsilon_0$ is independent of the choice of $r_1,r_2$. Finally, we let $K=d_{\cF^{ws}}(x,p)/\epsilon_0$.
	\end{proof}
	Let $m_n\in\NN$ be the smallest integer satisfying 
	$
	d_{\cF^{ss}}\big(g^{m_n\kappa}(z_n),p\big)\leq1$. Then
	$$
	m_n\leq\frac{\log L_n}{-\kappa\log\big(\max_{x\in\TT^d}\|Dg|_{E^{ss}(x)}\|\big)}+1.
	$$
	Let $k_n$ be the largest integer satisfying 
	$d\big(g^{\kappa j}(w_n),g^{\kappa j}(z_n)\big)\leq\eta/4$ for every $0\leq j\leq k_n-1$.
	Then there exists $N>0$, such that for every $n\geq\NN$,
	$$
	k_n~\geq~
	\frac{\theta\log L_n+\log(\eta/4)}{ \kappa\log\big(\max_{x\in\TT^d}\|Dg|_{E^u(x)}\|\big)}
	~\geq~\frac{\theta\log L_n}{ 2\kappa\log\big(\max_{x\in\TT^d}\|Dg|_{E^u(x)}\|\big)}.
	$$
	Using the definition of $\chi$ in (\ref{equ:chi-delta}), we have
	\begin{align}\label{equ:ratio}
	\frac{k_n}{m_n}~\geq~
	\frac{-\theta\cdot \log\big(\max_{x\in\TT^d}\|Dg|_{E^{ss}(x)}\|\big)}{ 4\cdot\log\big(\max_{x\in\TT^d}\|Dg|_{E^u(x)}\|\big)}
	~=~\frac{1}{\chi}.
	\end{align}
	
	\vskip2mm
	
	\noindent{\bf Case 1.} There exist infinitely many $n$, such that $k_n\geq m_n$.
	
	In this case, we have 
	$$
	d\big(g^{\kappa j}(z_n),q\big)~\leq~ 
	d_{\cF^u}\big(g^{\kappa j}(z_n),g^{\kappa j}(w_n)\big)
	~+~d_{\cF^{ws}}\big(g^{\kappa j}(w_n),q\big)~\leq~
	\frac{\eta}{4}+\frac{\eta}{4}~=~\frac{\eta}{2}
	\qquad \forall 0\leq j\leq m_n.
	$$
	On the other hand, Claim \ref{clm:N} shows 
	$d_{\cF^{ws}}\big(g^j(x_n),g^j(z_n)\big)<\eta/4$ for every $j\geq0$. Thus we have
	\begin{align*}
	\frac{
		d_{\cF^{ws}}\big(g^{m_n\kappa}(x_n),g^{m_n\kappa}(z_n)\big) }{d_{\cF^{ws}}\big(g^{m_n\kappa}(x),p\big)} 
	~&\geq~ \frac{\exp\big[\kappa m_n(\lambda_0(q)-\delta)\big] 
	}{\exp\big[\kappa m_n(\lambda_0(p)+\delta)\big] }\cdot
	\frac{d_{\cF^{ws}}(x_n,z_n)}{d_{\cF^{ws}}(x,p)} \\
	&\geq~\frac{\exp\big[\kappa m_n(\lambda_0(p)+\delta+2\chi\delta)\big]}{
		\exp\big[\kappa m_n(\lambda_0(p)+\delta)\big]}\cdot\frac{1}{K} \\
	&\longrightarrow\infty\qquad(n\rightarrow\infty).
	\end{align*}
	If we denote by $\H^{ss}_{p,g^{m_n\kappa}(z_n)}:\cF^{ws}(p)\rightarrow\cF^{ws}(g^{m_n\kappa}(z_n))$ the holonomy map induced by $\cF^{ss}$ inside the leaf $\cF^s(p)$, then we have
	$$
	\H^{ss}_{p,g^{m_n\kappa}(z_n)}(p)=g^{m_n\kappa}(z_n)
	\qquad {\rm and} \qquad
	\H^{ss}_{p,g^{m_n\kappa}(z_n)}(x)=g^{m_n\kappa}(x_n).
	$$
	But this is absurd. Indeed because $d_{\cF^{ss}}\big(g^{m_n\kappa}(z_n),p\big)\leq1$, this contradicts to the fact that $\cF^{ss}$ is $C^1$-smooth in $\cF^s(p)$ (see e.g., Corollary of \cite[Theorem B]{PSW} and \cite[Theorem 2.2]{Bro}).
	
	\vskip2mm
	
	\noindent{\bf Case 2.} We can now assume $k_n<m_n$ for all sufficiently large $n$.

	In this case, we have $\chi^{-1}\leq k_n/m_n<1$. So for all sufficiently large $n$, the estimation (\ref{equ:q}) and Claim \ref{clm:eta1} imply
	\begin{align*}
	\frac{
		d_{\cF^{ws}}\big(g^{m_n\kappa}(x_n),g^{m_n\kappa}(z_n)\big)} {d_{\cF^{ws}}\big(g^{m_n\kappa}(x),p\big)}
	~&\geq~ 
	\frac{\big[ \min_{x\in\TT^d}m\big(Dg|_{E^{ws}(x)}\big) \big]^{(m_n-k_n)\kappa}\cdot 
		\exp\big[k_n\kappa\big(\lambda_0(q)-\delta\big)\big]
		\cdot d_{\cF^{ws}}(z_n,x_n)}{\exp\big[m_n\kappa\cdot
		\big(\lambda_0(p)+\delta\big)\big]\cdot d_{\cF^{ws}}(p,x)
	} \\
	&\geq~\frac{ \exp\big[(m_n-k_n)\kappa
		\big(\lambda_0(p)-\delta\big)\big] \cdot \exp\big[k_n\kappa\big(\lambda_0(q)-\delta\big)\big]}{ \exp\big[m_n\kappa\cdot
		\big(\lambda_0(p)+\delta\big)\big] }
	\cdot\frac{d_{\cF^{ws}}(z_n,x_n)}{d_{\cF^{ws}}(p,x)} \\
	&\geq~\frac{ 
		\exp\left\{m_n\kappa\cdot\big[[1-\chi^{-1}](\lambda_0(p)-\delta)+ \chi^{-1}(\lambda_0(p)+\delta+2\chi\delta)\big]\right\}}{\exp\big[m_n\kappa\cdot
		\big(\lambda_0(p)+\delta\big)\big] }\cdot\frac{1}{K} \\
	&\geq~\frac{ \exp\big[m_n\kappa\cdot
		\big(\lambda_0(p)+\delta+2\chi^{-1}\delta\big)\big] }{ \exp\big[m_n\kappa\cdot
		\big(\lambda_0(p)+\delta\big)\big]\cdot K  }\\
	&\longrightarrow \infty \qquad (n\rightarrow+\infty).
	\end{align*}
	This is absurd. We finishes the proof of this proposition.
\end{proof}

\begin{corollary}	
	If $E^s_i\oplus E^u$ is integrable, then there exists a constant $\lambda_0\in(\mu,\lambda_0)$, such that the smallest Lyapunov exponent of $f$ inside $E^s_{i+1}(p)$ is equal to $\lambda_0$ for every $p\in{\rm Per}(f)$. 
	
	In particular, if ${\rm dim}E^s_{i+1}=1$, then the Lyapunov exponents of $f$ in  $E^s_{i+1}$ for all ergodic measures are equal to $\lambda_0$.
\end{corollary}

\begin{proof}
	We apply Proposition \ref{prop:smallest-Lya-expo} to $f$ and the absolutely partially hyperbolic splitting $T\TT^d=E^s_{(1,i)}\oplus E^s_{(i+1,k)}\oplus E^u$. Lemma \ref{lem:i-linear} says that integrability of $E^s_i\oplus E^u$ implies integrability of $E^s_{(1,i)}\oplus E^u$. Since the smallest Lyapunov exponent of $f$ along $E^s_{(i+1,k)}$ is also the smallest Lyapunov exponent along $E^s_{i+1}(p)$, we obtain the corollary.
\end{proof}

Now we need to deal with the case when ${\rm dim}E^s_{i+1}=2$. Recall that we have assumed that $f$ has a fixed point $p_f$ close to $0$ such that the derivative $\big(Df|_{E^s_{i+1}(p_f)}\big)$ has complex eigenvalues. This implies that
$$
\log\left|{\rm det}\big(Df|_{E^s_{i+1}(p_f)}\big)\right|~=~2\lambda_0.
$$

We need to introduce some notation. For any two points $x',y'\in\TT^d$, we denote by $T_{x',y'}:\TT^d\rightarrow\TT^d$ the linear translation on $\TT^d$:
$$
T_{x',y'}(z')=z'+(y'-x').
$$	
Here $+$ and $-$ refer to the standard $\TT^d$ action on itself.  Then $T_{x',y'}(x')=y'$ and it preserves all linear foliations on $\TT^d$.

\begin{proposition}\label{prop:det}
	If $E^s_i\oplus E^u$ is integrable, then for every $p\in{\rm Per}(f)$ with period $\pi(p)$ we have
	$$
	\frac{1}{\pi(p)}\sum_{k=0}^{\pi(p)-1}
	\log\left|{\rm det}\big(Df|_{E^s_{i+1}(f^k(p))}\big)\right|
	~=~2\lambda_0.
	$$
	This implies that for every $p\in{\rm Per}(f)$, both Lyapunov exponents of $p$ in $E^s_{i+1}$ are equal to $\lambda_0$.
\end{proposition}

\begin{proof}
	Define $\varphi:\TT^d\rightarrow\RR$ as follows
	$$
	\varphi(x)=\log\left|{\rm det}\big(Df|_{E^s_{i+1}(x)}\big)\right|,
	$$
	This function is H\"older continuous. Moreover, for every ergodic measure $m$ of $f$, the integral $\int\varphi{\rm d}m$ is equal to the sum of two Lyapunov exponents of $m$ along $E^s_{i+1}$. For every periodic point $p$ with period $\pi(p)$, we denote $m_p$ the invariant measure supported on the orbit of $p$ and let 
	$$
	\overline{\varphi}(p)~=~\int\varphi{\rm d}m_p
	~=~\frac{1}{\pi(p)}\sum_{j=0}^{\pi(p)-1}\varphi\circ f^j(p).
	$$
	
	Assume there exist $p_1,q\in{\rm Per}(f)$ such that $\overline{\varphi}(p_1)<\overline{\varphi}(q)$. Let $\theta$ be the constant given by Lemma~\ref{lem:Diophantine}. Let
	\begin{align}\label{equ:det-chi-delta}
	\chi=-\frac{4\cdot\log\big(\max_{x\in\TT^d}\|Df|_{E^u(x)}\|\big) }{\theta\cdot\log\big(\max_{x\in\TT^d}\|Dg|_{E^s_{(1,i)}(x)}\|\big)}>0,
	\qquad {\rm and} \qquad
	\delta=\frac{1}{2\chi+2}\big[ \overline{\varphi}(q)-\overline{\varphi}(p_1) \big]>0.
	\end{align}
	
	Since the periodic measures are dense in the space of ergodic measures of $f$,  there exists $p\in{\rm Per}(f)$, such that 
	$$
	\overline{\varphi}(p)~\leq~
	\inf\big\{\int\varphi{\rm d}m:~m~{\rm is~an~ergodic~measure~of~}f~\big\}+\delta/2.
	$$
	for any $\delta>0$.
	Here we can take $p=p_1$ when  $\overline{\varphi}(p_1)<\overline{\varphi}(p)$.
	There exists $N>0$, such that for every $n\geq N$ we have
	\begin{align}\label{equ:det-p}
	\overline{\varphi}(p)~\leq~
	\min\left\{\overline{\varphi}(p_1),
	~\min_{x\in\TT^d}\big[\frac{1}{n}\sum_{j=0}^{n-1}\varphi\circ f^j(x)\big]+\delta
	\right\}
	~\leq~\overline{\varphi}(q)-(2\chi+2)\delta.
	\end{align}
	
	Let $\kappa\in\NN$ be the minimal common period of $p$ and $q$. Since $\varphi$ is continuous, there exists $\eta>0$, such that for every $x,y\in\TT^d$ if $d(x,y)\leq\eta$, then
	$$
	|\varphi\circ f^j(x)-\varphi\circ f^j(y)|<\delta, 
	\qquad\forall 0\leq j\leq\kappa-1.
	$$

	\begin{claim}\label{clm:det-D}
		Let $p'=h(p)$.
		There exist a 2-dimensional disk $\,\,\,\cD(p)\subset\cF^s_{i+1}(p)$ which contains $p$ and a constant $A_0>0$, such that for every $x\in\TT^d$ and $x'=h(x)$, the disk
		$$
		\cD(x)~=~h^{-1}\circ T_{p',x'}\circ h\big(\cD(p)\big)
		~\subset~\cF^s_{i+1}(x)
		$$
		satisfies
		$$
		{\rm Area}(\cD(x))>A_0, \qquad {\rm and} \qquad
		{\rm diam}\big(f^n(\cD(x))\big)~\leq~\frac{\eta}{2},
		\qquad \forall n\geq 0.
		$$
		
		In particular, if we take $x=p$, then
		$$
		{\rm Area}\big(f^{\kappa n}(\cD(p))\big)~\leq~
		\exp\big[ \kappa n\cdot(\overline{\varphi}(p)+\delta)\big]
		\cdot{\rm Area}(\cD(p)),
		\qquad \forall n\geq0.
		$$
	\end{claim}
	
	\begin{proof}[Proof of the Claim]
		Lemma \ref{lem:i-linear} shows that $h\big(\cF^s_{i+1}\big)=\cL^s_{i+1}$ which is linear. So $\cD(p)\subset\cF^s_{i+1}(p)$ implies $\cD(x)\subset\cF^s_{i+1}(x)$. Since $f$ is uniformly contracting along $\cF^s_{i+1}$, we only need to take $\cD(p)$ small enough such that 
		${\rm diam}\big(f^n(\cD(x))\big)\leq\eta/2$ for every $x\in\TT^d$ and every $n\geq0$. The uniform continuity of $h$ and $h^{-1}$ ensures the existence of a constant $A_0$ satisfying the claim.
		
		Finally, for every $z\in\cD(p)$, we have $d\big(f^n(z),f^n(p)\big)<\eta/2$ for every $n\geq0$. Thus 
		$$
		|\phi\circ f^n(z)-\phi\circ f^n(p)|<\delta, \qquad \forall n\geq0.
		$$
		From the definition of $\varphi(z)$ and $\overline{\varphi}(p)$, we have
		$$
		{\rm Area}\big(f^{\kappa n}(\cD(p))\big)~\leq~
		\exp\big[ \kappa n\cdot(\overline{\varphi}(p)+\delta)\big]
		\cdot{\rm Area}(\cD(p)),
		\qquad \forall n\geq0.
		$$
	\end{proof}
	
	Applying Lemma \ref{lem:Diophantine} to $p$ and $q$ with respect to the splitting $T\TT^d=E^s_{(1,i)}\oplus E^s_{(i+1,k)}\oplus E^u$, we obtain points $z_n\in\cF^s_{(1,i)}(p)$ and $w_n\in\cF^s_{(i+1,k)}(z_n)\cap\cF^u(q)$, such that 
	$$
	d_{\cF^s_{(i+1,k)}}(w_n,z_n)~\leq~C\cdot L_n^{-\theta},
	\qquad {\it and} \qquad
	d_{\cF^u}(w_n,q)~\leq~C\cdot L_n^{-\theta},
	$$ 
	where $L_n=d_{\cF^s_{(1,i)}}(p,z_n)\rightarrow+\infty$ as $n\rightarrow\infty$.
	
	Since $f$ uniformly contracts $\cF^s_{(i+1,k)}$, we can assume that
	$$
	d_{\cF^s_{(i+1,k)}}(f^j(w_n),f^j(z_n))<\eta/4, \qquad \forall j\geq0.
	$$
	for sufficiently large $n$.
	Let $m_n\in\NN$ be the smallest integer such that 
	$
	d_{\cF^s_{(1,i)}}\big(f^{m_n\kappa}(z_n),p\big)\leq1,$
	then
	$$
	m_n\leq\frac{\log L_n}{-\kappa\log\big(\max_{x\in\TT^d}\|Df|_{E^s_{(1,i)}(x)}\|\big)}+1.
	$$
	Let $k_n$ be the largest integer satisfying 
	$d_{\cF^u}\big(f^{\kappa j}(w_n),q\big)\leq\eta/4$ for all $j\in [0, k_n-1]$.
	There exists $N>0$, such that for all $n\geq N$,
	$$
	k_n~\geq~
	\frac{\theta\log L_n+\log(\eta/4)}{ \kappa\log\big(\max_{x\in\TT^d}\|Df|_{E^u(x)}\|\big)}
	~\geq~\frac{\theta\log L_n}{ 2\kappa\log\big(\max_{x\in\TT^d}\|Df|_{E^u(x)}\|\big)}.
	$$
	According to the definition of $\chi$ in (\ref{equ:det-chi-delta}), we have
	\begin{align}\label{equ:det-ratio}
	\frac{k_n}{m_n}~\geq~
	-\frac{\theta\cdot \log\big(\max_{x\in\TT^d}\|Dg|_{E^s_{(1,i)}(x)}\|\big)}{ 4\cdot\log\big(\max_{x\in\TT^d}\|Df|_{E^u(x)}\|\big)}
	~=~\frac{1}{\chi}.
	\end{align}
	
	\vskip2mm
	
	\noindent{\bf Case 1.} There exist infinitely many $n$, such that $k_n\geq m_n$.
	
	In this case, we have $d_{\cF^u}\big(f^{\kappa j}(w_n),q\big)\leq\eta/4$ for all $j\in [0, m_n-1]$. Thus
	$$
	d\big(f^{\kappa j}(z_n),q\big)~\leq~ 
	d_{\cF^s_{(i+1,k)}}\big(f^{\kappa j}(z_n),f^{\kappa j}(w_n)\big)~+~
	d_{\cF^u}\big(f^{\kappa j}(w_n),q\big)~\leq \eta/2,
	\qquad\forall 0\leq j\leq m_n-1.
	$$
	Let $\cD(z_n)$ be defined as Claim \ref{clm:det-D}, then $\cD(z_n)=\H^{ss}_{p,z_n}(\cD(p))$, where $\H^{ss}_{p,z_n}:\cF^s_{i+1}(p)\rightarrow\cF^s_{i+1}(z_n)$ is the holonomy map  induced by $\cF^s_{(1,i)}$ inside the leaf $\cF^s_{(1,i+1)}(p)$.
	This implies that for every $z\in\cD(z_n)$, we have $d\big(f^j(z),f^j(q)\big)\leq\eta$ for $j=0,\cdots,\kappa m_n-1$.
	
	From the estimation (\ref{equ:det-chi-delta}), we have
	\begin{align}\label{equ:det-infty}
	\frac{{\rm Area}\big(f^{\kappa m_n}(\cD(z_n))\big)}{ {\rm Area}\big(f^{\kappa m_n}(\cD(p))\big)}~&\geq~
	\frac{\exp\big[ \kappa m_n\cdot(\overline{\varphi}(q)-\delta)\big] }{\exp\big[ \kappa m_n\cdot(\overline{\varphi}(p)+\delta)\big]}
	\cdot\frac{A_0}{{\rm Area}(\cD(p))} \notag\\
	&\geq~\frac{\exp\big[ \kappa m_n\cdot(\overline{\varphi}(p)+\delta+2\chi\delta)\big] }{\exp\big[ \kappa m_n\cdot(\overline{\varphi}(p)+\delta)\big]}
	\cdot\frac{A_0}{{\rm Area}(\cD(p))}
	~\longrightarrow~\infty\qquad (n\rightarrow\infty).
	\end{align}
	Let
	$\H^{ss}_{p,f^{\kappa m_n}(z_n)}:\cF^s_{i+1}(q)\rightarrow\cF^s_{i+1}\big(f^{\kappa m_n}(z_n)\big)$ be the holonomy map induced by $\cF^s_{(1,i)}$ restricted to $\cF^s_{(1,i+1)}(p)$, then 
	$$
	\H^{ss}_{p,f^{\kappa m_n}(y_n)}\big(f^{\kappa m_n}(\cD(p))\big)
	~=~f^{\kappa m_n}\big(f^{\kappa m_n}(\cD(z_n))\big).
	$$
	So equation (\ref{equ:det-infty}) is absurd since $d_{\cF^s_{(1,i)}}\big(p,f^{\kappa m_n}(z_n)\big)\leq1$ and $\cF^s_{(1,i)}$ is uniformly $C^1$-smooth inside the leaf $\cF^s_{(1,i+1)}(p)$, see \cite[Theorem 2.2]{Bro}.
	
	\vskip2mm
	
	\noindent {\bf Case 2.} We can assume that $k_n<m_n$ for all sufficiently large $n$. 
	
	In fact, we can assume that for all sufficiently large $n$, $m_n-k_n\geq N$, where $N$ is such that the equation (\ref{equ:det-p}) holds. (Indeed, if that's not the case, since $N$ is fixed we have infinitely many $n$ with $m_n-k_n<N$ and the same estimates as in Case~1 gives a contradiction.)
	Recall that $\overline{\varphi}(q)-\overline{\varphi}(p)=(2\chi+2)\cdot\delta$.
	Since $k_n/m_n>1/\chi>0$, from the definition of $m_n,k_n$ and equation (\ref{equ:det-p}), we have
	\begin{align*}
	\frac{{\rm Area}\big(f^{\kappa m_n}(\cD(z_n))\big)}{ {\rm Area}\big(f^{\kappa m_n}(\cD(p))\big)}~&\geq~
	\frac{
		\exp\big[\kappa (m_n-k_n)\cdot(\overline{\varphi}(p)-\delta) \big]\cdot
		\exp\big[\kappa k_n\cdot(\overline{\varphi}(q)-\delta)\big] 
	}{\exp\big[ \kappa m_n\cdot(\overline{\varphi}(p)+\delta)\big]}
	\cdot\frac{A_0}{{\rm Area}(\cD(p))} \\
	&\geq~\frac{\exp\big[ \kappa m_n\cdot\big((1-\chi^{-1})\cdot(\overline{\varphi}(p)-\delta)+ \chi^{-1}(\overline{\varphi}(p)-\delta)\big)\big] }{\exp\big[ \kappa m_n\cdot(\overline{\varphi}(p)+\delta)\big]}
	\cdot\frac{A_0}{{\rm Area}(\cD(p))} \\
	&\geq~\frac{\exp\big[ \kappa m_n\cdot(\overline{\varphi}(p)+\delta+2\chi^{-1}\delta)\big] }{\exp\big[ \kappa m_n\cdot(\overline{\varphi}(p)+\delta)\big]}
	\cdot\frac{A_0}{{\rm Area}(\cD(p))} \\
	&\longrightarrow~\infty\qquad (n\rightarrow\infty).
	\end{align*}
	This is absurd. Hence the proof is complete.
\end{proof}

This proposition shows that all Lyapunov exponents of $f$ along $E^s_{i+1}$ are equal to $\lambda_0$. We need to show that $\lambda_0$ is equal to $\lambda^s_{i+1}$, which is the Lyapunov exponents of $A$ in $L^s_{i+1}$. We prove a more general lemma that will be needed in next section.

\begin{lemma}\label{lem:k0=ws}
	Let $g\in{\rm Diff}^1(\TT^d)$ be an Anosov diffeomorphism on $\TT^d$ which is conjugate to a linear automorphism $A\in{\rm GL}(d,\ZZ)$. Assume $g$ admits a dominated splitting $T\TT^d=E^{ss}\oplus E_0\oplus F_0\oplus E^u$ where the stable bundle $E^s=E^{ss}\oplus E_0\oplus F_0$.
	Also assume the following
	\begin{itemize}
		\item the bundle $E_0$ is integrable to a foliation $\cF_0$ and the conjugacy $h_g$ maps $\cF_0$ to a linear $A$-invariant foliation $\cL_0$;
		\item for every $p\in{\rm Per}(g)$, the Lyapunov exponents of $g$ along $E_0$ are all equal to $\lambda_0$,
	\end{itemize}
	then all the Lyapunov exponents of $A$ along $L_0=T\cL_0$ are equal to $\lambda_0$.
	
	%In particular, when the bundle $F_0$ is trivial and $E_0=E^{ws}$, then all the Lyapunov exponents of $A$ in $L^{ws}$ are equal to $\lambda_0$.
\end{lemma}

\begin{proof}
	Since $g$ is Anosov and the bundle $E_0$ is H\"older continuous, we have that for every $0<\e\ll|\lambda_0|$ there exists an adapted metric on $\TT^d$, such that
	\begin{align}\label{equ:norm}
	\lambda_0-\e~<~
	\min_{x\in\TT^d}\left\{\log m\big(Dg|_{E_0(x)}\big)\right\}
	~\leq~\max_{x\in\TT^d}\left\{\log\|Dg|_{E_0(x)}\|\right\}
	~<~\lambda_0+\e.
	\end{align}
	
	Let $G:\RR^d\rightarrow\RR^d$ be a lift of $g$, and let $\tilde{\cF}_0$ be the lift of $\cF_0$ to $\RR^d$. Denote by $H:\RR^d\rightarrow\RR^d$ the conjugation of $G$: $\,\,H\circ G=A\circ H$. Then there exists a constant $C>0$, such that $\|H-{\rm Id}\|_{C^0}<C$. By Theorem \ref{thm:conjugate} we have that $H\big(\tilde{\cF}_0\big)=\tilde{\cL}_0$. In particular, the foliation  $\tilde{\cF}_0$ is quasi-isometric, i. e. there exist $a_0,b_0>0$, such that for every $x\in\RR^d$ and $y\in\tilde{\cF}_0(x)$, we have
	$$
	d_{\tilde{\cF}_0}(x,y)\leq a_0\cdot d(x,y)+b_0.
	$$ 
	
	For every  $x'\in\RR^d$ and $y'\in\tilde{\cL}_0(x)$, let $x=H^{-1}(x')$ and $y=H^{-1}(y')$. Then $y\in\tilde{\cF}_0(x)$. From estimation (\ref{equ:norm}), we have:
	\begin{align*}
	\lim_{n\rightarrow+\infty}
	\frac{d\big(A^{-n}(x'),A^{-n}(y')\big)}{\exp\big[-n\cdot(\lambda_0-\e)\big]}
	~&=\lim_{n\rightarrow+\infty}
	\frac{d\big(H\circ G^{-n}(x), H\circ G^{-n}(y)\big)}{\exp\big[-n\cdot(\lambda_0-\e)\big]} \\
	&\leq \lim_{n\rightarrow+\infty}
	\frac{d\big(G^{-n}(x), G^{-n}(y)\big)+2C}{\exp\big[-n\cdot(\lambda_0-\e)\big]} \\
	&\leq \lim_{n\rightarrow+\infty}
	\frac{d_{\tilde{\cF}_0}\big(G^{-n}(x), G^{-n}(y)\big)+2C }{\exp\big[-n\cdot(\lambda_0-\e)\big]} \\
	&\leq \lim_{n\rightarrow+\infty}
	\frac{\big[\min_{x\in\TT^d}m\big(Dg|_{E_0(x)}\big)\big]^{-n}}{ \exp\big[-n\cdot(\lambda_0-\e)\big]}
	\cdot d_{\tilde{\cF}_0}(x,y) +
	\frac{2C}{\exp\big[-n\cdot(\lambda_0-\e)\big]} \\
	&=0.
	\end{align*}
	Taking $\e\rightarrow0$ we obtain that every Lyapunov exponent of $A$ in $L^{ws}$ is no less than $\lambda_0$. 
	
	\vskip2mm
	
	On the other hand, let $\lambda_0'$ be the largest Lyapunov exponent of $A$ along $L^{ws}$. So $A$ has an eigenvalue in $L^{ws}$ whose absolute value is $\exp(\lambda_0')$. 
	
	Assume that $\lambda_0<\lambda_0'(<0)$, then we will obtain  a contradiction. Let $\e<(\lambda_0'-\lambda_0)/2$ in inequality (\ref{equ:norm}).
	Let  $L'(0)\subseteq L_0(0)\subset\RR^d$ be the eigenspace of $A$ generated by eigenvectors whose corresponding eigenvalues have absolute value $\exp(\lambda_0')$. We take a sequence of points $x_k'\in L'(0)\subset\RR^d$ satisfying
	\begin{itemize}
		\item $\|x_k'\|=\exp(-k\lambda_0')\rightarrow+\infty$ as $k\rightarrow+\infty$;
		\item $\|A^n(x_k')\|=\exp\big((n-k)\cdot\lambda_0'\big)$ for every $n\in\NN$.
	\end{itemize}
	The point $p=H^{-1}(0)$ is a fixed point of $G$, and $x_k=H^{-1}(x_k')$, then $x_k\in\tilde{\cF}_0(p)$. Since $\|H-{\rm Id}\|_{C^0}<C$ and the foliation $\tilde{\cF}_0$ is quasi-isometric, we have
	$$
	d_{\tilde{\cF}_0}(p,x_k)~\leq~ a_0\cdot d(p,x_k)+b_0 
	~\leq~a_0\exp(-k\lambda_0')+b_0+2C.
	$$
	
	Let $n_k=\left[k\cdot\frac{\lambda_0'}{\lambda_0+\e}\right]+1\in\NN$, where $[\cdot]$ is the integer floor function.  We have
	\begin{align}
	d_{\tilde{\cF}_0}\big(G^{n_k}(p),G^{n_k}(x_k)\big)
	~&\leq~ 
	\big[\max_{x\in\TT^d}\|Dg|_{E_0(x)}\|\big]^{n_k}\cdot d_{\tilde{\cF}_0}(p,x_k)  \notag \\ 
	&\leq~\exp\big[n_k\cdot(\lambda_0+\e)\big]\cdot
	\big[ a_0\exp(-k\lambda_0')+b_0+2C \big] \notag \\
	&\leq~ a_0+(b_0+2C)\exp\big[n_k\cdot(\lambda_0+\e)\big] \notag
	\end{align}
	Since $\lambda_0+\e<0$ and $p=G^{n_k}(p)$, this implies that $\|G^{n_k}(x_k)\|$ are uniformly bounded.
	However, since $0<\lambda_0'/(\lambda_0+\e)<1$, $(n_k-k)\rightarrow-\infty$ as $k\rightarrow+\infty$.  Thus the sequence $\|A^{n_k}(x_k')\|\rightarrow+\infty$ as $k\rightarrow\infty$, This contradicts to the fact that
	$$
	H\circ G^{n_k}(x_k)~=~H\circ G^{n_k}\circ H^{-1}(x_k')~=~A^{n_k}(x_k').
	$$
	This proves that every Lyapunov exponent of $A$ in $L_0$ is no larger than $\lambda_0$. Hence the proof is complete.
\end{proof}

Now we can prove Theorem \ref{thm:local} by induction.

\begin{proof}[Proof of Theorem \ref{thm:local}]
	Firstly, Lemma \ref{lem:i-linear} shows that the first and second items of Theorem \ref{thm:local} are equivalent. 
	
	From Proposition \ref{prop:smallest-Lya-expo}, Proposition \ref{prop:det} and Lemma \ref{lem:k0=ws}, the first item implies all the Lyapunov exponents of $f$ in $E^s_{i+1}$ are equal to the Lyapunov exponent of $A$ in $L^s_{i+1}$, which is $\lambda(L^s_{i+1},A)$. Moreover, Lemma \ref{lem:i-linear} shows that $h\big(\cF^s_{i+1}\big)=\cL^s_{i+1}$ which implies $E^s_{i+1}\oplus E^u$ is integrable. So we can apply the induction which shows for every $i+1\leq j\leq k$, all the Lyapunov exponents of $f$ in $E^s_j$ are equal to $\lambda(E^s_j,A)$. This proves the third item.
	
	If item 3 holds, then \cite[Proposition 2.4]{GKS1} implies the conjugacy $h$ is $C^{1+{\rm H\"older}}$-smooth along $\cF^s_j$ for every $i=i+1\cdots,k$. 
	
	Finally, if item 4 holds, then Journ\'e's theorem \cite{J}\footnote{We remark that the proof of Journ\'e theorem in $C^{1+{\rm H\"older}}$ regularity is a simple calculus exercise, since mixed derivatives don't have to be recovered.} implies $f$ is $C^{1+{\rm H\"older}}$-smooth along $\cF^s_{(i+1,k)}$. Proposition 2.2 of \cite{GKS} shows that $h\big(\cF^s_i\big)=\cL^s_i$. Thus $E^s_i\oplus E^u$ is integrable, i.e., the second item holds. The proof is complete.
\end{proof}

\section{Global Lyapunov exponents rigidity}

In this section, we show that if the conjugacy $h$ preserves the strong stable foliation $h(\cF^{ss})=\cL^{ss}$, then $f$  admits a dominated splitting along $E^{ws}$ with sub-bundle dimensions matching the dimensions for the dominated splitting of $A$ along $L^{ws}$. Moreover, $f$ has spectral rigidity along every  sub-bundle of $E^{ws}$.

In this section, we always assume that $f$ is Anosov and that it admits an absolutely partially hyperbolic splitting $T\TT^d=E^{ss}\oplus E^{ws}\oplus E^u$ satisfying
$$
\|Df|_{E^{ss}(x)}\|<\mu<m(Df|_{E^{ws}(x)})\leq
\|Df|_{E^{ws}(x)}\|<1<\|Df|_{E^u(x)}\|.
$$
Recall that by Theroem \ref{thm:conjugate}, $f$ is dynamically coherent with invariant foliations $\cF^{\sigma}$ tangent to $E^{\sigma}$ for $\sigma={\it ss,ws,u}$. The linear part $A$ of $f$ also admits a partially hyperbolic splitting $T\TT^d=L^{ss}\oplus L^{ws}\oplus L^u$ with same dimensions as the splitting for $f$. Moreover, let $h:\TT^d\rightarrow\TT^d$ be the topological conjugacy $h\circ f=A\circ h$, then $h$ is H\"older continuous and preserves all invariant foliations:
$$
h(\cF^{\sigma})=\cL^{\sigma}, \qquad \sigma={\it ss,ws,u}.
$$
The key property which we will utilize for proving the spectral rigidity is that the conjugacy $h$ preserves the strong stable foliation $h(\cF^{ss})=\cL^{ss}$.

Since $h(\cF^{ss})=\cL^{ss}$ implies $E^{ss}\oplus E^u$ is integrable, we can apply Proposition \ref{prop:smallest-Lya-expo} which shows that every periodic point of $f$ admits the same smallest Lyapunov exponent $\lambda_0$ in $E^{ws}$.
For every $p\in{\rm Per}(f)$, let $E_0(p)\subset T_p\TT^d$ be the Lyapunov subspace of $p$ associated to the Lyapunov exponent $\lambda_0$, and let $k_0(p)\in\NN$ the dimension of $E_0(p)$: 
$$
1~\leq~k_0(p)~=~{\rm dim}E_0(p)~\leq~{\rm dim}E^{ws}.
$$ 
Let $F_0(p)$ be the Lyapunov subspace of $p$ associated to the Lyapunov exponents contained in the interval $(\lambda_0,0)$, then
$$
T_p\TT^d=E^{ss}(p)\oplus E_0(p)\oplus F_0(p)\oplus E^u(p)
\qquad {\rm where} \qquad E^{ws}(p)=E_0(p)\oplus F_0(p).
$$
Denote by $\lambda_1(p)$ the smallest Lyapunov exponent of $p$ along the bundle $F_0(p)$. Obviously, we have $\lambda_1(p)>\lambda_0$.

Let $\cF_0$ be the $k_0(p)$-dimensional strong stable foliation in $\cF^{ws}(p)$ defined by the Pesin invariant manifolds associated to the Lyapunov exponent $\lambda_0$. That is, for  any $x,y\in\cF^{ws}(p)$, $y\in\cF_0(x)$ if and only if for every $\epsilon>0$,
\begin{align}\label{equ:strong-converge}
\lim_{n\rightarrow+\infty}\frac{1}{n}\log d\big(f^n(x),f^n(y)\big)
<\lambda_0+\epsilon.
\end{align}
Since every weak stable leaf $\cF^{ws}(p)$ is a $C^{1+{\rm H\"older}}$-smooth submanifold, and $f$ is uniformly contracting in $\cF^{ws}_{\it loc}(p)$, Theorem 2.2 of \cite{Bro} implies $\cF_0$ is a $C^{1+{\rm H\"older}}$-foliation in $\cF^{ws}_{\it loc}(p)$. In particular, the tangent bundle $T\cF_0$ is H\"older continuous in $\cF^{ws}_{\it loc}(p)$. Then we define
$$
\cF_0|_{\cF^{ws}(p)}~=~
\bigcup_{n\geq0}f^{-n\kappa}\left( \cF_0|_{\cF^{ws}_{\it loc}(p)} \right),
$$
where $\kappa$ is the period of $p$. This implies $\cF_0$ is a $C^{1+{\rm H\"older}}$-foliation in $\cF^{ws}(p)$, where $T\cF_0$ is H\"older continuous in $\cF^{ws}(p)$. Notice here the H\"older constants are only uniform in $\cF^{ws}_{\it loc}(p)$.

There are two possibilities:
\begin{itemize}
	\item either there exists $p\in{\rm Per}(f)$, such that ${\rm dim}E_0(p)<{\rm dim}E^{ws}$;
	\item or ${\rm dim}E_0(p)={\rm dim}E^{ws}$ for every $p\in{\rm Per}(f)$.
\end{itemize}
In the first case, we will show that the dimensions $k_0(p)={\rm dim}E_0(p)$ are equal for all $p\in{\rm Per}(f)$, and subspaces $E_0(p)$ belong to a sub-bundle $E_0$ of a finer dominated splitting $E^{ws}=E_0\oplus F_0$. This then allows us to continue by induction and to complete the proof of Theorem \ref{thm:main}. We can apply Lemma \ref{lem:k0=ws} to the second case, which shows that all Lyapunov exponents of $A$ in $L^{ws}$ are equal to $\lambda_0$. This finishes the proof Theorem \ref{thm:main}.

\vskip2mm

First we assume that there exists a periodic point $p$ of $f$, such that  ${\rm dim}E_0(p)<{\rm dim}E^{ws}$. 

\begin{lemma}\label{lem:local-cs}
	Let $p$ be a periodic point of $f$ with minimal period $\kappa$ such that $k_0(p)<{\rm dim}E^{ws}$. Then for every $0<\delta\ll(\lambda_1(p)-\lambda_0)$ and every pair of points $x,y\in\cF^{ws}(p)$ satisfying $y\notin\cF_0(x)$, 
	there exists $N>0$, such that
	for every $n\geq N$ and $i\geq0$
	$$
	d_{\cF^{ws}}\left(f^{\kappa (i+n)}(x),f^{\kappa (i+n)}(y)\right)
	~\geq~
	\exp\left[\kappa i\cdot\left(\lambda_1(p)-\delta\right)\right]
	\cdot d_{\cF^{ws}}\left(f^n(x),f^n(y)\right).
	$$
\end{lemma}

\begin{proof}
	Since $x,y\in\cF^{ws}(p)\subset\cF^s(p)$ and $y\notin\cF_0(x)$, we can assume $x,y\in\cF^{ws}_{\it loc}(p)$ after some iteration. 
	Let $\exp_p:T_p\TT^d\rightarrow\TT^d$ be the exponential map. We use the inverse of the exponential map at $p$ to introduce the following notation
	$$
	\exp_p^{-1}(*)~=~v_*~=~v_*^{ss}+v_*^E+v_*^F+v_*^u~\in~
	E^{ss}(p)\oplus E_0(p)\oplus F_0(p)\oplus E^u(p)=T_p\TT^d,
	$$
	for $*=f^{\kappa i}(x),f^{\kappa i}(y)$ and every $i\in\NN$.
	
	Notice that $Df^{\kappa}:T_p\TT^d\rightarrow T_p\TT^d$ preserves the splitting $T_p\TT^d=E^{ss}(p)\oplus E_0(p)\oplus F_0(p)\oplus E^u(p)$. We can express $f^{\kappa}$ in the local coordinates of tangent space $T_p\TT^d$ as
	\begin{align*}
	&\exp_p^{-1}\circ f^{\kappa}\circ\exp_p:~
	T_p\TT^d(\eta)\longrightarrow T_p\TT^d \\
	&\exp_p^{-1}\circ f^{\kappa}\circ\exp_p(v)=Df^{\kappa}(v)+o(v),
	\end{align*}
	where $\lim_{v\rightarrow 0}\|o(v)\|/\|v\|=0$. Since $f^{\kappa i}(y)\notin\cF_0(f^{\kappa i}(x))$ for every $i\geq0$, we have
	\begin{align}\label{equ:d-ws}
	\lim_{i\rightarrow+\infty}
	\frac{\|v_{f^{\kappa i}(x)}^E-v_{f^{\kappa i}(y)}^E\|}{
		\|v_{f^{\kappa i}(x)}^F-v_{f^{\kappa i}(y)}^F\|}=0
	\qquad {\rm and} \qquad
	\lim_{i\rightarrow+\infty}
	\frac{d_{\cF^{ws}}\left(f^{\kappa i}(x),f^{\kappa i}(y)\right)}{
		\|v_{f^{\kappa i}(x)}^F-v_{f^{\kappa i}(y)}^F\|}=1
	\end{align}
	
	For every $\delta>0$, since 
	$$
	\lim_{i\rightarrow+\infty}
	\frac{m\big(Df^{\kappa i}|_{F_0(p)}\big)}{ 
	\exp\left[\kappa i\cdot\lambda_1(p)\right]}=1,
	$$
	we can take $N$ large enough, such that for every $n\geq N$ and $i\geq0$, we have 
	$$
	\|v_{f^{\kappa (n+i)}(x)}^F-v_{f^{\kappa (n+i)}(y)}^F\|
	~\geq~
	\exp\left[\kappa i\cdot\left(\lambda_1(p)-\delta/2\right)\right]\cdot
	\|v_{f^{\kappa n}(x)}^F-v_{f^{\kappa n}(y)}^F\|.
	$$
	The estimation (\ref{equ:d-ws}) finishes the proof of this lemma.
\end{proof}

Let $h:\TT^d\rightarrow\TT^d$ be the conjugacy $h\circ f=A\circ h$, and let $p\in{\rm Per}(f)$. Let $p'=h(p)$ and let
$$
\cL_0|_{\cL^{ws}(p')}=h\big(\cF_0|_{\cF^{ws}(p)}\big)
$$ 
be the push forward of $\cF_0$ to $\cL^{ws}(p')$. Since $\cF_0$ is $f$-invariant
 $f\big(\cF_0|_{\cF^{ws}(p)}\big)=\cF_0|_{\cF^{ws}(f(p))}$ for every $p\in{\rm Per}(f)$, the foliation $\cL_0$ is $A$-invariant:
 $$
 A\big(\cL_0|_{\cL^{ws}(p')}\big)=\cL_0|_{\cL^{ws}(A(p'))},
 \qquad \forall p'\in{\rm Per}(A).
 $$

Since $\cL^{ws}$ is a linear foliation, for any $p',q'\in\TT^d$ and $z'\in \cL^{ws}(q)$, we can define the linear translation 
\begin{align*}
T_{p',z'}:\quad
\cL^{ws}(p') \quad &\longrightarrow  \quad\cL^{ws}(q'),\\
x'\quad \quad &\longmapsto \quad x'+(z'-p'),
\end{align*}
where $+,-$ stand for the standard operations on the group $\TT^d$. It is clear that $T_{p',z'}(p')=z'\in\cL^{ws}(q')$.

Recall that $f$ and $A$ have matching foliations: $h(\cF^{su})=\cL^{su}$ and that $h(\cF^{ws})=\cL^{ws}$, so we have the following lemma.

\begin{lemma}\label{lem:su-holonomy}
	For every $x\in\TT^d$ and $y\in\cF^{su}(x)$, let $\H^{su}_{x,y}:\cF^{ws}(x)\rightarrow\cF^{ws}(y)$ be the holonomy map induced by $\cF^{su}$ from $\cF^{ws}(x)$ to $\cF^{ws}(y)$ satisfying $\H^{su}_{x,y}(x)=y$. 
	Then for every $z\in\cF^{ws}(x)$, we have
	$$
	\H^{su}_{x,y}(z)~=~h^{-1}\circ T_{x',y'}(z')~=~h^{-1}\left(z'+(y'-z')\right),
	$$
	where $x'=h(x)$, $y'=h(y)$, and $z'=h(z)$.
\end{lemma}

The following proposition implies that $\cL_0$ is invariant under all linear translations.
The philosophy of  proposition is similar to that of Proposition \ref{prop:smallest-Lya-expo}.

\begin{proposition}\label{prop:holonomy-invariant}
	For any $p',q'\in{\rm Per}(A)$ and $z'\in \cL^{ws}(q')$, the linear translation $T_{p',z'}:\cL^{ws}(p')\rightarrow\cL^{ws}(q')$ satisfies
	$$
	T_{p',z'}\big( \cL_0(p') \big) ~\subseteq~ \cL_0(z').
	$$	
\end{proposition}

\begin{proof}
	We prove it by contradiction. Assume that there exist $p',q'\in{\rm Per}(A)$, $z'\in\cL^{ws}(q')$, and $x'\in\cL_0(p')$, such that
	$w'=T_{p',z'}(x')\notin\cL_0(z')$. We denote $p=h^{-1}(p'),q=h^{-1}(q')\in{\rm Per}(f)$, $z=h^{-1}(z')\in\cF^{ws}(q)$ and let $x=h^{-1}(x')\in\cF_0(p)$, then
	$w=h^{-1}(w)\notin\cF_0(z)$.
	This implies that $k_0(q)={\rm dim}E_0(q)<{\rm dim}E^{ws}(q)$.

	Let $\lambda_1(q)>\lambda_0$ be the smallest Lyapunov exponent of $q$ along $F_0(q)$, and let $\theta$ be the constant given by Lemma \ref{lem:Diophantine}.  Let
	\begin{align}\label{equ:dim-chi-delta}
	\chi=-\frac{2\cdot\log\big(\max_{x\in\TT^d}\|Df|_{E^u(x)}\|\big) }{\theta\cdot\log\big(\max_{x\in\TT^d}\|Df|_{E^{ss}(x)}\|\big)}>0,
	\qquad {\rm and} \qquad
	\delta=\frac{1}{2\chi+2}\big[ \lambda_1(q)-\lambda_0 \big]>0.
	\end{align}
	By taking an adapted metric, we can assume that
	\begin{align}\label{equ:ws-minnorm}
	\log\big( \min_{x\in\TT^d}m\big(Df|_{E^{ws}(x)}\big) \big)~>~\lambda_0-\delta.
	\end{align}
	
	 Since the linear translation is commuting with the action of $A$:
	$$
	A^i\circ T_{p',z'}(x')~=~T_{A^i(p'),A^i(z')}(A^i(x')), \qquad \forall i\in\ZZ,
	$$ 
	and $\cL_0$ is $A$-invariant, we can assume that $x'\in\cL_{0,{\it loc}}(p')$ and 
	$z',T_{p',z'}(x')\in\cL^{ws}_{\it loc}(p')$ after some $A$-iterations. Let $\kappa$ be the minimal common period of $p'$ and $q'$, which is also the minimal common period of $p,q\in{\rm Per}(f)$. According to Lemma \ref{lem:local-cs}, by taking the positive $A^{\kappa}$-iterations of $p',q',x',z',w'$, we can assume that the corresponding points $p,q,x,z,w$ satisfy the following properties:
	\begin{itemize}
	     \item  
	     $d_{\cF^{ws}}\big(p,f^{\kappa n}(x)\big)
	     \leq \exp\left[\kappa n\cdot(\lambda_0+\delta)\right]\cdot
	     d_{\cF^{ws}}(p,x)$, for every $n\geq0$;
	     \item the points $z,w\in\cF^{ws}(q)$ have the decomposition for $*=z,w$
	     $$ 
	     \exp_q^{-1}(*)~=~v_*~=~v_*^{ss}+v_*^E+v_*^F+v_*^u~\in~
	     E^{ss}(q)\oplus E_0(q)\oplus F_0(q)\oplus E^u(q)=T_q\TT^d,
	     $$
	     which satisfy the following: $d_{\cF^{ws}}(z,w)/\|v^F_z-v^F_w\|$ is arbitrarily close to $1$, and  $\|v^E_z-v^E_w\|/\|v^F_z-v^F_w\|$ is arbitrarily close to $0$. In particular, it follows from the proof of Lemma \ref{lem:local-cs} that there exists $\e_0\ll d_{\cF^{ws}}(z,w)$, such that if $d_{\cF^{ws}}(z_1,z)<\e_0$ and $d_{\cF^{ws}}(w_1,w)<\e_0$, then the corresponding decompositions for $z_1,w_1$ have the same property as $z,w$ avove and 
	     \begin{align}\label{equ:zw-growth}
	         d_{\cF^{ws}}\big(f^{\kappa n}(z_1),f^{\kappa n}(w_1)\big)
	         ~\geq~ \frac{1}{2}\exp\left[\kappa n\cdot(\lambda_1(q)-\delta)\right]\cdot
	         d_{\cF^{ws}}(z,w), \qquad \forall n\geq 0.
	     \end{align}
	\end{itemize}

    Lemma \ref{lem:Diophantine} shows that there exist two sequences of points $\{y_n\}\subset\cF^{ss}(p)$ and $\{z_n\}\subset\cF^{ws}(z)=\cF^{ws}(q)$, such that $z_n=\cF^u(y_n)\cap\cF^{ws}(z)$. Moreover, they satisfy
    $$
    K_n=d_{\cF^{ss}}(p,y_n)\rightarrow+\infty, n\rightarrow\infty, 
    \qquad d_{\cF^u}(y_n,z_n)\leq K_n^{-\theta}, 
    \qquad d_{\cF^{ws}}(z_n,z)\leq K_n^{-\theta}.
    $$
    Let $y_n'=h^{-1}(y_n)$ and $z_n'=h^{-1}(z_n)$ and consider the linear translation $T_{p',x'}=T_{z',w'}$. We have the following 
    $$
    x_n'=T_{p',x'}(y_n')\in\cL^{ss}(x'),
    \qquad {\rm and} \qquad w_n'=T_{p',x'}(z_n')\in\cL^{ws}(z)\cap\cL^u(x_n). 
    $$
    In particular, this implies that
    $$
    d_{L^{ss}}(x',x_n')=d_{L^{ss}}(p,y_n), \qquad
    d_{L^u}(x_n',w_n')=d_{L^u}(y_n',z_n'), \qquad {\rm and} \qquad
    d_{L^{ws}}(w_n',w)=d_{L^{ws}}(z_n',z).
    $$ 
    
    Also consider points $x_n=h(x_n')$ and $w_n=h(w_n')$. Since the conjugacy $h$ preserves all invariant foliations, we have
    $$
    x_n=\cF^{ss}(x)\cap\cF^{ws}(y_n), \qquad {\rm and} \qquad
    w_n=\cF^u(y_n)\cap\cF^{ws}(w)=\cF^u(y_n)\cap\cF^{ws}(z).
    $$
    Moreover, from Lemma \ref{lem:Diophantine}, we have
    $$
    \tilde{K}_n=d_{\cF^{ss}}(x,x_n)\rightarrow+\infty, n\rightarrow\infty, 
    \qquad d_{\cF^u}(x_n,w_n)\leq \tilde{K}_n^{-\theta}, 
    \qquad d_{\cF^{ws}}(w_n,w)\leq \tilde{K}_n^{-\theta}.
    $$
    
    \begin{claim}\label{clm:holo-norm}
    	There exists a constant $C_1>1$, such that for every pair of points $\tilde{y},\tilde{z}$, if $\tilde{z}\in\cF^u(\tilde{y})$ and $d_{\cF^u}(\tilde{y},\tilde{z})\leq1$, then the $\cF^u$-holonomy map $\H^u_{\tilde{y},\tilde{z}}:\cF^{ws}(\tilde{y})\rightarrow\cF^{ws}(\tilde{z})$ is globally defined, and the derivative of $\,\,\H^u_{\tilde{y},\tilde{z}}$ satisfies the following bounds
    	$$
    	\frac{1}{C_1}~<~m\big(D\H^u_{\tilde{y},\tilde{z}}(\tilde{y}_1)\big)
    	~\leq~\|D\H^u_{\tilde{y},\tilde{z}}(\tilde{y}_1)\|~<~C_1,
    	\qquad \forall \tilde{y}_1\in\cF^{ws}(\tilde{y}).
    	$$
    \end{claim}

    \begin{proof}[Proof of the Claim]
    	It follows from the global product structure of $\cF^u$ and $\cF^{ws}$ in each leaf of $\cF^{cu}$ and from the fact that the conjugacy $h$ preserves all invariant foliations, that for every pair of points $\tilde{y},\tilde{z}$ if $\tilde{z}\in\cF^u(\tilde{y})$ and $d_{\cF^u}(\tilde{y},\tilde{z})\leq1$,  the holonomy map $\H^u_{\tilde{y},\tilde{z}}:\cF^{ws}(\tilde{y})\rightarrow\cF^{ws}(\tilde{z})$ is globally defined and $d_{\cF^u}\big(\tilde{y}_1,\H^u_{\tilde{y},\tilde{z}}(\tilde{y}_1)\big)$ is uniformly bounded for every $\tilde{y}_1\in\cF^{ws}(\tilde{y})$. 
    	
    	From the center bunching condition on $f$, the unstable foliation $\cF^u$ is uniformly $C^1$-smooth inside each center unstable leaf (see Lemma \ref{lem:C1-su}). Thus there exists a constant $C_1>1$, such that 
    	$$
    	\frac{1}{C_1}~<~m\big(D\H^u_{\tilde{y},\tilde{z}}(\tilde{y}_1)\big)
    	~\leq~\|D\H^u_{\tilde{y},\tilde{z}}(\tilde{y}_1)\|~<~C_1,
    	\qquad \forall \tilde{y}_1\in\cF^{ws}(\tilde{y}).
    	$$
    \end{proof}
   
    Since $z_n\rightarrow z$ and $w_n\rightarrow w$ as $n\rightarrow\infty$, equation (\ref{equ:zw-growth}) implies that for every sufficiently large $n$ we have
    \begin{align}\label{equ:zw_n-growth}
    d_{\cF^{ws}}\big(f^{\kappa m}(z_n),f^{\kappa n}(w_n)\big)
    ~\geq~ \frac{1}{2}\exp\left[\kappa m\cdot(\lambda_1(q)-\delta)\right]\cdot
    d_{\cF^{ws}}(z,w), \qquad \forall m\geq 0.
    \end{align}
    Also for every sufficiently large $n$ we have the bound
    $d_{\cF^u}(y_n,z_n)\leq K_n^{-\theta}$ where $K_n=d_{\cF^{ss}}(p,y_n)$. Let $k_n$ be defined by
    $$
    k_n=\max\left\{m\in\NN:~d_{\cF^u}\big(f^{\kappa m}(z_n),f^{\kappa m}(y_n)\big)\leq1 \right\},
    \qquad {\rm then} \qquad
    k_n\geq\frac{\theta\cdot\log K_n}{\kappa\cdot\log\big(\max_{x\in\TT^d}\|Df|_{E^u(x)}\|\big)}-1.
    $$
    On the other hand if we define $m_n$ as
    $$
    m_n=\min\left\{m\in\NN:~d_{\cF^{ss}}\big(p,f^{\kappa m}(y_n)\big)\leq1 \right\},
    \qquad {\rm then} \qquad
    m_n\leq
    -\frac{\log K_n}{\kappa\cdot\log\big(\max_{x\in\TT^d}\|Df|_{E^{ss}(x)}\|\big)}+1.
    $$
    So for every sufficiently large $n$, using the definition of $\chi$ in (\ref{equ:dim-chi-delta}) , we have
    \begin{align}\label{equ:k_n/m_n}
        \frac{k_n}{m_n}~\geq~
        \frac{\theta\cdot\log\big(\max_{x\in\TT^d}\|Df|_{E^{ss}(x)}\|\big)}{2\cdot\log\big(\max_{x\in\TT^d}\|Df|_{E^u(x)}\|\big) }
        ~=~\frac{1}{\chi}.
    \end{align}
    
    \vskip2mm
    
    \noindent {\bf Case 1.} There are infinitely many $n$ such that $k_n\geq m_n$.
    
    By taking corresponding subsequence, we can assume that
     $d_{\cF^u}\big(f^{\kappa m_n}(y_n),f^{\kappa m_n}(z_n)\big)\leq 1$ for all $n$ from this subsequence. From Claim \ref{clm:holo-norm} and equation (\ref{equ:zw_n-growth}), we have
     \begin{align*}
          \frac{d_{\cF^u}\big(f^{\kappa m_n}(y_n),f^{\kappa m_n}(x_n)\big) }{d_{\cF^{ws}}\big(p,f^{\kappa m_n}(x)\big)}
          ~&\geq~\frac{1}{C_1}\cdot\frac{d_{\cF^u}\big(f^{\kappa m_n}(z_n), f^{\kappa m_n}(w_n)\big)}{d_{\cF^{ws}}\big(p,f^{\kappa m_n}(x)\big)} \\
          &\geq~\frac{1}{2C_1}\cdot\frac{\exp\left[\kappa m_n\cdot(\lambda_1(q)-\delta)\right]
          	}{\exp\left[\kappa m_n\cdot(\lambda_0+\delta)\right]}
          \cdot\frac{d_{\cF^{ws}}(z,w)}{d_{\cF^{ws}}(p,x)}
     \end{align*}
     Recall from the definition (\ref{equ:dim-chi-delta}) that $\lambda_1(q)-\lambda_0=(2\chi+2)\cdot\delta>2\delta$. Hence we have
     \begin{align}\label{equ:infty}
     \lim_{n\rightarrow\infty}\frac{d_{\cF^u}\big(f^{\kappa m_n}(y_n),f^{\kappa m_n}(x_n)\big) }{d_{\cF^{ws}}\big(p,f^{\kappa m_n}(x)\big)}=+\infty.
     \end{align}
     Notice that if we denote
     $\H^{ss}_{p,f^{\kappa m_n}(y_n)}:\cF^{ws}(p)\rightarrow\cF^{ws}\big(f^{\kappa m_n}(y_n)\big)$ be the holonomy map induced by $\cF^{ss}$ restricted in $\cF^s(p)$, then 
     $$
     \H^{ss}_{p,f^{\kappa m_n}(y_n)}\big(f^{\kappa m_n}(x)\big)~=~f^{\kappa m_n}(x_n).
     $$
     Therefore the equation (\ref{equ:infty}) gives a contradiction since $d_{\cF^{ss}}\big(p,f^{\kappa m_n}(y_n)\big)\leq1$ and $\cF^{ss}$ is uniformly $C^1$-smooth inside the stable leaf $\cF^s(p)$.
     
     \vskip2mm
    
    \noindent {\bf Case 2.} We can assume have $k_n<m_n$ for all sufficiently large $n$.
    
    Recall that $\lambda_1(q)-\lambda_0=(2\chi+2)\cdot\delta$.
    Since $k_n/m_n>1/\chi>0$, from the definition of $m_n$, equation (\ref{equ:ws-minnorm}) and (\ref{equ:zw_n-growth}) imply that
    \begin{align*}
    \frac{d_{\cF^u}\big(f^{\kappa m_n}(y_n),f^{\kappa m_n}(x_n)\big) }{d_{\cF^{ws}}\big(p,f^{\kappa m_n}(x)\big)}
    ~&\geq~\frac{\big[\min_{x\in\TT^d}m\big(Df|_{E^{ws}(x)}\big)\big]^{\kappa(m_n-k_n)}
    	\cdot d_{\cF^{ws}}\big(f^{\kappa k_n}(y_n),f^{\kappa k_n}(x_n)\big)}{d_{\cF^{ws}}\big(p,f^{\kappa m_n}(x)\big)} \\
    &\geq~\frac{1}{2C_1}\cdot\frac{\exp\big[\kappa(m_n-k_n)(\lambda_0-\delta)\big]\cdot
    	\exp\left[\kappa k_n\cdot(\lambda_1(q)-\delta)\right]
    }{\exp\left[\kappa m_n\cdot(\lambda_0+\delta)\right]}
    \cdot\frac{d_{\cF^{ws}}(z,w)}{d_{\cF^{ws}}(p,x)} \\
    &\geq~\frac{1}{2C_1}\cdot
    \frac{\exp\left[\kappa m_n\cdot\left((1-\chi^{-1})(\lambda_0-\delta)+\chi^{-1}(\lambda_1(q)-\delta)\right)\right] }{\exp\left[\kappa m_n\cdot(\lambda_0+\delta)\right]}
    \cdot\frac{d_{\cF^{ws}}(z,w)}{d_{\cF^{ws}}(p,x)} \\
    &\geq~\frac{1}{2C_1}\cdot
    \frac{\exp\left[\kappa m_n\cdot(\lambda_0+\delta+2\chi^{-1}\delta)\right] }{\exp\left[\kappa m_n\cdot(\lambda_0+\delta)\right]}
    \cdot\frac{d_{\cF^{ws}}(z,w)}{d_{\cF^{ws}}(p,x)} \\
    &\longrightarrow~\infty\quad(n\rightarrow\infty).
    \end{align*}
    This is a contradiction, thus finishing the proof of the proposition.
\end{proof}

\begin{corollary}\label{cor:holonomy-invariant}
	For any periodic point of $f$, the dimension of the invariant space associated to the Lyapunov exponent $\lambda_0$ is a constant $k_0$ independent of the point. 
	Moreover, there exists a $k_0$-dimensional linear $A$-invariant foliation $\cL_0$ on $\TT^d$ such that its restriction to $\cL^{ws}(p')$ coincides with $\cL_0|_{\cL^{ws}(p')}$ for every $p'\in{\rm Per}(A)$.
\end{corollary}

\begin{proof}
	From Proposition \ref{prop:holonomy-invariant}, we know that there exists $k_0\in\NN$, such that  $k_0(p)=k_0$ for every $p\in{\rm Per}(f)$. 
	For any $p'\in{\rm Per}(f)$ and $x'\in\cL^{ws}(p)$, by applying Proposition \ref{prop:holonomy-invariant} we have
	$$
	T_{p',x'}(\cL_0(p'))\subset\cL_0(x').
	$$
	
    In particular, if we take $x'\in\cL_0(p')$, then
	$$
	T_{p',x'}(\cL_0(p'))\subset\cL_0(p'), \qquad \forall x'\in\cL_0(p'),
	$$
	which implies $\cL_0(p')$ is a linear subspace in $\cL^{ws}(p')$. Here we use the fact that, if a $k_0$-dimensional subset contained in a linear vector space is invariant by adding any element in this set, then it is a $k_0$-dimensional linear subspace. By taking the linear translation in $\cL^{ws}(p')$, we have $\cL_0|_{\cL^{ws}(p')}$ is a linear foliation in $\cL^{ws}(p')$. This holds for every $p'\in{\rm Per}(A)$.

	Finally, since for every $p',q'\in{\rm Per}(A)$ and $z'\in\TT^d$, the fact that $T_{p',q'}(\cL_0|_{\cL^{ws}(p')})=\cL_0|_{\cL^{ws}(q')}$ implies that
	$$
	T_{p',z'}(\cL_0|_{\cL^{ws}(p')})=T_{q',z'}(\cL_0|_{\cL^{ws}(q')}).
	$$
	So we can  define a $k_0$-dimensional linear foliation $\cL_0$ on $\TT^d$ such that its restriction to $\cL^{ws}(p')$ coincides with $\cL_0|_{\cL^{ws}(p')}$ for every $p'\in{\rm Per}(A)$. Finally, the foliation $\cL_0$ is $A$-invariant since it is $A$-invariant in  $\cL^{ws}(p')$ for every $p'\in{\rm Per}(A)$.
\end{proof}

\begin{proposition}\label{prop:foliation}
	The foliation $\cF_0=h^{-1}(\cL_0)$ is a $C^0$-foliation on $\TT^d$ with uniformly $C^1$-smooth leaves. The tangent bundle $E_0=T\cF_0$ is a H\"older continuous $k_0$-dimensional $Df$-invariant bundle which is jointly integrable with $E^{ss}\oplus E^u$:
	$$
	T(\cF^{ss}\oplus\cF_0\oplus\cF^u)~=~E^{ss}\oplus E_0\oplus E^u.
	$$
\end{proposition}

\begin{proof}
	We have shown that $\cL_0$ is a linear foliation on $\TT^d$. This implies $\cL_0$ is jointly integrable with $\cL^{su}=\cL^{ss}\oplus\cL^u$. Since the conjugacy $h$ matches the strong unstable foliation $h(\cF^{su})=\cL^{su}$,
	the foliation $\cF_0=h^{-1}(\cL_0)$ is preserved by the holonomy map induced by the foliation $\cF^{su}$. That is, for every $x\in\TT^d$ and every $y\in\cF^{su}(x)$, if we denote by 
	$\,\H^{su}_{x,y}:\cF^{ws}(x)\rightarrow\cF^{ws}(y)$  the holonomy map of $\cF^{su}$ which maps $x$ to $y$, then
	$$
	\H^{su}_{x,y}\big(\cL_0|_{\cF^{ws}(x)}\big)~=~\cL_0|_{\cF^{ws}(y)}.
	$$
	
	Recall that $\cF_0$ is a $C^{1+{\rm H\"older}}$-foliation inside $\cF^{ws}(p)$ for every $p\in{\rm Per}(f)$. Lemma \ref{lem:C1-su} shows that the holonomy map $\H^{su}_{x,y}:\cF^{ws}(x)\rightarrow\cF^{ws}(y)$ is 
	$C^{1+{\rm H\"older}}$-smooth for every $x\in\TT^d$ and every $y\in\cF^{su}(x)$. For every $x\in\TT^d$, there exists $y\in\cF^{su}(x)\cap\cF^{ws}(p)$ for some periodic point $p$ such that
	$$
	\cF_0(x)~=~\H^{su}_{y,x}\big(\cF_0(y)\big)
	$$
	which is, hence,  $C^{1+{\rm H\"older}}$-smooth. Moreover, the tangent bundle 
	$$
	E_0(x)=T\cF_0(x)=D\H^{su}_{y,x}\big(T\cF_0(y)\big)
	$$
	varies H\"older continuously with respect to $x$. From the compactness of $\TT^d$, the foliation $\cF_0$ is a $C^0$-foliation with uniformly $C^1$-smooth leaves. Thus
	$$
	D\H^{su}_{x,y}\big(E_0(x)\big)=E_0(y), 
	\qquad \forall x\in\TT^d, \quad\forall y\in\cF^{su}(x).
	$$
    This implies $E_0$ is H\"older continuous and jointly integrable with $E^{ss}\oplus E^u$, and $T(\cF^{ss}\oplus\cF_0\oplus\cF^u)=E^{ss}\oplus E_0\oplus E^u$.
\end{proof}

\begin{proposition}\label{prop:second-Lya-expo}
	There exists $\delta_0>0$, such that for every periodic point $p$ of $f$ if $\lambda_1(p)$ be the smallest Lyapunov exponent of $p$ different to $\lambda_0(p)$ in $E^{ws}(p)$, then  $\lambda_1(p)\geq\lambda_0+\delta_0$.
\end{proposition}

\begin{proof}
	Assume that for every $\delta>0$, there exists $q_{\delta}\in{\rm Per}(f)$ such that $\lambda_1(q_{\delta})<\lambda_0+\delta$. We proceed with a proof by contradiction.
	
	Let $p$ be the fixed point of $f$ and let $\lambda_1(p)>\lambda_0$ be the smallest Lyapunov exponent of $p$ which is different to $\lambda_0$ inside $E^{ws}(p)$. Let $\theta$ be the constant given by Lemma \ref{lem:Diophantine}. Let
	\begin{align}\label{equ:second-chi-delta}
	\chi=-\frac{2\cdot\log\big(\max_{x\in\TT^d}\|Df|_{E^u(x)}\|\big) }{\theta\cdot\log\big(\max_{x\in\TT^d}\|Df|_{E^{ss}(x)}\|\big)}>0,
	\qquad {\rm and} \qquad
	\delta=\frac{1}{2\chi+2}\big[ \lambda_1(p)-\lambda_0 \big]>0.
	\end{align}
	By our assumption, there exists $q\in{\rm Per}(f)$ such that $\lambda_1(q)<\lambda_0+\delta$. Denote by $\kappa\in\NN$ the minimal period of $q$.
	As before, by taking the adapted metric, we can assume that
	$\,\,\log\big( \min_{x\in\TT^d}m(Df|_{E^{ws}(x)}) \big)>(\lambda_0-\delta)$.
	 
	\begin{claim}\label{clm:cs-growth}
		There exist $x\in\cF^{ws}_{\it loc}(q)\setminus\cF_0(q)$, $z\in\cF^{ws}_{\it loc}(p)\setminus\cF_0(p)$ and a constant $\e$ satisfying $0<\e\ll d_{\cF^{ws}}(p,z)$, such that the following properties hold:
		\begin{enumerate}
			\item For every $m\geq0$, the point $x$ satisfies
			  $d_{\cF^{ws}}\big(q,f^{\kappa m}(x)\big)\leq
			     \exp[\kappa m\cdot(\lambda_0+\delta)]\cdot d_{\cF^{ws}}(q,x)$.
			\item For every $m\geq0$ and any two points $z_1,w_1\in\cF^{ws}(p)$ such that $d_{\cF^{ws}}(p,w_1)<\e$  and $d_{\cF^{ws}}(z,z_1)<\e$, the following holds
			$$
			d_{\cF^{ws}}\big(f^{\kappa m}(w_1),f^{\kappa m}(z_1)\big)~\geq~
			\frac{1}{2}\exp\big[\kappa m\cdot(\lambda_1(p)-\delta)\big]\cdot 
			d_{\cF^{ws}}(p,z).
			$$
			\item Let $q'=h^{-1}(q)$, $x'=h^{-1}(x)$, $p'=h^{-1}(p)$, and $z'=h^{-1}(z)$, then $T_{q',p'}(x')=z'$. 
		\end{enumerate}
	\end{claim} 
	
	\begin{proof}[Proof of the Claim]
		The proof follows the analysis in Lemma \ref{lem:local-cs} and Proposition \ref{prop:holonomy-invariant}. Since $\lambda_1(q)<\lambda_0+\delta$, we first take $x\in\cF^{ws}_{\it loc}(q)\setminus\cF_0(q)$ which is contained in the local Pesin stable manifold of $q$ corresponding to the Lyapunov exponent $\lambda_1(q)$. If we take $x$ sufficiently close to $q$, then we, clearly, we have the first item. Then we define 
		$z=h\big(T_{q',p'}(x')\big)$, which proves the third item. It is obvious that $z\notin\cF_0(p)$. If we consider appropriate positive $f^{\kappa}$-iteration of $x,z$, then, according to Lemma \ref{lem:local-cs}, we have
		$$
		d_{\cF^{ws}}\big(f^{\kappa m}(p),f^{\kappa m}(z)\big)
		~\geq~
		\frac{1}{2}\exp\big[\kappa m\cdot(\lambda_1(p)-\delta)\big]\cdot 
		d_{\cF^{ws}}(p,z),
		\qquad \forall m\geq0.
		$$
		Finally, we can take a sufficiently small positive $e\ll d_{\cF^{ws}}(p,z)$ such that the second item of claim also holds.
	\end{proof} 

    Recall that Lemma \ref{lem:Diophantine} states that there exist two sequences of points $\{y_n\}\subset\cF^{ss}(q)$ and $\{w_n\}\subset\cF^{ws}(p)$, such that $w_n=\cF^u(y_n)\cap\cF^{ws}(p)$. Moreover, they satisfy
    $$
    K_n=d_{\cF^{ss}}(q,y_n)\rightarrow+\infty~(n\rightarrow\infty), 
    \qquad d_{\cF^u}(y_n,w_n)\leq K_n^{-\theta}, 
    \qquad d_{\cF^{ws}}(w_n,p)\leq K_n^{-\theta}.
    $$
   
    Let $\H^{ss}_{q,y_n}:\cF^{ws}(q)\rightarrow\cF^{ws}(y_n)$ be the holonomy map induced by $\cF^{ss}$ restricted to $\cF^s(q)$, and let $x_n=\H^{ss}_{q,y_n}(x)$. Let $\H^u_{y_n,w_n}:\cF^{ws}(y_n)\rightarrow\cF^{ws}(w_n)=\cF^{ws}(p)$ be the holonomy map induced by $\cF^u$ restricted to $\cF^{cu}(p)$. Also let $z_n=\H^u_{y_n,w_n}(x_n)$. Then it is obvious that $w_n\rightarrow p$ and $z_n\rightarrow z$ as $n\rightarrow\infty$.
    So for every $n$ large enough, according to the second item of Claim \ref{clm:cs-growth} we have 
	\begin{align}\label{equ:second-zw_n-growth}
	d_{\cF^{ws}}\big(f^{\kappa n}(w_n),f^{\kappa m}(z_n)\big)
	~\geq~ \frac{1}{2}\exp\left[\kappa m\cdot(\lambda_1(p)-\delta)\right]\cdot
	d_{\cF^{ws}}(p,z), \qquad \forall m\geq 0.
	\end{align}
	
	For every sufficiently large $n$, we have
	$\,d_{\cF^u}(y_n,w_n)\leq K_n^{-\theta}$, where $K_n=d_{\cF^{ss}}(q,y_n)$. Let $k_n$ be defined as
	$$
	k_n=\max\left\{m\in\NN:~d_{\cF^u}\big(f^{\kappa m}(w_n),f^{\kappa m}(y_n)\big)\leq1 \right\},
	\qquad {\rm then} \qquad
	k_n\geq\frac{\theta\cdot\log K_n}{\kappa\cdot\log\big(\max_{x\in\TT^d}\|Df|_{E^u(x)}\|\big)}-1.
	$$
	On the other hand, let $m_n$ be defined as
	$$
	m_n=\min\left\{m\in\NN:~d_{\cF^{ss}}\big(q,f^{\kappa m}(y_n)\big)\leq1 \right\},
	\qquad {\rm then} \qquad
	m_n\leq
	-\frac{\log K_n}{\kappa\cdot\log\big(\max_{x\in\TT^d}\|Df|_{E^{ss}(x)}\|\big)}+1.
	$$
	So for every sufficiently large $n$, using the definition of $\chi$~(\ref{equ:second-chi-delta}) , we obtain
	\begin{align}\label{equ:second-k_n/m_n}
	\frac{k_n}{m_n}~\geq~
	\frac{\theta\cdot\log\big(\max_{x\in\TT^d}\|Df|_{E^{ss}(x)}\|\big)}{2\cdot\log\big(\max_{x\in\TT^d}\|Df|_{E^u(x)}\|\big) }
	~=~\frac{1}{\chi}.
	\end{align}
	
	If there exist infinitely many $n$ such that $k_n\geq m_n$, then
	$d_{\cF^u}\big(f^{\kappa m_n}(y_n),f^{\kappa m_n}(z_n)\big)\leq 1$ for all these $n$. From Claim \ref{clm:holo-norm} and Equation (\ref{equ:second-zw_n-growth}), we have
	\begin{align*}
	\frac{d_{\cF^u}\big(f^{\kappa m_n}(y_n),f^{\kappa m_n}(x_n)\big) }{d_{\cF^{ws}}\big(p,f^{\kappa m_n}(x)\big)}
	~&\geq~\frac{1}{C_1}\cdot\frac{d_{\cF^u}\big(f^{\kappa m_n}(z_n), f^{\kappa m_n}(w_n)\big)}{d_{\cF^{ws}}\big(p,f^{\kappa m_n}(x)\big)} \\
	&\geq~\frac{1}{2C_1}\cdot\frac{\exp\left[\kappa m_n\cdot(\lambda_1(q)-\delta)\right]
	}{\exp\left[\kappa m_n\cdot(\lambda_0+\delta)\right]}
	\cdot\frac{d_{\cF^{ws}}(z,w)}{d_{\cF^{ws}}(p,x)}
	\end{align*}
    Recall that from the definition (\ref{equ:second-chi-delta}) we have that $\lambda_1(q)-\lambda_0=(2\chi+2)\cdot\delta>2\delta$. Hence
	$$
	\lim_{n\rightarrow\infty}\frac{d_{\cF^u}\big(f^{\kappa m_n}(y_n),f^{\kappa m_n}(x_n)\big) }{d_{\cF^{ws}}\big(q,f^{\kappa m_n}(x)\big)}=+\infty.
	$$
	Let 
	$\H^{ss}_{q,f^{\kappa m_n}(y_n)}:\cF^{ws}(q)\rightarrow\cF^{ws}\big(f^{\kappa m_n}(y_n)\big)$ be the holonomy map induced by $\cF^{ss}$ restricted to $\cF^s(q)$. Then 
	$$
	\H^{ss}_{q,f^{\kappa m_n}(y_n)}\big(f^{\kappa m_n}(x)\big)~=~f^{\kappa m_n}(x_n).
	$$
	This contradicts to the fact that $d_{\cF^{ss}}\big(q,f^{\kappa m_n}(y_n)\big)\leq1$ and $\cF^{ss}$ is uniformly $C^1$-smooth in the stable leaf $\cF^s(p)$.
	
	\vskip2mm
	
	Therefore, now we can assume have $k_n<m_n$ for every $n$. Recall that
    $\lambda_1(p)-\lambda_0=(2\chi+2)\cdot\delta$, and that
    $$
    \log\big( \min_{x\in\TT^d}m(Df|_{E^{ws}(x)}) \big)>(\lambda_0-\delta).
    $$
     Since $k_n/m_n>1/\chi>0$, from the definition of $m_n$, equation (\ref{equ:second-zw_n-growth}) implies that
	\begin{align*}
	\frac{d_{\cF^u}\big(f^{\kappa m_n}(y_n),f^{\kappa m_n}(x_n)\big) }{d_{\cF^{ws}}\big(q,f^{\kappa m_n}(x)\big)}
	~&\geq~\frac{\big[\min_{x\in\TT^d}m\big(Df|_{E^{ws}(x)}\big)\big]^{\kappa(m_n-k_n)}
		\cdot d_{\cF^{ws}}\big(f^{\kappa k_n}(y_n),f^{\kappa k_n}(x_n)\big)}{d_{\cF^{ws}}\big(q,f^{\kappa m_n}(x)\big)} \\
	&\geq~\frac{1}{2C_1}\cdot\frac{\exp\big[\kappa(m_n-k_n)(\lambda_0-\delta)\big]\cdot
		\exp\left[\kappa k_n\cdot(\lambda_1(p)-\delta)\right]
	}{\exp\left[\kappa m_n\cdot(\lambda_0+\delta)\right]}
	\cdot\frac{d_{\cF^{ws}}(z,w)}{d_{\cF^{ws}}(q,x)} \\
	&\geq~\frac{1}{2C_1}\cdot
	\frac{\exp\left[\kappa m_n\cdot\left((1-\chi^{-1})(\lambda_0-\delta)+\chi^{-1}(\lambda_1(p)-\delta)\right)\right] }{\exp\left[\kappa m_n\cdot(\lambda_0+\delta)\right]}
	\cdot\frac{d_{\cF^{ws}}(z,w)}{d_{\cF^{ws}}(q,x)} \\
	&\geq~\frac{1}{2C_1}\cdot
	\frac{\exp\left[\kappa m_n\cdot(\lambda_0+\delta+2\chi^{-1}\delta)\right] }{\exp\left[\kappa m_n\cdot(\lambda_0+\delta)\right]}
	\cdot\frac{d_{\cF^{ws}}(z,w)}{d_{\cF^{ws}}(q,x)} \\
	&\longrightarrow~\infty\quad(n\rightarrow\infty).
	\end{align*}
	This is absurd. We finished the proof of this proposition.
\end{proof}

\begin{proposition}\label{prop:domiated-splitting}
	There exists a continuous $Df$-invariant bundle $F_0\subset E^{ws}$ and an integer $m\in\NN$, such that 
	$$
	E^{ws}=E_0\oplus F_0, \qquad {\it and} \qquad
	\|Df^m|_{E_0(x)}\|<\mu_0^m<m(Df^m|_{F_0(x)}), \quad\forall x\in\TT^d.
	$$ 
	Here the constant $\mu_0=\exp\big(\lambda_0+(\delta_0/2)\big)\in(\mu,1)$, and $\delta_0$ is the constant given by Proposition \ref{prop:second-Lya-expo}. 
	
	This implies that $E^{ws}=E_0\oplus F_0$ is a finer dominated splitting for $f$, and $f$ admits another absolutely partially hyperbolic splitting:
	$$
	T\TT^d=\big(E^{ss}\oplus E_0\big)\oplus F_0\oplus E^u,
	$$
	where the bundle $(E^{ss}\oplus E_0)\oplus E^u$ is integrable.
\end{proposition}

\begin{proof}
	By Proposition~\ref{prop:foliation} we have that $E_0$ is a $k_0$-dimensional H\"older continuous $Df$-invariant bundle on $\TT^d$.
	Let $E_0^{\perp}(x)=\big\{v\in E^{ws}(x):v\perp E^0(x)\big\}$, which obviously also varies H\"older continuously with respect to $x$. Let $E_0^{\perp}=\bigcup_{x\in\TT^d}E_0^{\perp}(x)$, then we have the decomposition
	$$
	E^{ws}=E_0\oplus E_0^{\perp}.
	$$ 
	Let ${\it proj}_0:E^{ws}\rightarrow E_0^{\perp}$ be the projection with respect to the direct sum $E^{ws}=E_0\oplus E_0^{\perp}$.
	
	We take a H\"older continuous base
	$\big\{e_1(x),\cdots,e_{k_0}(x)\big\}_{x\in\TT^d}$ on $E_0$, and a H\"older continuous base 
    $\big\{e_{k_0+1}(x),\cdots,e_{{\rm dim}E^{ws}}(x)\big\}_{x\in\TT^d}$ on $E_0^{\perp}$.
	There exist three families of matrices $\{A(x)\}_{x\in\TT^d}$, $\{B(x)\}_{x\in\TT^d}$, and $\{C(x)\}_{x\in\TT^d}$, such that in the base
	$\big\{e_1(x),\cdots,e_{k_0}(x),e_{k_0+1}(x),\cdots,e_{{\rm dim}E^{ws}}(x)\big\}_{x\in\TT^d}$, we have
	\begin{equation*}
	    Df|_{E^{ws}(x)}~=~\left(  
	       \begin{array}{cc}
	           A(x) & B(x) \\
	           0    & C(x) \\
	       \end{array}
	       \right)
	       \qquad \forall x\in\TT^d.
	\end{equation*}
	Thus
	$$
	Df(v)=A(x)v, \quad\forall v\in E_0(x),
	\qquad {\rm and} \qquad
	{\it proj}_0\circ Df(u)=C(x)u,
	\quad \forall u\in E_0^{\perp}(x).
	$$
	In particular, since $Df|_{E^{ws}(x)}$ is also H\"older continuous, $A(x)$, $B(x)$ and $C(x)$ vary H\"older continuously with respect to $x\in\TT^d$.
	
    For every $x\in\TT^d$ and $n\geq1$, we introduce the following notation for the cocycles
    $$
    A^n(x)=\prod_{i=0}^{n-1}A(f^i(x))
    \qquad {\rm and} \qquad 
    C^n(x)=\prod_{i=0}^{n-1}C(f^i(x)).
    $$ 
    
    \begin{claim}\label{clm:AC-norm}
    	For every $\e>0$, there exists an integer $N(\e)>0$, such that for every $x\in\TT^d$ and every $n\geq N(\e)$, we have
    	$$
    	\|A^n(x)\|~\leq~\exp\big(n\cdot(\lambda_0+\e)\big),
    	\qquad {\it and} \qquad
    	m\big(C^n(x)\big)~\geq~
    	\exp\big(n\cdot(\lambda_0+\delta_0-\e)\big).
    	$$
    \end{claim}

    \begin{proof}[Proof of the Claim]
    	For every $p\in{\rm Per}(f)$, all the Lyapunov exponents of $f$ in $E_0(p)$ are equal to $\lambda_0$. This implies the Lyapunov exponents $A^{\pi(p)}(p)$ are equal to $\lambda_0$, where $\pi(p)$ is the period of $p\in{\rm Per}(f)$. Since $f$ is Anosov and $\{A(x)\}_{x\in\TT^d}$ is H\"older continuous, Theorem 1.3 of \cite{K} says that for every $\e>0$, there exists $N_1(\e)>0$, such that
    	$$
    	\|A^n(x)\|~\leq~\exp\big(n\cdot(\lambda_0+\e)\big), 
    	\qquad \forall x\in\TT^d,~\forall n\geq N_1(\e).
    	$$
    	
    	For every periodic point $p$ of $f$, let $F_0(p)\subset E^{ws}(p)$ be the Lyapunov subspace of $A^{\pi(p)}(p)$ corresponding to the Lyapunov exponents larger than $\lambda_0$ in $E^{ws}$, then $E^{ws}(p)=E_0(p)\oplus F_0(p)$. Then there exists a constant $C_p>1$, such that the norm of the projection
    	${\it proj}_0|_{F_0(p)}:F_0(f^n(p))\rightarrow E_0^{\perp}(f^n(p))$ satisfies
    	$$
    	C_p^{-1}~\leq~m\big({\it proj}_0|_{F_0(f^n(p))}\big)~\leq~
    	\|{\it proj}_0|_{F_0(f^n(p))}\|~\leq~C_p, \qquad \forall n\in\ZZ.
    	$$ 
    	Since $E_0$ is $Df$-invariant, for every unit vector $u\in E_0^{\perp}(p)$, we have
    	\begin{align*}
    	     \lim_{n\rightarrow+\infty}\frac{1}{n}\log\|C^n(p)u\|~&=~
    	     \lim_{n\rightarrow+\infty}
    	     \frac{1}{n}\log\|{\it proj}_0\circ Df^n(u)\| \\
    	     &\geq~\lim_{n\rightarrow+\infty}
    	     \frac{1}{n}\log\big[C_p^{-1}\cdot\|Df^n(u)\|\big] \\
    	     &\geq~\lambda_0+\delta_0.
    	\end{align*}
    	The last inequality holds because all the Lyapunov exponents of $f$ in $F_0(p)$ is no less than $\lambda_0+\delta_0$. This implies that the Lyapunov exponents of $C^{\pi(p)}(p)$ are no less than $\lambda_0+\delta_0$ for every $p\in{\rm Per}(f)$. Using Theorem 1.3 of~\cite{K} again we have that for every $\e>0$, there exists $N_2(\e)>0$, such that
    	$$
    	m\big(C^n(x)\big)~\geq~
    	\exp\big(n\cdot(\lambda_0+\delta_0-\e)\big),
    	\qquad \forall x\in\TT^d,~\forall n\geq N_2(\e).
    	$$
    	This finishes the proof of the claim.
    \end{proof}
    
    Finally, since $B(x)$ varies continuously with respect to $x\in\TT^d$, the norm $\|B(x)\|$ is uniformly bounded. Thus there exists a continuous cone-field $\{\cC(x)\}_{x\in\TT^d}$ containing $E_0^{\perp}$ in $E^{ws}$, such that 
    $$
    Df\big(\overline{\cC(x)}\big)~\subset~\cC(f(x)), \qquad \forall x\in\TT^d.
    $$
    Therefore, by the cone-field criterion \cite[Theorem 2.6]{CP}, there exists a dominated splitting $E^{ws}=E_0\oplus F_0$. Moreover, if we take $\e\ll(\delta_0/2)$ in Claim \ref{clm:AC-norm} and let $\mu_0=\exp\big(\lambda_0+(\delta_0/2)\big)$, then there exists $m\in\NN$, such that
    $$
    \|Df^m|_{E_0(x)}\|<\mu_0^m<m(Df^m|_{F_0(x)}), \quad\forall x\in\TT^d.
    $$
    Thus $f$ is absolutely partially hyperbolic with respect to the splitting 
    $T\TT^d=\big(E^{ss}\oplus E_0\big)\oplus F_0\oplus E^u$, and the bundle $(E^{ss}\oplus E_0)\oplus E^u$ is integrable.
\end{proof}

Now we can prove Theorem \ref{thm:main}.

\begin{proof}[Proof of Theorem \ref{thm:main}]
	According to Proposition \ref{prop:domiated-splitting}, $f$ admits a finer dominated splitting $T\TT^d=E^{ss}\oplus E_0\oplus F_0\oplus E^u$, where
	$$
	T\TT^d=\big(E^{ss}\oplus E_0\big)\oplus F_0\oplus E^u
	$$
	is an absolutely partially hyperbolic splitting. Moreover, Proposition \ref{prop:smallest-Lya-expo} and Corollary \ref{cor:holonomy-invariant} show that there exists an $f$-invariant foliation $\cF_0$ satisfying $T\cF_0=E_0$, and such that all Lyapunov exponents of $f$ along $E_0$ are equal to $\lambda_0$. If ${\rm dim}E_0={\rm dim}E^{ws}$, then, by Lemma~\ref{lem:k0=ws} we have that all Lyapunov exponents of $A$ along $L^{ws}$ are equal to $\lambda_0$ which completes the proof Theorem \ref{thm:main}.
	
	Otherwise we have that $F_0$ is non-trivial subbundle. Then Theorem \ref{thm:conjugate} implies that $A$ admits a finer dominated splitting $T\TT^d=L^{ss}\oplus L_0\oplus G_0\oplus L^u$ where ${\rm dim}L_0={\rm dim}E_0$. By Lemma \ref{lem:k0=ws}  all Lyapunov exponents of $A$ along $L_0$ are equal to $\lambda_0$.
	Since $T\TT^d=\big(E^{ss}\oplus E_0\big)\oplus F_0\oplus E^u$ is also absolutely paritally hyperbolic and $f$ satisfies the center bunching condition for this splitting, we have
	$$
	\|Df|_{F_0(x)}\|=\|Df|_{E^{ws}(x)}\|<
	m\big(Df|_{E^{ws}(x)}\big)\cdot m\big(Df|_{E^u(x)}\big)
	<m\big(Df|_{F_0(x)}\big)\cdot m\big(Df|_{E^u(x)}\big),
	\qquad\forall x\in\TT^d.
	$$
	Then we can  apply Proposition~\ref{prop:smallest-Lya-expo}, Corollary~\ref{cor:holonomy-invariant} and Proposition~\ref{prop:domiated-splitting} again to this splitting. Proceeding by induction finishes the proof of Theorem \ref{thm:main}.
\end{proof}

\section{Integrability of extremal sub-bundles for symplectic diffeomorphisms}\label{sec:T4}

In this section, we prove Theorem \ref{thm:T4}. 
It is a corollary of the following theorem, which is Theorem \ref{thm:T4general}.

\begin{theorem}\label{thm:C2}
	Let $A\in{\rm GL}(d,\ZZ)$ be hyperbolic irreducible automorphism which admits a dominated splitting 
	$$
	T\TT^d=L^{ss}\oplus L^{ws}\oplus L^{wu}\oplus L^{uu}
	\qquad \text{with} \qquad
	{\rm dim}L^{ws}={\rm dim}L^{wu}=1.
	$$ 
	For every $f\in{\rm Diff}^2(\TT^d)$ which is sufficiently $C^1$-close to $A$, let $T\TT^d=E^{ss}\oplus E^{ws}\oplus E^{wu}\oplus E^{uu}$ be the corresponding dominated splitting of $f$. Then $E^{ss}\oplus E^{uu}$ is integrable if and only if $f$ is spectrally rigid along $E^{ws}$ and $E^{wu}$.
\end{theorem}

Using this theorem we can prove Theorem \ref{thm:T4}.

\begin{proof}[Proof of Theorem \ref{thm:T4}]
	Since $A\in{\rm Sp}(4,\ZZ)$ is irreducible and non-conformal, it admits a dominated splitting $T\TT^4=L^{ss}\oplus L^{ws}\oplus L^{wu}\oplus L^{uu}$. 
	%For every symplectic diffeomorphism $f\in{\rm Diff}^2_{\omega}(\TT^2)$ which is $C^1$-close to $A$, it admits a dominated splitting $T\TT^4=E^{ss}\oplus E^{ws}\oplus E^{wu}\oplus E^{uu}$. 
	
	If $f$ is smoothly conjugate to $A$ via a diffeomorphism $h:\TT^4\to\TT^4$, $h\circ f=A\circ h$, then $\cF^{su}=h^{-1}(\cL^{su})$ is a foliation tangent to the extremal symplectic bundle $E^{su}=E^{ss}\oplus E^{uu}$. Here $\cL^{su}$ is the linear foliation tangent to $L^{su}=L^{ss}\oplus L^{uu}$. Thus $E^{su}$ is integrable.
	
	Conversely, if the extremal symplectic bundle $E^{su}=E^{ss}\oplus E^{uu}$ is integrable, then Theorem \ref{thm:C2} implies that $f$ is spectrally rigid along $E^{ws}$ and $E^{wu}$. Hence the topological conjugacy $h$ is, in fact, $C^{1+{\rm H\"older}}$-smooth along $\cF^{ws}$ and along $\cF^{wu}$. Hence, by applying Journ\'e's theorem \cite{J}, we have that $h$ is smooth along $\cF^c=\cF^{ws}\oplus\cF^{wu}$. Then, applying Theorem \ref{thm:local} we have that $h(\cF^{ss})=\cL^{ss}$ and $h(\cF^{uu})=\cL^{uu}$. 	
    Note that now holonomy of the foliation $\cF^{su}$ can be viewed as the holonomy $\cL^{su}$ (which is smooth) conjugated by restrictions of $h$ to the leaves of $\cF^c$. Hence the holonomy and the foliation $\cF^{su}$ is also $C^{1+{\rm H\"older}}$.
	From the smoothness of every leaf of $\cF^{su}$, the conjugacy $h$ is smooth along $\cF^c$ implies the foliation $\cF^{su}$ is $C^{1+{\rm H\"older}}$-smooth. Further, since the bundle $E^c$ is symplectic orthogonal to $E^{su}$, the foliation $\cF^c$ is also $C^{1+{\rm H\"older}}$-smooth. Moreover, from symplectic orthogonality an elementary lemma~\cite[Lemma 8.14]{AV} shows that the symplectic structure $\omega|_{\cF^{su}}$ is invariant under the holonomy map induced by center foliation $\cF^c$ on $\TT^4$. 

    Recall that the restriction $\omega|_{\cF^{su}}$ is non-degenerate. Indeed, if $e^{ss}$ is a unit vector field in $E^{ss}$ then from mismatch of contraction and expansion rates we easily see that both $E^{ws}$ and $E^{wu}$ are in the kernel of $\omega(e^{ss},\cdot)$. Because $\omega$ is a non-degenerate symplectic form this kernel is 2-dimensional, and, hence does not contain $E^{uu}$.
	
	We can view the restriction $\omega|_{\cF^{su}}$ as a tranverse invariant measure for the foliation $\cF^c$. We can push it forward using $h$ and obtain the measure  $h_*\big(\omega|_{\cF^{su}}\big)$  which is holonomy invariant with respect to $h(\cF^c)=\cL^c$. Since $\cL^c$ is linear irrational foliation, it must be uniquely  ergodic on $\TT^4$. Hence, we can conclude that $h_*\big(\omega|_{\cF^{su}}\big)$ is the Lebesgue measure on $\cL^{su}$ and the restriction of $h$ to $\cF^{su}$ is absolutely continuous. In fact, we have
	$$h_*\big(\omega|_{\cF^{su}}\big)=\omega|_{\cL^{su}}$$
	Now recall that, since we have $h(\cF^{ss})=\cL^{ss}$ and $h(\cF^{uu})=\cL^{uu}$, so $h$ sends conditional measures of $\omega|_{\cF^{su}}$ along $\cF^{ss}$-leaves and $\cF^{uu}$-leaves to the conditional measures of $\omega|_{\cL^{su}}$ along $\cL^{ss}$-leaves and $\cL^{uu}$-leaves, respectively. Thus $h$ is absolutely continuous along the leaves of $\cF^{ss}$ and along the leaves of $\cF^{uu}$. Lemma 2.4 of \cite{Go17} shows that the conjugacy $h$ is actually $C^{1+{\rm H\"older}}$-smooth along every leaf of $\cF^{ss}$ and every leaf of $\cF^{uu}$. Finally, Journ\'e's theorem \cite{J} implies that the conjugacy $h$ is $C^{1+{\rm H\"older}}$-smooth on $\TT^4$.
\end{proof}

In the rest of this section, we prove Theorems \ref{thm:C2}.
We will always assume that $A\in{\rm GL}(d,\ZZ)$ is hyperbolic and irreducible with the dominated splitting
$$
T\TT^d=L^{ss}\oplus L^{ws}\oplus L^{wu}\oplus L^{uu},
\qquad \text{where} \qquad
{\rm dim}L^{ws}={\rm dim}L^{wu}=1.
$$
Let $f\in{\rm Diff}^2(\TT^d)$ be a diffeomorphism which is sufficiently $C^1$-close to $A$ so that it admits a dominated splitting $T\TT^d=E^{ss}\oplus E^{ws}\oplus E^{wu}\oplus E^{uu}$ with the same dimensions as $A$. Under some further assumptions we will show that $E^{ss}\oplus E^{uu}$ being integrable implies that $f$ has spectral rigidity along $E^{ws}$ and $E^{wu}$.

Denote by $\cF^{su}$ the foliation to which $\cF^{ss}$ and $\cF^{uu}$ jointly integrate, then $T\cF^{su}=E^{su}=E^{ss}\oplus E^{uu}$. It is easy to see that this foliation is $f$-invariant. Let $\cF^c=\cF^{ws}\oplus\cF^{wu}$ and $\cL^c=\cL^{ws}\oplus\cL^{wu}$ be the center foliations tangent to $E^{ws}\oplus E^{wu}$ and $L^{ws}\oplus L^{wu}$, respectively. Then the lifting foliations of $\cF^{su}$ and $\cF^c$ have global product structures on $\RR^d$.

We begin with the following observation. 
\begin{lemma}
	\label{lemma_31}
	If $f\in{\rm Diff}^2(\TT^d)$ is sufficiently $C^1$-close to $A$ and at least one of the distributions $E^{ss}\oplus E^{wu}$ and $E^{ws}\oplus E^{uu}$ integrates to a foliation then the conjugacy $h$ matches the extremal foliations:
	$$h(\cF^{uu})=\cL^{uu}\,\,\,\,\,\,\,\mbox{and}\,\,\,\,\,\,\,h(\cF^{ss})=\cL^{ss}$$
\end{lemma}

\begin{proof}
	If $E^{ws}$ integrates jointly with $E^{uu}$ to a foliation $\mathcal W$. Then the topological foliation $h(\mathcal W)$ is $A$-invariant and subfoliated by the linear minimal foliation $\cL^{ws}=h(\cF^{ws})$. By Lemma~\ref{lem:linear-foli}, $h(\mathcal W)$ is linear. We have $\cF^{uu}=\mathcal W\cap \cF^u$. Hence $h(\cF^{uu})=h(\mathcal W)\cap h(\cF^u)$ and we can conclude that $h(\cF^{uu})$ is a linear foliation as an intersection of two linear foliations. Then it follows that $h(\cF^{uu})=\cL^{uu}$. Once we have this, we know that the foliation $h(\cF^{su})$ is also subfoliated by a linear minimal foliation $h(\cF^{uu})=\cL^{uu}$. Hence, by applying Lemma~\ref{lem:linear-foli} again, we have that $h(\cF^{su})$ is also linear and we can conclude that $h(\cF^{ss})=\cL^{ss}$. The proof is the same in the case when $E^{ss}\oplus E^{wu}$ is integrable.
\end{proof}

\begin{remark}
	This lemma also holds if we assume that one of $E^{ss}\oplus E^u$ and $E^s\oplus E^{uu}$ is integrable. The proof is the same.
\end{remark}

Once we know that both $\cF^{uu}$ and $\cF^{ss}$ become corresponding linear foliations after the conjugacy, Theorem \ref{thm:local} gives rigidity of Lyapunov exponents along the weak unstable and weak stable subbundles.
Similarly if $E^u=E^{wu}\oplus E^{uu}$ integrates jointly with $E^{ss}$ we can finish the proof in the same way. 

Hence, from now on {\it we will assume that $E^{ws}$ does not integrate jointly with $E^{uu}$ and $E^{wu}$ does not integrate jointly with $E^{ss}$. } Our objective now is to rule this out, which would then complete the proof.

\subsection{Solvable action on the center leaf}\label{subsec:solvable}

Consider lifted dynamical systems $A\colon\RR^d\to\RR^d$ and $F\colon\RR^d\to\RR^d$.  By composing with a small translation, we can assume that $f$ has a lift $F:\RR^d\rightarrow\RR^d$ such that $F(0)=0$. Hence, from now on, we always assume that the lift $F$ fixes $0\in\RR^d$.

We also lift all foliations to the universal cover.  Because we have arranged that $F(0)=0$, the conjugacy $h$ has a lift $H:\RR^d\rightarrow\RR^d$ such that  $H\circ F=A\circ H$ on $\RR^d$ and $H(0)=0$.  

We will use the same notation for the lifted foliations and it is not going to cause any confusion because for the remainder of the proof we will work only on the universal cover.  Then we have  $H(\cF^c)=\cL^c$ and $H(\cF^c(0))=\cL^c(0)$.

Denote by $\Gamma$ the semidirect product $\ZZ\ltimes \ZZ^d$, where the generator of $\ZZ$ acts on $\ZZ^d$ via the hyperbolic matrix $A\in{\rm GL}(d,\ZZ)$.
The natural action of  $\Gamma$ on $\RR^d$ is given by affine transformations
$$
(k,n)(x)=A^k(x)+n, \qquad\forall (k,n)\in\Gamma, ~\forall x\in\RR^d.
$$
By a direct calculation one can easily check that $[\Gamma,\Gamma]=\ZZ^d$.

The action of $\Gamma$ on $\RR^d$ preserves the linear foliation $\cL^{su}=\cL^{ss}\oplus \cL^{uu}$ and, hence, descends to an action on the leaf space of the foliation $\cL^{su}$. Each leaf of $\cL^{su}$ intersects the fixed leaf $\cL^c(0)\subset\RR^4$ precisely once. Hence we can identify the space of leaves of $\cL^{su}$ with $\cL^c(0)$ and we have the affine action $ \Gamma\times\cL^c(0)\to\cL^c(0)$.

Because we have the $F$-invariant and $\ZZ^d$-invariant foliation $\cF^{su}$ which has global product structure $\cF^c=\cF^{ws}\oplus\cF^{wu}$, we can obtain a similar action on the non-linear side. Namely, we identify $\cF^c(0)$ with the leaf space of $\cF^{su}$. The action of $\Gamma$ on $\RR^d$ given by $(k,n)(x)=F^k(x)+n$ induces the action $\alpha\colon \Gamma\times\cF^c(0)\to\cF^c(0)$, which is explicitly given by
$$
\alpha(k,n)(x)=\cF^{su}\big(F^k(x)+n\big)\cap \cF^c(0), 
\qquad\forall (k,n)\in\Gamma, ~\forall x\in\cF^c(0).
$$
Also we will use the following notation for the $\ZZ^d$ subaction
$$
\bar\alpha(n)=\alpha(0,n), \qquad \forall n\in\ZZ^d. 
$$
Note that,  while the conjugacy $H$ conjugates $A|_{\cL^c(0)}$ and $F|_{\cF^c(0)}$, the affine action on $\cL^c(0)$ and $\alpha$ are, a priori, not conjugate. This is because we do not have $H(\cF^{su})=\cL^{su}$. However, the action $\alpha$ is still solvable.

\begin{lemma}\label{lem:solve}
	For every $k\in\ZZ$ and $n\in\ZZ^d$, the action  $\alpha\colon \Gamma\times\cF^c(0)\to\cF^c(0)$ satisfies
	$$
	\alpha(k,0)\circ\alpha(0,n)\circ\alpha(-k,0)=\alpha(0,A^kn).
	$$
	Moreover, for every $k\in\ZZ\setminus\{0\}$ and $n\in\ZZ^d\setminus\{0\}$, there exist $p_1,p_2,\cdots,p_{d+1}\in\ZZ$ which are not all zeros, such that
	$$
	\prod_{i=1}^{d+1}
	\big[\alpha(ik,0)\circ\alpha(0,n)\circ
	\alpha(-ik,0)\circ\alpha(0,-n)\big]^{p_i}=\alpha(0,0)
	={\rm Id}_{\cF^c(0)}.
	$$
\end{lemma} 

\begin{proof} The first equation is just a relation in $\ZZ\ltimes\ZZ^d$ and can be verified by a direct calculation.
%	For every $x\in\cF^c(0)$, we have
%	\begin{align*}
%	   \alpha(k,0)\circ\alpha(0,n)\circ\alpha(-k,0)(x)
%	   ~&=~\alpha(k,0)\circ\alpha(0,n)\big(F^{-k}(x)\big) \\
%	    &=~\alpha(k,0)\big(\cF^{su}\big(F^{-k}(x)+n\big)\cap\cF^c(0)\big) \\
%	    &=~F^k\big(\cF^{su}\big(F^{-k}(x)+n\big)\big)\cap\cF^c(0) \\
%	    \big(F^k(n)=A^kn\big)~&=~
%	    \cF^{su}\big(x+A^kn\big)\cap\cF^c(0) \\
%	    &=~\alpha(0,A^kn)(x).
%	\end{align*}
%	Thus $\alpha(k,0)\circ\alpha(0,n)\circ\alpha(-k,0)=\alpha(0,A^kn)$
%	for every $k\in\ZZ$ and $n\in\ZZ^d$. 
	
	Now for every $k\in\ZZ\setminus\{0\}$, $n\in\ZZ^d\setminus\{0\}$, and $i=1,\cdots,d+1$, we have
	$$
	\alpha(ik,0)\circ\alpha(0,n)\circ
	\alpha(-ik,0)\circ\alpha(0,-n)=
	\alpha(0,A^{ik}n-n)=
	\bar\alpha(A^{ik}n-n).
	$$
	Since $A$ is Anosov, $A^{ik}n-n$ is nonzero in $\ZZ^d$. This implies $\left\{A^{ik}n-n\right\}_{i=1}^{d+1}$ are $d+1$ nonzero elements in $\ZZ^d$. So there exist $p_1,p_2,\cdots,p_{d+1}\in\ZZ$ which are not all zeros, such that
	$$
	\sum_{i=1}^{d+1}p_i\cdot(A^{ik}n-n)=0\in\ZZ^d.
	$$
	Finally, $\bar\alpha(\cdot)=\alpha(0,\cdot)$ is a $\ZZ^d$-action on $\cF^c(0)$, so we have
	$$
	\prod_{i=1}^{d+1}
	\big[\alpha(ik,0)\circ\alpha(0,n)\circ
	\alpha(-ik,0)\circ\alpha(0,-n)\big]^{p_i}=\alpha(0,0)
	={\rm Id}_{\cF^c(0)}.
	$$
\end{proof}

\begin{remark}
	In general, since the foliation $\cF^{su}$ is only H\"older continuous, the action $\alpha$ is also merely H\"older.
\end{remark}

Now recall that $\cF^c(0)$ is subfoliated by $\cF^{ws}$ and $\cF^{wu}$, however we don't have $\alpha$-invariance property for these foliations.
\begin{lemma}
	\label{lemma_32}
	If $\bar \alpha(n)\big(\cF^{wu}(0)\big)\subseteq\cF^{wu}\big(\bar\alpha(n)(0)\big)$ for all $n\in\ZZ^d$, then the distribution $E^{ss}\oplus E^{wu}\oplus E^{uu}$ is integrable, which implies $h(\cF^{ss})=\cL^{ss}$ and $h(\cF^{uu})=\cL^{uu}$.
\end{lemma}
\begin{proof} 
	Let's assume that $E^{ss}\oplus E^{wu}\oplus E^{uu}$ does not integrate to a foliation. Then the foliations $\cF^{su}$ and $\cF^{wu}$ do not integrate together and, consequently, we can find points $x\in\RR^d$, $a\in \cF^{su}(x)$ and $b\in\cF^{wu}(a)$ such that the point $y=\cF^c(x)\cap \cF^{su}(b)$ does not lie on the leaf $\cF^{wu}(x)\subset \cF^c(x)$. Such configuration of points is robust, that is if we move $x$ a little bit and all the other points accordingly then we have the same configuration with the property $y\notin \cF^{wu}(x)$ persisting. This is clear from continuity of foliations.
	
	Now recall that $\cF^{wu}$ is conjugate to a linear irrational foliation and, hence $\cup_{n\in\ZZ^d}\cF^{wu}(n)$ is dense in $\RR^d$. Hence, by slightly moving $x$ we may assume that $x\in\cF^{wu}(m)$ for some $m\in\ZZ^d$ and then by applying the tranlation by $-m$ to all the points, we may in fact assume that $x\in\cF^{wu}(0)$.
	
	Using denseness of $\cup_{n\in\ZZ^d}\cF^{wu}(n)$ we can find a sequence $\{n_i: i\ge 1\}\subset\ZZ^d$, such that the leaves $\cF^{wu}(n_i)$ converge to $\cF^{wu}(a)$ as $i\to\infty$. Because  $y\notin \cF^{wu}(x)$ we have $\H^{su}_{a,x}(\cF^{wu}(a))\not\subset \cF^{wu}(x)$, where $\H^{su}_{a,x}\colon \cF^c(a)\to\cF^c(x)$ is the holonomy map given by sliding along the leaves of $\cF^{su}$. Therefore, for a sufficiently large $i$ we also have $\H^{su}_{a,x}\big(\cF^{wu}(n_i)\big)\not\subset \cF^{wu}(x)$, but that precisely means that we have
	$$
	\bar\alpha(n_i)\big(\cF^{wu}(0)\big)\not\subset \cF^{wu}\big(\bar\alpha(n_i)(0)\big)
	$$ 
	yielding a contradiction.
	
	Finally, if $E^{ss}\oplus E^{wu}\oplus E^{uu}$ is integrable to a foliation $\cF$, then $h(\cF)$ is sub-foliated by the minimal linear foliation $\cL^{ws}=h(\cF^{ws})$. Lemma \ref{lem:linear-foli} shows that $h(\cF)$ is linear and $h(\cF^{ss})=h(\cF)\cap h(\cF^s)$ which is also linear. So we must have $h(\cF^{ss})=\cL^{ss}$. Again, this implies $h(\cF^{su})$ is linear, and we have $h(\cF^{uu})=\cL^{uu}$.
\end{proof}

\subsection{Double transverse position in 2-dimensional center}

In this subsection, we show that under the assumptions of Theorem \ref{thm:C2}, if $E^{ss}\oplus E^{uu}$ is integrable and the corresponding action $\alpha\colon \Gamma\times\cF^c(0)\to\cF^c(0)$ does not preserve $\cF^{ws}$ and $\cF^{wu}$ in $\cF^c(0)$, then there exists some $m\in\ZZ^d$ such that $\bar\alpha(m)$ puts them topologically transverse to themselves. 

Here $\cF^{ws}(0)$ and $\cF^{wu}(0)$ are imbedded lines in $\cF^c(0)$ which both separate $\cF^c(0)$ into two connected components. For some $n\in\ZZ^d$, we say that $\bar\alpha(n)\big(\cF^{ws}(0)\big)$ is topologically transverse to $\cF^{ws}(0)$ if $\bar\alpha(n)\big(\cF^{ws}(0)\big)$ intersect both connected components of $\cF^c(0)\setminus\cF^{ws}(0)$ (and similarly for $\cF^{wu}(0)$).

\begin{lemma}\label{lem:intersection}
	For every $m\in\ZZ^d$, we have
	\begin{itemize}
		\item $\bar\alpha(m)\big(\cF^{ws}(0)\big)$ intersects $\cF^{wu}(0)$ topologically transversely at a unique point in $\cF^c(0)$;
		\item $\bar\alpha(m)\big(\cF^{wu}(0)\big)$ intersects $\cF^{ws}(0)$ topologically transversely at a unique point in $\cF^c(0)$.
	\end{itemize}
\end{lemma}

\begin{proof}
	For every $m\in\ZZ^d$, $\bar\alpha(m)\big(\cF^{ws}(0)\big)$ is the image of $\cF^{ws}(m)$ under the holonomy map induced by the foliation $\cF^{su}$: 
	$$
	\H^{su}_{m,\bar\alpha(m)(0)}:
	\cF^c(m)\longrightarrow \cF^c(0),
	\qquad \text{and} \qquad
	\bar\alpha(m)\big(\cF^{ws}(0)\big)= 
	\H^{su}_{m,\bar\alpha(m)(0)}\big(\cF^{ws}(m)\big).
	$$ 
	Let $\xi$ be the unique intersection point of $\cF^{ss}(m)$ and $\cF^{uu}(\bar\alpha(m)(0))$ in $\cF^{su}(m)$. Then
	$$
	\H^{su}_{m,\bar\alpha(m)(0)}=
	\H^{uu}_{\xi,\bar\alpha(m)(0)}\circ
	\H^{ss}_{m,\xi}:~
	\cF^c(m)\longrightarrow \cF^c(0).
	$$
	
	Since $E^{ss}\oplus E^{ws}$ is integrable, the holonomy map  $\H^{ss}_{m,\xi}:\cF^c(m)\to\cF^c(\xi)$ preserves the weak foliation $\cF^{ws}$. Hence 
	$\H^{ss}_{m,\xi}\big(\cF^{ws}(m)\big)=\cF^{ws}(\xi)$. On the other hand, $E^{wu}\oplus E^{uu}$ being integrable implies that $\H^{uu}_{\xi,\bar\alpha(m)(0)}$ preserves $\cF^{wu}$. Thus 
	$$
	\left(\H^{uu}_{\xi,\bar\alpha(m)(0)}\right)^{-1}\big(\cF^{wu}(0)\big)
	=\cF^{wu}(\zeta)
	\qquad \text{for~some} \quad
	\zeta\in\cF^c(\xi).
	$$
	
	Since $\cF^{ws}(\xi)$ transversely intersects $\cF^{wu}(\zeta)$ at a unique point in $\cF^c(\xi)$, and the holonomy map $\H^{uu}_{\xi,\bar\alpha(m)(0)}:\cF^c(\xi)\to\cF^c(0)$ is a homeomorphism, we have
	$$
	\bar\alpha(m)\big(\cF^{ws}(0)\big)=
	H^{uu}_{\xi,\bar\alpha(m)(0)}\big(\cF^{ws}(\xi)\big)
	\qquad \text{intersects} \qquad
	\cF^{wu}(0)=H^{uu}_{\xi,\bar\alpha(m)(0)}\big(\cF^{wu}(\zeta)\big)
	$$
	topologically transversely at a unique point in $\cF^c(0)$. The proof for the second item is the same.
\end{proof}

The following proposition is the key observation for the case when both $E^{ws}\oplus E^{uu}$ and $E^{ss}\oplus E^{wu}$ are not integrable.

\begin{proposition}\label{prop:C2-general}
	Let $A$ and $f$ satisfy the assumption of Theorem \ref{thm:C2}. Assume $E^{ss}\oplus E^{uu}$ is integrable and let $\alpha:\Gamma\times\cF^c(0)\rightarrow\cF^c(0)$ be the corresponding action induced by $\cF^{su}$. If both $E^{ws}\oplus E^{uu}$ and $E^{ss}\oplus E^{wu}$ are not integrable, then there exists $m\in\ZZ^d$, such that
	\begin{itemize}
		\item $\bar\alpha(m)\big(\cF^{ws}(0)\big)$ intersects $\cF^{ws}(0)$ topologically transversely in $\cF^c(0)$;
		\item $\bar\alpha(m)\big(\cF^{wu}(0)\big)$ intersects $\cF^{wu}(0)$ topologically transversely in $\cF^c(0)$.
	\end{itemize}  
\end{proposition}

It will be important to keep in mind that foliations $\cF^c$, $\cF^{wu}$ and $\cF^{ws}$ are conjugate to corresponding irrational linear foliations $\cL^c$, $\cL^{wu}$ and $\cL^{ws}$, respectively. In particular, if points $x$ and $y$ are close then corresponding leaves  $\cF^{wu}(x)$ and $\cF^{wu}(y)$ also stay uniformly close. We begin with a lemma.

\begin{lemma} There exists a vector $m_1\in\ZZ^d$ such that $\bar\alpha(m_1)(\cF^{wu}(0))$ has topologically transverse  intersection with $\cF^{wu}(0)$.
\end{lemma}

\begin{proof}
	By Lemma~\ref{lemma_32} we have $\bar\alpha(n_1)(\cF^{wu}(0))\not\subset \cF^{wu}(\bar\alpha(n_1)(0))$ for some $n_1\in\ZZ^d$. 
	Hence we have that $\bar\alpha(n_1)(\cF^{wu}(0))$ intersects some $\cF^{wu}(x)$ topologically transversely. 
	
	Assume that $n\in\ZZ^d$ is $\delta$-close to $\cF^c(0)$ then $\bar\alpha(n)(\cF^{wu}(0))$ is contained $\eps(\delta)$-neighborhood of the leaf $\cF^{wu}(\bar\alpha(n)(0))$, where $\eps(\delta)\to 0$ as $\delta\to 0$. 
	Indeed, $\bar\alpha(n)(\cF^{wu}(0))$ is the image of $\cF^{wu}(n)$ under $C\delta$-short holonomy $\H^{su} \colon \cF^c(n)\to \cF^c(0)$ induced by $\cF^{su}$, and hence is contained is small neighborhoods of $\cF^{wu}(n)$ and also of $\cF^{wu}(\bar\alpha(n)(0))$ as well. More precisely, for any $\delta>0$ there exists an $\eps>0$ such that if
	$$
	d(n,\cF^c(0))<\eps, \qquad \text{then} \qquad
	\alpha(n)(\cF^{wu}(0))\subset U_\delta(\cF^{wu}(\bar\alpha(n))),
	$$
	where $U_\delta(\cF^{wu}(\bar\alpha(n)))$ is the $\delta$-neighborhood of $\cF^{wu}(\bar\alpha(n))$.
	
	Now pick an $x \in\cF^c(0)$ such that $\bar\alpha(n_1)(\cF^{wu}(0)))$ intersects $\cF^{wu}(x)$ transversely. Pick points $a,b\in \bar\alpha(n_1)(\cF^{wu}(0)))$ which lie in different connected components of $\cF^c(0)\backslash \cF^{wu}(x)$. We can pick $n_1'\in\ZZ^d$ so that $\cF^{wu}(n_1')$ is arbitrarily close to $\cF^{wu}(x)$. Then, according to the above discussion $\bar\alpha(n_1')(\cF^{wu}(0))$ is also arbitrarily close to $\cF^{wu}(x)$. Hence, for an appropriate choice of $n_1'$, we have that $a$ and $b$ also lie in different connected components of $\cF^c(0)\backslash \bar\alpha(n_1')(\cF^{wu}(0))$. This means that $\bar\alpha(n_1)(\cF^{wu}(0))$ and $\bar\alpha(n_1')(\cF^{wu}(0))$ have transverse topological intersection, which is equivalent to  $\bar\alpha(n_1-n_1')(\cF^{wu}(0))$ transversally intersecting $\cF^{wu}(0)=\bar\alpha(0)(\cF^{wu}(0))$. Hence we obtain the lemma by setting $m_1=n_1-n_1'$.
\end{proof}

By the same argument we can adjust the value of $m_2$ such that $\bar\alpha(m_2)(\cF^{ws}(0))$ intersects $\cF^{ws}(0)$ topologically transversely. Now we can prove Proposition \ref{prop:C2-general}.

\begin{proof}[Proof of Proposition \ref{prop:C2-general}]
	We want to show that for every sufficiently large $k\in\NN$, let 
	$$
	m=A^km_2-A^{-k}m_1\in\ZZ^d,
	$$ 
	then we have that $\bar\alpha(m)\big(\cF^{ws}(0)\big)$ intersects $\cF^{ws}(0)$ topologically transversely, and $\bar\alpha(m)\big(\cF^{wu}(0)\big)$ intersects $\cF^{wu}(0)$ topologically transversely.
	
	Recall that $\bar\alpha(m_1)\big(\cF^{ws}(0)\big)$ is given by the image of $\cF^{ws}(m_1)$ under the $\cF^{su}$-holonomy 
	$$
	\H^{su}_{m_1,\bar\alpha(m_1)(0)}:
	\cF^c(m_1)\longrightarrow
	\cF^c(\bar\alpha(m_1)(0))=\cF^c(0).
	$$ 
	So there exists a constant $C>0$, such that $\bar\alpha(m_1)\big(\cF^{ws}(0)\big)$ is contained in the $C$-neighborhood of $\cF^{ws}(0)$. By the same taken $-m_1$, we can assume $\cF^{ws}(0)$ is also contained in the $C$-neighborhood of $\bar\alpha(m_1)\big(\cF^{ws}(0)\big)$.

	Since $F^{-1}$ uniformly contracts $\cF^{wu}$ and uniformly expands $\cF^{ws}$ as $k\to+\infty$, we have 
	$$
	\bar\alpha(A^{-k}m_1)\big(\cF^{ws}(0)\big)=
	F^{-k}\circ\bar\alpha(m_1)\circ F^k\big(\cF^{ws}(0)\big)=
	F^{-k}\circ\bar\alpha(m_1)\big(\cF^{ws}(0)\big)
	~\longrightarrow~\cF^{ws}(0)
	$$ 
	in the Hausdorff topology on $\cF^c(0)$. The same argument shows that 
	$$
	\bar\alpha(A^km_2)\big(\cF^{wu}(0)\big)~\longrightarrow~
	\cF^{wu}(0)
	\qquad \text{as} \qquad
	k\to+\infty
	$$
	in Hausdorff topology of $\cF^c(0)$. 
	
	On the other hand, $\bar{\alpha}(m_1)\big(\cF^{wu}(0)\big)$ intersects $\cF^{wu}(0)$ topologically transversely in $\cF^c(0)$. This implies $\bar{\alpha}(m_1)\big(\cF^{wu}(0)\big)$ intersects both connected components of $\cF^c(0)\setminus\cF^{wu}(0)$. Since 
	$$
	\bar\alpha(A^{2k}m_2)\big(\cF^{wu}(0)\big)
	~\longrightarrow~ \cF^{wu}(0)
	\qquad \text{as } \qquad k\to+\infty,
	$$
	$\bar{\alpha}(m_1)\big(\cF^{wu}(0)\big)$ intersects both connected components of $\cF^c(0)\setminus\bar\alpha(A^{2k}m_2)\big(\cF^{wu}(0)\big)$ for every sufficiently large $k$. That is, $\bar{\alpha}(m_1)\big(\cF^{wu}(0)\big)$ intersects $\bar\alpha(A^{2k}m_2)\big(\cF^{wu}(0)\big)$ transversely.
    By applying the action of $F^{-k}$ on $\cF^c(0)$, we have $\bar{\alpha}(A^{-k}m_1)\big(\cF^{wu}(0)\big)$ intersects $\bar\alpha(A^km_2)\big(\cF^{wu}(0)\big)$ topologically transversely for every sufficiently large $k\in\NN$.
	
	The same argument shows that $\bar\alpha(A^km_2)\big(\cF^{ws}(0)\big)$ intersects $\bar{\alpha}(A^{-k}m_1)\big(\cF^{ws}(0)\big)$ transversely in $\cF^c(0)$ for every sufficiently large $k\in\NN$.
	Hence we can let
	$$
	m=A^km_2-A^{-k}m_1\in\ZZ^d,
	$$ 
	Then it satisfies $\bar\alpha(m)\big(\cF^{ws}(0)\big)$ intersects $\cF^{ws}(0)$ topologically  transversely and $\bar\alpha(m)\big(\cF^{wu}(0)\big)$ intersects $\cF^{wu}(0)$ topologically transversely in $\cF^c(0)$ for sufficiently large $k$.
\end{proof}

\subsection{The ping-pong argument}\label{subsec:ping-pong}

In this section, we prove Theorem \ref{thm:C2} from Lemma \ref{lem:intersection} and Proposition \ref{prop:C2-general}. Recall that in the the case when both $E^{ss}\oplus E^{wu}$ and $E^{ws}\oplus E^{uu}$ are not integrable we have the following. There exists $m\in\ZZ^d$, such that 
\begin{itemize}
	\item $\bar\alpha(m)(\cF^{ws}(0))$ topologically transversely intersects both $\cF^{ws}(0)$ and $\cF^{wu}(0)$ in $\cF^c(0)$;
	\item $\bar\alpha(m)(\cF^{wu}(0))$ topologically transversely intersects both $\cF^{ws}(0)$ and $\cF^{wu}(0)$ in $\cF^c(0)$.
\end{itemize}

Now denote by $F=\alpha(1,0):\cF^c(0)\to\cF^c(0)$ the restriction $F:\RR^d\to\RR^d$ to $\cF^c(0)$, and let 
$$
G~=~\bar\alpha(m)\circ\alpha(1,0)\circ\bar\alpha(-m)
~=~\bar\alpha(m)\circ F\circ\bar\alpha(-m):
~\cF^c(0)\to\cF^c(0).
$$
Here $F$ is uniformly contracting on $\cF^{ws}$ and expanding on $\cF^{wu}$, and $G$ is uniformly contracting along $\bar\alpha(\cF^{ws})$ and expanding along $\bar\alpha(\cF^{wu})$. Notice that since the action $\bar\alpha$ is just H\"older continuous, $G$ is globally a ``topological hyperbolic saddle'' on $\cF^c(0)$.

\begin{lemma}\label{lem:FG-solve}
	For every $k\in\NN$, there exist integers $a_1,\cdots,a_l,b_1\cdots,b_l\in\ZZ$ with only $a_l$ and $b_1$ allowed to be zeros, such that
	$$
	F^{a_l\cdot k}\circ G^{b_l\cdot k}\circ 
	F^{a_{l-1}\cdot k}\circ G^{b_{l-1}\cdot k}\circ\cdots\cdots\circ
	F^{a_1\cdot k}\circ G^{b_1\cdot k}
	={\rm Id}_{\cF^c(0)}.
	$$
\end{lemma}

\begin{proof}
	From Lemma \ref{lem:solve}, for every $k\in\NN$ and $i=1\cdots,d+1$, we have
	$$
	F^{ik}\circ G^{-ik}=\alpha(ik,0)\circ \bar\alpha(m)\circ\alpha(-ik,0)\circ\bar\alpha(-m)
	=\bar\alpha(A^{ik}m-m).
	$$
	Lemma \ref{lem:solve} shows that there exist $p_1,p_2,\cdots,p_{d+1}\in\ZZ$ which are not all zeros, such that
	$$
	\prod_{i=1}^{d+1}\left(F^{ik}\circ G^{-ik}\right)^{p_i}=
	\prod_{i=1}^{d+1}
	\big[\alpha(ik,0)\circ\bar\alpha(m)\circ
	\alpha(-ik,0)\circ\bar\alpha(-m)\big]^{p_i}=\alpha(0,0)
	={\rm Id}_{\cF^c(0)}.
	$$
	Since for every $i=1,\cdots,d$ and $j>i$,
	$$
	\big(F^{jk}\circ G^{-jk}\big)^{-1}\circ 
	\big(F^{ik}\circ G^{-ik}\big)=
	G^{-jk}\circ F^{(i-j)k}\circ G^{-ik},
	$$
	we have integers $a_1,\cdots,a_l,b_1\cdots,b_l\in\ZZ$ (only $a_l$ and $b_1$ can be zeros), such that
	$$
	F^{a_l\cdot k}\circ G^{b_l\cdot k}\circ 
	F^{a_{l-1}\cdot k}\circ G^{b_{l-1}\cdot k}\circ\cdots\cdots\circ
	F^{a_1\cdot k}\circ G^{b_1\cdot k}
	={\rm Id}_{\cF^c(0)}.
	$$
\end{proof}

Finally, we can prove Theorem \ref{thm:C2}.

\begin{proof}[Proof of Theorem \ref{thm:C2}]
	If $f$ is spectrally rigid along $E^{ws}$ and $E^{wu}$, then Theorem \ref{thm:local} shows that $h(\cF^{ss})=\cL^{ss}$ and $h(\cF^{uu})=\cL^{uu}$. So $\cF^{ss}$ and $\cF^{uu}$ are jointly integrable to a foliation $\cF^{su}$, which is tangent to $E^{ss}\oplus E^{uu}$.
	
	Conversely, we want to show $E^{ss}\oplus E^{uu}$ is integrable implies spectral rigidity along $E^{ws}$ and $E^{wu}$. If one of $E^{ss}\oplus E^{wu}$ and $E^{ws}\oplus E^{uu}$ is integrable, then Lemma \ref{lemma_31} shows that $h(\cF^{ss})=\cL^{ss}$ and $h(\cF^{uu})=\cL^{uu}$. Again Theorem \ref{thm:local} proves the spectral rigidity along $E^{ws}$ and $E^{wu}$, which finishes the proof. So we only need to consider the case that both $E^{ss}\oplus E^{wu}$ and $E^{ws}\oplus E^{uu}$ are not integrable. In the rest of the proof, we use ping-pong argument to show this case couldn't happen and finishes the proof.
	
	\vskip2mm
	
	Following from Lemma \ref{lem:intersection} and Proposition \ref{prop:C2-general}, if both $E^{ss}\oplus E^{wu}$ and $E^{ws}\oplus E^{uu}$ are not integrable, then there exists $m\in\ZZ^d$ such that both  $\bar{\alpha}(m)(\cF^{ws}(0))$ and $\bar{\alpha}(m)(\cF^{wu}(0))$ are topologically transverse to both $\cF^{ws}(0)$ and $\cF^{ws}(0)$ in $\cF^c(0)$.

	Here we have two families of foliations $\cF^{ws}\times\cF^{wu}$ and $\bar\alpha(m)(\cF^{ws})\times\bar\alpha(m)(\cF^{wu})$ with global product structures on $\cF^c(0)$. Moreover, $F$ uniformly contracts $\cF^{ws}$ and expands $\cF^{wu}$; $G$ uniformly contracts $\bar\alpha(m)(\cF^{ws})$ and expands $\bar\alpha(m)(\cF^{wu})$. 
	
	We can take 
	\begin{itemize}
		\item two large compact segments $D^s_F\subset\cF^{ws}(0)$ and $D^u_F\subset\cF^{wu}(0)$, which are disks centered at $0$ with radius large enough in $\cF^{ws}(0)$ and $\cF^{wu}(0)$ respectively, such that both $D^s_F$ and $D^u_F$ intersect both components of $\cF^c(0)\setminus\alpha(m)(\cF^{ws}(0))$ and both components of $\cF^c(0)\setminus\alpha(m)(\cF^{wu}(0))$;
		\item two large compact segments $D^s_G\subset\alpha(m)(\cF^{ws}(0))$ and $D^u_G\subset\alpha(m)(\cF^{wu}(0))$, which are disks centered at $\alpha(m)(0)$ with radius large enough in $\alpha(m)(\cF^{ws}(0))$ and $\alpha(m)(\cF^{wu}(0))$ respectively, such that both $D^s_G$ and $D^u_G$ intersect both components of $\cF^c(0)\setminus\cF^{ws}(0)$ and both components of $\cF^c(0)\setminus\cF^{wu}(0)$;
	\end{itemize} 
    such that there exist four points $x,y,z,w$ satisfying
    $$
    x\in{\rm Int}(D^s_F)\cap {\rm Int}(D^s_G),\quad
    y\in{\rm Int}(D^s_F)\cap {\rm Int}(D^u_G),\quad
    z\in{\rm Int}(D^u_F)\cap {\rm Int}(D^s_G),\quad
    w\in{\rm Int}(D^u_F)\cap {\rm Int}(D^u_G).
    $$
    
    Moreover, since $ x\in{\rm Int}(D^s_F)\cap {\rm Int}(D^s_G)$, we can take two compact small segments $D_F(x)\subset{\rm Int}(D^s_F)$ and $D_G(x)\subset{\rm Int}(D^s_G)$ containing $x$ in their interiors and $\delta_x>0$, such that 
    for every compact segment $D\subset\cF^c(0)$,
    \begin{itemize}
    	\item if $d_H(D,D_F(x))<\delta_x$, then $D$ intersects both components of $\cF^c(0)\setminus\alpha(m)(\cF^{ws}(0))$ and $D\cap D_G(x)\neq\emptyset$;
    	\item if $d_H(D,D_G(x))<\delta_x$, then $D$ intersects both components of $\cF^c(0)\setminus\cF^{ws}(0)$ and $D\cap D_F(x)\neq\emptyset$.
    \end{itemize}
    Here $d_H(\cdot,\cdot)$ is the Hausdorff distance on $\cF^c(0)$.
    
    \begin{figure}\label{fig:TIL}
    	\centering
    	\includegraphics[width=10cm]{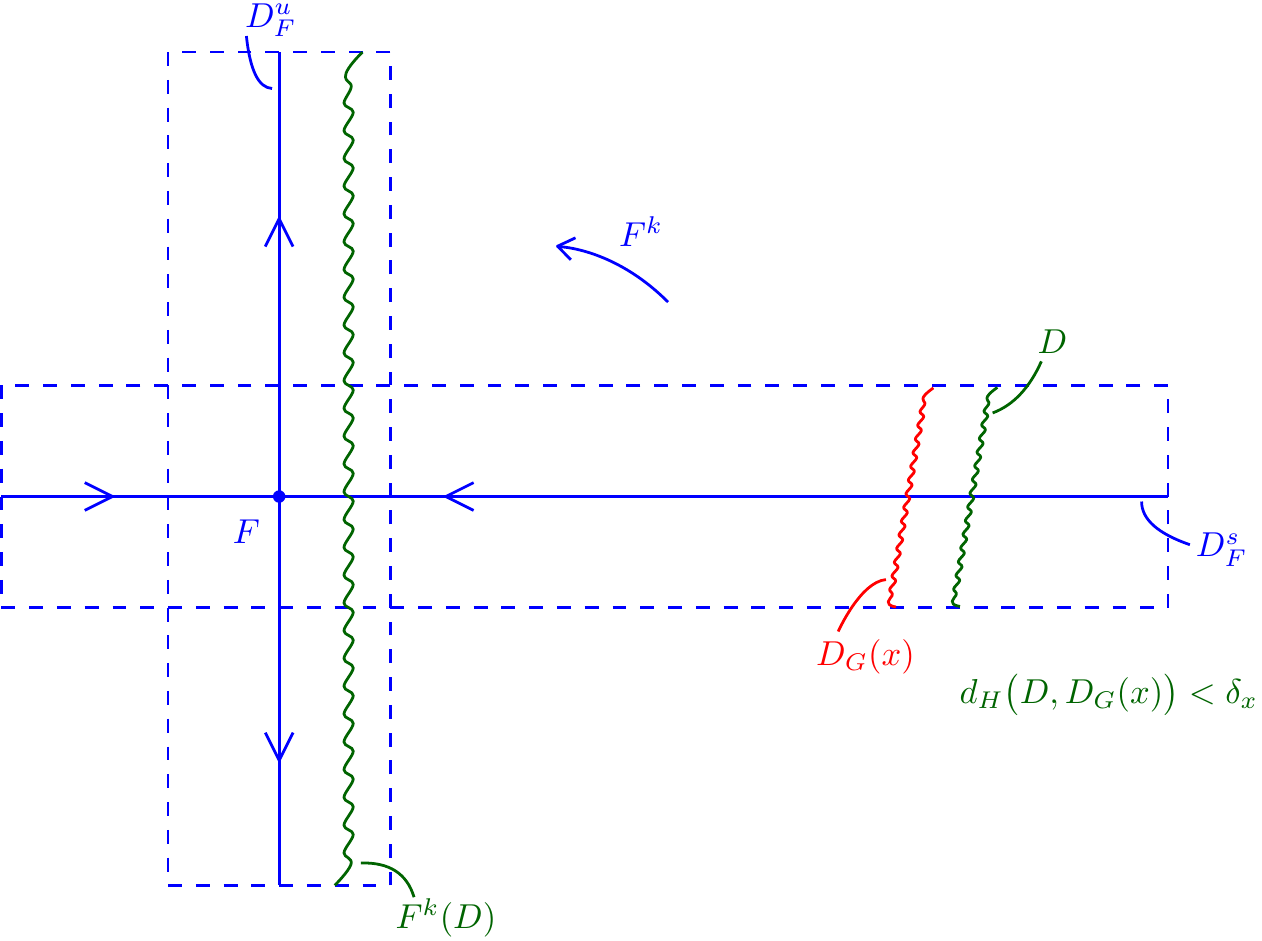}
    	\caption{Topological Inclination Lemma}
    \end{figure}
    
    Since $F$ uniformly contracts $\cF^{ws}$ and expands $\cF^{wu}$; $G$ uniformly contracts $\bar\alpha(m)(\cF^{ws})$ and expands $\bar\alpha(m)(\cF^{wu})$, 
    the fact that $D_G(x)\subset\alpha(m)(\cF^{ws}(0))$ and $D_F(x)\subset\cF^{ws}(0)$ implies we have the following claim, which is a topological Inclination Lemma, see Figure 1.
    
    \begin{claim}[Topological Inclination Lemma]\label{clm:TIL}
    	For every $\e>0$, there exists $K=K(x,\e)\in\NN$, such that for every $k>K$,
    	\begin{itemize}
    		\item if $d_H(D,D_F(x))<\delta_x$, then $G^k(D)$ contains a segment $D'$ such that $d_H(D',D^u_G)<\e$;
    		\item  if $d_H(D,D_G(x))<\delta_x$, then $F^k(D)$ contains a segment $D'$ such that $d_H(D',D^u_F)<\e$.
    	\end{itemize}
    \end{claim}

    Similarly, we can take compact small segments 
    \begin{align*}
    &D_F(y)\subset{\rm Int}(D^s_F), \quad
    D_G(y)\subset{\rm Int}(D^u_G); \qquad
    D_F(z)\subset{\rm Int}(D^u_F), \quad
    D_G(z)\subset{\rm Int}(D^s_G); \\
    &D_F(w)\subset{\rm Int}(D^u_F), \quad
    D_G(w)\subset{\rm Int}(D^u_G);
    \end{align*}
    and find constants
    $\delta_y,\delta_z,\delta_w$ for $y,z,w$, such that they satisfy the corresponding transversality conditions. Here all these  $D_{F/G}(*)$ for every
    $*=x,y,z,w$ are small segments comparing with every $D_{F/G}^{s/u}$.
    
    Let 
    $$
    \delta_0=\min\{\delta_x,\delta_y,\delta_z,\delta_w\}
    $$
    then there exists $\e_0>0$, such that 
    \begin{itemize}
    	\item if $d_H(D,D^s_F)<\e_0$, then $D$ contains two small segments $D',D''$ such that $d_H(D',D_F(x))<\delta_0$ and $d_H(D'',D_F(y))<\delta_0$;
    	\item if $d_H(D,D^u_F)<\e_0$, then $D$ contains two small segments $D',D''$ such that $d_H(D',D_F(z))<\delta_0$ and $d_H(D'',D_F(w))<\delta_0$;
    	\item if $d_H(D,D^s_G)<\e_0$, then $D$ contains two small segments $D',D''$ such that $d_H(D',D_G(x))<\delta_0$ and $d_H(D'',D_G(z))<\delta_0$;
    	\item if $d_H(D,D^u_G)<\e_0$, then $D$ contains two small segments $D',D''$ such that $d_H(D',D_G(y))<\delta_0$ and $d_H(D'',D_G(w))<\delta_0$.
    \end{itemize}
    Recall that $D_{F/G}(*)$ for every
    $*=x,y,z,w$ are small segments comparing with every $D_{F/G}^{s/u}$. So we can take a smaller $\e_0$ if necessary, such that $D_{F/G}(*)$, 
    $*=x,y,z,w$ is not equal to any subset of $\cF^c(0)$ that contains a segment which is $\e_0$-close to one of $D_{F/G}^{s/u}$ in the Hausdorff topology of $\cF^c(0)$.

    From the definition of $\e_0$, we can apply Claim \ref{clm:TIL} of the topological Inclination Lemma repeatedly to $\e_0>0$ and  $k_0\in\NN$, such that for every $k\geq k_0$, we have
    \begin{itemize}
    	\item if $d_H(D,D_F(x))<\delta_0$ or $d_H(D,D_F(z))<\delta_0$, then $G^k(D)$ contains two segments $D',D''$ such that 
    	$$
    	d_H\big(D',D_G(y)\big)<\delta_0,
    	\qquad \text{and} \qquad
    	d_H\big(D'',D_G(w)\big)<\delta_0;
    	$$
    	\item if $d_H(D,D_F(y))<\delta_0$ or $d_H(D,D_F(w))<\delta_0$, then $G^{-k}(D)$ contains two segments $D',D''$ such that 
    	$$
    	d_H\big(D',D_G(x)\big)<\delta_0,
    	\qquad \text{and} \qquad
    	d_H\big(D'',D_G(z)\big)<\delta_0;
    	$$
    	\item if $d_H(D,D_G(x))<\delta_0$ or $d_H(D,D_G(y))<\delta_0$, then $F^k(D)$ contains two segments $D',D''$ such that 
    	$$
    	d_H\big(D',D_F(z)\big)<\delta_0,
    	\qquad \text{and} \qquad
    	d_H\big(D'',D_F(w)\big)<\delta_0;
    	$$
    	\item if $d_H(D,D_G(z))<\delta_0$ or $d_H(D,D_G(w))<\delta_0$, then $F^{-k}(D)$ contains two segments $D',D''$ such that 
    	$$
    	d_H\big(D',D_F(x)\big)<\delta_0,
    	\qquad \text{and} \qquad
    	d_H\big(D'',D_F(y)\big)<\delta_0.
    	$$
    \end{itemize} 
    
    We apply Lemma \ref{lem:FG-solve} to this $k_0\in\NN$: there exist integers $a_1,\cdots,a_l,b_1\cdots,b_l\in\ZZ$ with only $a_l$ and $b_1$ allowed to be zeros, such that
    $$
    F^{a_l\cdot k_0}\circ G^{b_l\cdot k_0}\circ 
    F^{a_{l-1}\cdot k_0}\circ G^{b_{l-1}\cdot k_0}\circ\cdots\cdots\circ
    F^{a_1\cdot k_0}\circ G^{b_1\cdot k_0}
    ={\rm Id}_{\cF^c(0)}.
    $$
    We will assume $a_l,b_1\neq0$. The proof of other cases is the same.  
    
    If $b_1>0$, we take $D_0=D_F(x)$, then $G^{b_1\cdot k_0}(D_0)$ contains two segments $D_0',D_0''$ such that
    $$
    d_H\big(D_0',D_G(y)\big)<\delta_0,
    \qquad \text{and} \qquad
    d_H\big(D_0'',D_G(w)\big)<\delta_0.
    $$
    If $b_1<0$, we take $D_0=D_F(y)$, then $G^{b_1\cdot k_0}(D_0)$ contains two segments $D_0',D_0''$ such that
    $$
    d_H\big(D_0',D_G(x)\big)<\delta_0,
    \qquad \text{and} \qquad
    d_H\big(D_0'',D_G(z)\big)<\delta_0.
    $$
    In both cases, we can see that $D_0'$ is transverse to $\cF^{ws}(0)$ and $\delta_0$-close to $D_G(x)$ or $D_G(y)$ in Hausdorff topology; and $D_0''$ is transverse to $\cF^{wu}(0)$ and $\delta_0$-close to $D_G(z)$ or $D_G(w)$ in  Hausdorff topology, see Figure 2 (we have presented some of the segments as thin rectangles for clarity).
    
    \begin{figure}
    	\centering
    	\includegraphics[width=10cm]{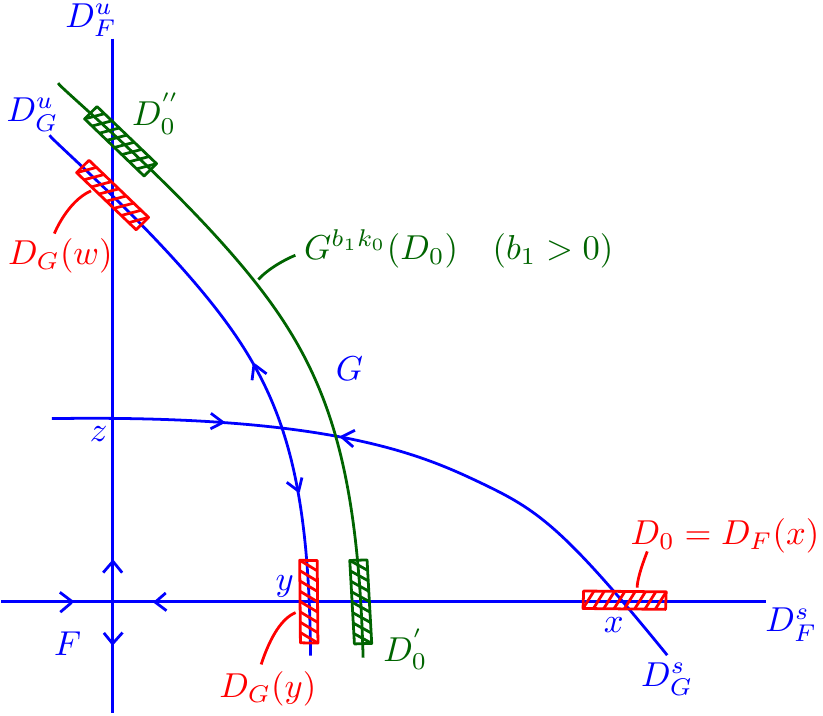}
    	\caption{Ping-pong Argument}
    \end{figure}
    
    Now we consider the action of $F^{a_1\cdot k_0}$. If $a_1>0$, we take $D_1=D_0'$, then $F^{a_1\cdot k_0}(D_1)$ contains two segments $D_1',D_1''$ such that
    $$
    d_H\big(D_1',D_F(z)\big)<\delta_0,
    \qquad \text{and} \qquad
    d_H\big(D_1'',D_F(w)\big)<\delta_0.
    $$ 
    If $a_1<0$, we take $D_1=D_0''$, then $F^{a_1\cdot k_0}(D_1)$ contains two segments $D_1',D_1''$ such that
    $$
    d_H\big(D_1',D_F(x)\big)<\delta_0,
    \qquad \text{and} \qquad
    d_H\big(D_1'',D_F(y)\big)<\delta_0.
    $$ 
    In both cases, we can see that $D_1'$ is transverse to $\bar\alpha(m)(\cF^{ws}(0))$ and $\delta_0$-close to $D_F(x)$ or $D_F(z)$; and $D_1''$ is transverse to $\bar\alpha(m)(\cF^{wu}(0))$ and $\delta_0$-close to $D_F(y)$ or $D_F(w)$.
    
    We repeat this procedure to obtain that the segment $D_{2l-2}'$ is transverse to $\cF^{ws}(0)$ and $\delta_0$-close to $D_G(x)$ or $D_G(y)$; and the segment $D_{2l-2}''$ is transverse to $\cF^{wu}(0)$ and $\delta_0$-close to $D_G(z)$ or $D_G(w)$. 
    
    Finally, we consider the action of $F^{a_l\cdot k_0}$. 
    \begin{itemize}
    	\item If $a_l>0$, we take $D_{2l-1}=D_{2l-2}'$, then $F^{a_l\cdot k_0}(D_{2l-1})$ contains a segment $D$ which is $\e_0$-close to $D_F^u$ in Hausdorff topology.
    	\item If $a_l<0$, we take $D_{2l-1}=D_{2l-2}''$, then $F^{a_l\cdot k_0}(D_{2l-1})$ contains a segment $D$ which is $\e_0$-close to $D_F^s$ in Hausdorff topology.
    \end{itemize}

    So the set
    $$
    F^{a_l\cdot k_0}\circ G^{b_l\cdot k_0}\circ 
    F^{a_{l-1}\cdot k_0}\circ G^{b_{l-1}\cdot k_0}\circ\cdots\cdots\circ
    F^{a_1\cdot k_0}\circ G^{b_1\cdot k_0}(D_0)
    $$ 
    contains a segment $D_{2l-1}$ which is $\e_0$-close to $D_F^s$ or $D_F^u$. From the definition of $\e_0$,  we must have
    $$
    F^{a_l\cdot k_0}\circ G^{b_l\cdot k_0}\circ 
    F^{a_{l-1}\cdot k_0}\circ G^{b_{l-1}\cdot k_0}\circ\cdots\cdots\circ
    F^{a_1\cdot k_0}\circ G^{b_1\cdot k_0}(D_0)
    ~\neq~D_0.
    $$
    This contradicts to Lemma \ref{lem:FG-solve}. So we must actually have that at least one of $E^{ss}\oplus E^{wu}$ and $E^{ws}\oplus E^{uu}$ is integrable. This finishes the proof of Theorem \ref{thm:T4general}.   
\end{proof}

\begin{remark}
	We actually proved that $\{F^k,G^k\}$ generates a free group on two generators for every $k\geq k_0$, cf.~\cite{HX}. 
\end{remark}

\bibliographystyle{plain}

\vskip5mm

\flushleft{\bf Andrey Gogolev} \\
Department of Mathematics, The Ohio State University, Columbus, OH 43210, USA\\
\textit{E-mail:} \texttt{gogolyev.1@osu.edu}\\

\flushleft{\bf Yi Shi} \\
School of Mathematical Sciences, Peking University, Beijing, 100871,  China\\
\textit{E-mail:} \texttt{shiyi@math.pku.edu.cn}\\

\end{document}